\documentclass[preprint,compress,3p,12pt]{elsarticle}
\usepackage{amsfonts}
\usepackage{amsthm}
\usepackage{mathrsfs}
\usepackage{amsmath}
\usepackage{amscd}
\usepackage{amssymb}
\usepackage{graphicx}
\usepackage{epstopdf}
\usepackage{subfigure}
\usepackage{array}
\usepackage{bm}
\usepackage{subfigure}
\usepackage{color}
\usepackage{booktabs}
\renewcommand{\vec}[1]{\mbox{\boldmath \small $#1$}}
 \usepackage{setspace}

\biboptions{numbers,sort&compress}

\allowdisplaybreaks \numberwithin{equation}{section}

\theoremstyle{plain}

\newtheorem{theorem}{Theorem}[section]
\newtheorem{lemma}[theorem]{Lemma}

\newtheorem{remark}{Remark}[section]

\theoremstyle{definition}
\newtheorem{example}{Example}[section]

\usepackage{geometry}
\usepackage[
                  bookmarksnumbered=true,
                  bookmarksopen=true,
                  colorlinks,
                  pdfborder=001,
                  linkcolor=green,
                  anchorcolor=green,
                  citecolor=green
                  ]{hyperref}
\journal{Elsevier}
\begin{document}

\begin{frontmatter}

\title{A general class of linear unconditionally energy stable schemes for the gradient flows, II}

\author{Zengqiang Tan}
\ead{tzengqiang@163.com}
\address{Center for Applied Physics and Technology, HEDPS and LMAM,
				School of Mathematical Sciences, Peking University, Beijing 100871, P.R. China}
\author{Huazhong Tang\corref{mycorrespondingauthor}}
\cortext[cor1]{Corresponding author. Fax:~+86-10-62751801.}\ead{hztang@pku.edu.cn}
			\address{Nanchang Hangkong University, Jiangxi Province, Nanchang 330000, P.R. China;  Center for Applied Physics and Technology, HEDPS and LMAM,
				School of Mathematical Sciences, Peking University, Beijing 100871, P.R. China}

\begin{abstract}
This paper continues to study linear and unconditionally modified-energy stable (abbreviated as SAV-GL) schemes for the gradient flows.
The schemes are built on the   SAV  technique and the general linear time discretizations (GLTD) as well as the extrapolation for the nonlinear term. Different from \cite{TanZQ22},  the GLTDs with three parameters  discussed here  are not necessarily algebraically stable. Some algebraic identities are derived by  using the method of undetermined coefficients
and further used to establish the  modified-energy  inequalities for the unconditional  modified-energy stability of the semi-discrete-in-time SAV-GL schemes.
It is worth emphasizing that those algebraic identities or energy inequalities are not necessarily unique  for some  choices of  three parameters in the GLTDs.
Numerical experiments on the Allen-Cahn, the Cahn-Hilliard and the phase field crystal models with the periodic boundary conditions are conducted to validate the unconditional modified-energy stability of the SAV-GL schemes, where  the Fourier pseudo-spectral method is employed in space with the   zero-padding to
eliminate the aliasing error  and  the  time {stepsizes}  for  ensuring the  original-energy decay are   estimated by using the stability regions of our SAV-GL schemes for the test equation.
The resulting time stepsize constraints  for the SAV-GL  schemes are almost consistent with the numerical results on the above gradient flow models.
\end{abstract}

\begin{keyword}
\texttt{} Gradient flows\sep   SAV approach\sep  energy stability\sep Fourier pseudo-spectral method.
\end{keyword}

\end{frontmatter}

\section{Introduction}

Many  practical problems   could be modeled by the gradient flows,
e.g., the interface dynamics \cite{Anderson98,Yue04}, the liquid crystallization \cite{Larson90,Leslie79}, the thin films \cite{Karma98,Wang03}, the polymers \cite{Fraaije93,Fraaije03}, and the the tumor growth \cite{Oden10,Wise08}.
For a given free energy $\mathcal{E}(u)$,  the gradient flow model can be given by
\begin{align}  \label{1.1}
\frac{\partial u}{\partial t} = \mathcal{G} \frac{\delta \mathcal{E}}{\delta u},~~~~ (\vec{x},t) \in \Omega \times (0,T],
\end{align}
supplemented with suitable initial and boundary conditions,
where $\Omega \subset \mathbb{R}^d, d =1,2,3$, $u=u(\vec{x},t)\in\mathbb R$, the operator $\mathcal{G}$ is negative, $\delta \mathcal{E}/\delta u$  denotes the variational derivative of the free energy functional $\mathcal{E}(u)$ with respect to the variable $u$, known as the chemical potential. Obviously, \eqref{1.1} implies that the free energy  is monotonically
non-increasing, that is,
\begin{align}    \label{1.2}
\frac{d \mathcal{E}}{d t}  = \left(\frac{\delta \mathcal{E}}{\delta u},  \frac{\partial u}{\partial t} \right) =   \left( \frac{\delta \mathcal{E}}{\delta u},  \mathcal{G}\frac{\delta \mathcal{E}}{\delta u} \right) \le 0,
\end{align}
and  the triple $\left\{u, \mathcal{G},  \mathcal{E}\right\}$ determines the gradient flow  uniquely,
where $(\cdot,\cdot)$ is the $L^2$ inner product defined by $(\phi,\psi) = \int_{\Omega}\phi\psi d\vec{x}$ for any  $\phi, \psi \in L^2(\Omega)$.
It is worth noting that \eqref{1.2} holds only for the boundary conditions such as  periodic or homogeneous Neumann boundary conditions
which can make the boundary integrals resulted from the integration by parts vanish.

In the last few decades, {many high-order accurate and unconditionally energy stable schemes}
have been developed for various nonlinear gradient flow models.
Those include, but are not limited to, the convex splitting method \cite{Elliott93,Eyre98,ShenJ12}, the stabilization method \cite{Tang06,ShenJ10b,WangL18}, the Lagrange multiplier method \cite{Badia11,Guillen13}, the exponential time differencing method \cite{WangX16,Du19}, and more recently, the invariant energy quadratization method \cite{YangX16,YangX17a,YangX17b,YangX20}, the scalar auxiliary variable (SAV) method \cite{ShenJ18a,ShenJ18b,ShenJ19} and its extensions, such as the exponential SAV  \cite{LiuZ20,LiuZ21}, the generalized SAV (G-SAV) \cite{HuangF20,HuangF21} and the SAV with relaxation  \cite{JiangM22}, etc.
Among those, the SAV approach and its variants become a particular powerful tool to construct modified-energy stable numerical schemes  and has been successfully applied to many existing gradient flow models, see e.g. \cite{Akrivis19,Cheng18,Cheng19,Gong20,HouD19, YangZ19,YangJ21,JiangM22, ZhangY22}.
Its main idea is to reformulate the gradient flow model into an equivalent form with the help of some SAVs, and then to develop  efficient numerical schemes by approximating the reformulated system instead of the original gradient flow model. Based on those SAV approaches, it is convenient to construct second- or higher-order unconditionally modified-energy stable schemes, and the derived schemes are   easy to be implemented
and only need to solve several linear equations at each time step if the nonlinear term is explicitly
approximated by the extrapolation etc.

Recently, in \cite{TanZQ22}, the authors studied  a general class  of linear unconditionally  modified-energy stable schemes for the gradient flows.
Those schemes (abbreviated as SAV-GL) are built on  the (original) SAV approach and the general linear time discretizations (GLTD) as well as  the linearization based on
the extrapolation for the nonlinear term. The proof of their
unconditional   modified-energy stability uses the {algebraical} stability of the GLTDs.
This paper continues to study the SAV-GL schemes for the gradient flows,
and will mainly addresses  three  issues:
\texttt{1)} How  the modified-energy inequality of  the SAV-GL
is derived if  the GLTDs are not necessarily algebraically stable?
\texttt{2)} Whether the   modified-energy inequality is unique?
\texttt{3)} How  a suitable time stepsize is chosen  to ensure the  original-energy
decay because  the unconditional  modified-energy stability does not imply  the unconditional  original-energy stability generally?
The main contributions are as follows: Different from \cite{TanZQ22},  the GLTDs with three parameters  discussed here  are not necessarily algebraically stable. Some algebraic identities are first
derived by  using the method of undetermined coefficients
and are then used to establish the  modified-energy  inequalities for the unconditional  modified-energy stability of the semi-discrete-in-time SAV-GL schemes. Those algebraic identities or energy inequalities are not necessarily unique  for some  choices of  three parameters in the GLTDs.
In order to validate the  energy stabilities of the SAV-GL schemes,
numerical experiments on the Allen-Cahn, the Cahn-Hilliard and the phase field crystal models with the periodic boundary conditions are conducted,  the Fourier pseudo-spectral method is employed in space
 with the   zero-padding to
eliminate the aliasing error,
 and the restrictions on the  time stepsize  for  preserving the  original-energy stability are  estimated by studying the stability regions of our SAV-GL schemes for the test equation.

The rest of this paper is organized as follows. Section \ref{sec:2} presents our new linear   unconditionally modified-energy stable schemes (still abbreviated as SAV-GL) for the gradient flows,  built on the  GLTDs with three parameters and the SAV  approach. Here the GLTDs are not necessarily algebraically stable.
Some algebraic identities are derived for the modified-energy inequality of  the SAV-GL, and they may not be necessarily unique  for some  choices of  three parameters in the GLTDs.
Section \ref{sec:6} conducts some numerical experiments
to validate the theoretical analysis of the SAV-GL schemes
in comparison to another SAV-GL schemes built on the generalized SAV \cite{HuangF20,HuangF21},
where  the Allen-Cahn, Cahn-Hilliard and phase field crystal models
with the periodic boundary conditions are considered,
the Fourier pseudo-spectral method is employed in space with the  de-aliasing by zero-padding, and the    time stepsizes  for  ensuring the  original-energy decay are also estimated by  the stability regions of our SAV-GL schemes for the test equation.
Some concluding remarks are given in Section \ref{sec:7}.

\section{SAV-GL schemes for the gradient flows}  \label{sec:2}

This section studies the general linear time discretizations (GLTDs)  with three parameters,  which are not necessarily algebraically stable, and develops the semi-discrete-in-time linear  SAV schemes
(still abbreviated as SAV-GL) for the gradient flow model \eqref{1.1}
with the help of  the original SAV  approach \cite{ShenJ18a,ShenJ18b,ShenJ19}.
Their unconditional  modified-energy stability will be derived
with some algebraic identities, established by  using the method of undetermined coefficients.

Assume that  the free energy $\mathcal{E}(u)$   contains some quadratic terms
such as
\begin{align}    \label{3.1.1}
\mathcal{E}(u) = \frac{1}{2}(\mathcal{L}u, u) + \mathcal{E}_1(u),
\end{align}
where $\mathcal{L}$ is a linear, positive and self-adjoint operator,     and $\mathcal{E}_1(u)$
denotes other nonlinear parts. Following the  SAV approach \cite{ShenJ18a,ShenJ18b,ShenJ19},
introduce the SAV $z(t) := \sqrt{\mathcal{E}_1(u)+C_0}$ with $C_0$ being a positive constant
so that $z$ is real-valued,  and then to rewrite the gradient flow model \eqref{1.1}  as
\begin{align}\label{3.1.2}
\begin{aligned}
\frac{\partial u}{\partial t} &= \mathcal{G} \mu,~~ \mu = \mathcal{L} u + z W(u),    \\
\frac{d z}{dt} &= \frac{1}{2} \left( W(u), \frac{\partial u}{\partial t}\right),~~~
\ W(u):= \frac1{z(t)}
\frac{\delta \mathcal{E}_1}{\delta u},
\end{aligned}
\end{align}
supplemented with suitable initial and boundary conditions.
Based on  \eqref{3.1.2}, one can construct the SAV schemes  for the  gradient flow model \eqref{1.1}.
It is easy to check that the reformulated system \eqref{3.1.2} satisfies the  energy dissipation law
\begin{align*}
\frac{d\mathcal{F}}{dt}(u) = \left(\mathcal{L} u, \frac{\partial u}{\partial t}\right) + 2z \frac{dz}{dt} = \left(\mathcal{L}u + z W(u), \frac{\partial u}{\partial t}\right) = \left( \mathcal{G}\mu, \mu \right)  \le 0,
\end{align*}
where the reformulated free energy $\mathcal{F}(u)=\frac{1}{2}(\mathcal{L}u, u) + z^2 - C_0$  is the same as the original  $\mathcal{E}(u)$.

Let $\tau$ be a given time stepsize, $t_n = n\tau$ for $n\ge 0$ and  $\chi^{n}$ denote an approximation to the generic variable $\chi$ at $t_n$.  Approximate  the variables $\chi$  and $\frac{\partial \chi}{\partial t}$ at $t_{n+\kappa} \!=\! t_n + \kappa\tau$ as follows
\begin{align}
\label{3.1.3}
& \frac{\partial \chi}{\partial t}\Big|^{n+\kappa} \!\approx\! \frac{1}{\tau({1\!-\!\alpha_0})}\!{ \left[    \chi^{n+1} \!-\!  (1\!+\!\alpha_0)   \chi^n \!+\!  {\alpha_0}   \chi^{n-1} \right] }\!,   \\
\label{3.1.4}
& \chi^{n+\kappa} \!=\! \frac{1}{1\!-\!\alpha_0}
\left(\beta_2\chi^{n+1} \!+ \!  {\beta_1}  \chi^n  \!+\!  {\beta_0}  \chi^{n-1} \right),   \\
& \bar{\chi}^{n+\kappa}  = \left(1\!+\!\kappa\right) \chi^n - \kappa  \chi^{n-1},
\label{3.1.5}
\end{align}
where $\alpha_0\neq 1, \beta_2\neq 0$ and $\beta_0$ {are} three free parameters, $\kappa =  \frac{\beta_2-\beta_0}{1-\alpha_0}$, $\beta_1 = 1-\alpha_0-\beta_0-\beta_2$, and $\chi^{n+\kappa}$ (resp.  $\bar{\chi}^{n+\kappa}$) denotes an implicit (resp. explicit) approximation to $\chi(t_{n+\kappa})$,   so that \eqref{3.1.3}-\eqref{3.1.5}  can provide  at least first-order accurate time discretizations.

\begin{lemma}\label{lemma2.1}
The fully implicit time discretizations based on \eqref{3.1.3}-\eqref{3.1.4}
are $A-$stable
(but are not necessarily algebraically stable)
if
\begin{align} \label{8.10B}
-1\le \alpha_0<1,~~~\beta_2>0,~~~|\beta_0| \le \beta_2,~~~1-\alpha_0-2\beta_0-2\beta_2 \le 0.
\end{align}
Specially, (i) when $\alpha_0 = \beta_0 = 0$, the time discretizations based on \eqref{3.1.3}-\eqref{3.1.4} are one-step and $A-$stable for any $\beta_2\ge \frac{1}{2}$;
 (ii) when $\beta_2 = \frac{1+\alpha_0}{2} + \beta_0$  and $|\alpha_0|+|\beta_0|\neq 0$
(i.e. $\alpha_0$ and $\beta_0$ are not zero simultaneously), the time discretizations based on \eqref{3.1.3}-\eqref{3.1.4} are two-step and second-order accurate, which are $A-$stable for any $-1\le \alpha_0<1$ and $2\beta_0+\alpha_0\ge 0$; and
(iii)  when  $\beta_2 \neq \frac{1+\alpha_0}{2} + \beta_0$ and $|\alpha_0|+ |\beta_0|\neq 0$,
   the time discretizations based on \eqref{3.1.3}-\eqref{3.1.4} are two-step and first-order accurate, which are $A-$stable under \eqref{8.10B}.
 \end{lemma}
The proof of Lemma \ref{lemma2.1}  is given in \ref{Appx1}.

Assume that $(u^{n-1}, z^n)$ and $(u^n,z^n)$ are given. Applying  \eqref{3.1.3}-\eqref{3.1.5} to the reformulated system \eqref{3.1.2} yields the following
 semi-discrete-in-time SAV-GL scheme
\begin{align}
\begin{aligned}   \label{3.1.6}
& \frac{1}{1\!-\!\alpha_0} u^{n+1} \!-\! \frac{1\!+\!\alpha_0}{1\!-\!\alpha_0}  u^{n} \!+\! \frac{\alpha_0}{1\!-\!\alpha_0} u^{n-1} \!=\! \tau \mathcal{G} \mu^{n+\kappa},~~
\mu^{n+\kappa} \!=\! \mathcal{L} u^{n+\kappa}  \!+ \!z^{n+\kappa} W(\bar{u}^{n+\kappa}), \\[1 \jot]
& \frac{1}{1\!-\!\alpha_0} z^{n+1} \!-\! \frac{1\!+\!\alpha_0}{1\!-\!\alpha_0}  z^{n} \!+ \! \frac{\alpha_0}{1\!-\!\alpha_0}  z^{n-1} \!=\! \frac{1}{2} \!\left(\! W(\bar{u}^{n+\kappa}),  \frac{1}{1\!-\!\alpha_0} u^{n+1} \!- \! \frac{1\!+\!\alpha_0}{1\!-\!\alpha_0}  u^{n} \!+ \! \frac{\alpha_0}{1\!-\!\alpha_0} u^{n-1} \!\right)\!,
\end{aligned}
\end{align}
where $u^{n+\kappa}, z^{n+\kappa}$, and $\bar{u}^{n+\kappa}$ are given by \eqref{3.1.4} and \eqref{3.1.5}, respectively. In order to derive its
unconditional modified-energy stability,
several algebraic identities  are  established as follows.

\begin{lemma}   \label{lem3.1}
(i) When $\alpha_0 = \beta_0 = 0$, the   identity
\begin{align}  \label{3.1.7}
\left(\chi^{n+1} - \chi^n \right) \left(\beta_2\chi^{n+1} + (1\!-\!\beta_2) \chi^n \right) =  \frac{1}{2} \left[ \left(\chi^{n+1}\right)^2 -\left(\chi^{n}\right)^2\right] + \left( \beta_2\!-\!\frac{1}{2}\right)\left( \chi^{n+1} - \chi^n \right)^2,
\end{align}
holds for any $\beta_2 \ge \frac{1}{2}$.

(ii) When $\beta_2 = \frac{1+\alpha_0}{2}+\beta_0$    and $|\alpha_0|+|\beta_0|\neq 0$,  then  the   identity
\begin{align}   \label{3.1.8}
& \left(\frac{1}{1\!-\!\alpha_0}\chi^{n+1} - \frac{1\!+\!\alpha_0}{1\!-\!\alpha_0} \chi^n + \frac{\alpha_0}{1\!-\!\alpha_0}\chi^{n-1}\right) \left(\frac{\beta_2}{1\!-\!\alpha_0}\chi^{n+1} +\frac{\beta_1}{1\!-\!\alpha_0} \chi^n + \frac{\beta_0}{1\!-\!\alpha_0}\chi^{n-1}\right)    \nonumber \\
= \;& \frac{2\!+\!\alpha_0\!-\!\alpha_0^2 \!+  \!2\beta_0(1\!-\!\alpha_0) }{4(1\!-\!\alpha_0)^2} \!\left[ \left(\chi^{n+1}\right)^2 \!-\!\left(\chi^{n}\right)^2\right] \!\! +\! \frac{\alpha_0\!+\!\alpha_0^2 \!+  \!2\beta_0(1\!-\!\alpha_0) }{4(1\!-\!\alpha_0)^2} \left[ \left(\chi^{n}\right)^2 \!-\!\left(\chi^{n-1}\right)^2\right] \nonumber \\
& +\! \frac{(\alpha_0\!-\!1)(2\beta_0\!+\!\alpha_0\!-\!1) \!-\!(\alpha_0\!+\!1)} {2(1\!-\!\alpha_0)^2} \!\left[ \chi^{n+1}\chi^{n} \!-\!\chi^{n}\chi^{n-1} \right]\!\!   \nonumber \\
& + \! \frac{(1\!+\!\alpha_0)(2\beta_0\!+\!\alpha_0)}{4(1\!-\!\alpha_0)^2} \!\left( \chi^{n+1} \!-\! 2 \chi^n \!+ \!\chi^{n-1}\right)^2,
\end{align}
holds for any $-1\le \alpha_0<1$ and $2\beta_0+\alpha_0\ge 0$.

(iii) When $\beta_2 \neq  \frac{1+\alpha_0}{2}+\beta_0$    and $|\alpha_0|+|\beta_0|\neq 0$, then one has
\begin{align}   \label{3.1.8*}
& \left(\frac{1}{1\!-\!\alpha_0}\chi^{n+1} - \frac{1\!+\!\alpha_0}{1\!-\!\alpha_0} \chi^n + \frac{\alpha_0}{1\!-\!\alpha_0}\chi^{n-1}\right) \left(\frac{\beta_2}{1\!-\!\alpha_0}\chi^{n+1} +\frac{\beta_1}{1\!-\!\alpha_0} \chi^n + \frac{\beta_0}{1\!-\!\alpha_0}\chi^{n-1}\right)    \nonumber \\
= \;& \!\!\left[\frac{1\!-\!\alpha_0^2 \!+  \!2\beta_2\!-\!2\alpha_0\beta_0}{4(1\!-\!\alpha_0)^2} \!-\! \tilde{c}c\right]\!\left[ \left(\chi^{n+1}\right)^2 \!-\!\left(\chi^{n}\right)^2\right] \!\! +\!\! \left[\frac{2\beta_2\!+\!\alpha_0^2\!-\!1}{4(1\!-\!\alpha_0)^2} \!- \!\tilde{c}c\right] \!\left[ \left(\chi^{n}\right)^2 \!-\!\left(\chi^{n-1}\right)^2\right] \nonumber \\
& +\! \!\left[\frac{1}{2}\!+\!\frac{\alpha_0\beta_0\!-\!\beta_2}{(1\!-\!\alpha_0)^2} \!+\!2\tilde{c}c\!\right] \!\left[ \chi^{n+1}\chi^{n} \!-\!\chi^{n}\chi^{n-1} \right]\!\!   + \!\!\left[\!\left(\!c\!-\!\frac{\tilde{c}}{2}\!\right) \!\chi^{n+1} \!+\!\tilde{c} \chi^n \!-\! \left(\!c\!+\!\frac{\tilde{c}}{2}\!\right) \!\chi^{n-1}\right]^2\!,
\end{align}
under the conditions   \eqref{8.10B} and $2\beta_2-2\beta_0-\alpha_0-1>0$,
where
\[ c = \sqrt{\frac{2\beta_2\!-\!2\beta_0\!-\!\alpha_0\!-\!1}{8(1-\alpha_0)}},~~~\tilde{c} = - \frac{\sqrt{2(1\!+\!\alpha_0)(2\beta_0\!+\!2\beta_2\!+\!\alpha_0\!-\!1)}} {2(1\!-\!\alpha_0)}.    \]
\end{lemma}
The proof of this lemma   is given in \ref{Appx2} by using the method of undetermined coefficients.
%
Using those identities in Lemma \ref{lem3.1} can give the following results
on  the semi-discrete-in-time SAV-GL scheme \eqref{3.1.6}.

\begin{theorem}  \label{Thm3.1}
(i) When $\alpha_0 = \beta_0 = 0$, the semi-discrete scheme \eqref{3.1.6} is unconditionally modified-energy stable for any $\beta_2 \ge \frac{1}{2}$  in the sense that
\begin{align}   \label{3.1.9}
\frac{1}{2}\left( \mathcal{L} u^{n+1}, u^{n+1} \right) + \left( z^{n+1} \right)^2 \le \frac{1}{2}\left( \mathcal{L} u^{n}, u^{n} \right) + \left( z^{n} \right)^2.
\end{align}

(ii) When $\beta_2 = \frac{1+\alpha_0}{2} + \beta_0$ and $|\alpha_0|+|\beta_0|\neq 0$,  the semi-discrete scheme \eqref{3.1.6} is unconditionally modified-energy stable for any $-1\le \alpha_0<1$ and $2\beta_0+\alpha_0\ge 0$ in the sense that
\begin{align}   \label{3.1.10}
E\left(u^{n+1},u^{n},z^{n+1},z^{n} \right) \le E\left(u^{n},u^{n-1},z^{n},z^{n-1} \right),
\end{align}
where
\begin{align*}
E\left(u^{n+1},u^{n},z^{n+1},z^{n} \right) : = \;& \frac{(\alpha_0\!-\!1)(2\beta_0\!+\!\alpha_0\!-\!1)\!-\!(\alpha_0\!+\!1)} {(1\!-\!\alpha_0)^2}\! \left[ \frac{1}{2}\left( \mathcal{L}u^{n+1}, u^{n}\right) + z^{n+1}z^{n}\right]  \\
&\hspace{-2cm} +  \frac{2\!+\!\alpha_0\!-\!\alpha_0^2\!+\!2\beta_0(1\!-\!\alpha_0)} {2(1\!-\!\alpha_0)^2}\! \left[ \frac{1}{2}\left( \mathcal{L}u^{n+1}, u^{n+1}\right) + \left( z^{n+1}\right)^2 \right]    \\
&\hspace{-2cm} + \frac{\alpha_0\!+\!\alpha_0^2\!+\!2\beta_0(1\!-\!\alpha_0)} {2(1\!-\!\alpha_0)^2}\! \left[ \frac{1}{2}\left( \mathcal{L}u^{n}, u^{n}\right) + \left( z^{n}\right)^2 \right].
\end{align*}

(iii) When $\beta_2 \neq \frac{1+\alpha_0}{2}+\beta_0$    and $|\alpha_0|+|\beta_0|\neq 0$,   the semi-discrete scheme \eqref{3.1.6} is unconditionally modified-energy stable under the conditions \eqref{8.10B} and $ 2\beta_2-2\beta_0-\alpha_0-1>0$ in the sense that
\begin{align}   \label{3.1.10*}
\bar{E}\left(u^{n+1},u^{n},z^{n+1},z^{n} \right) \le \bar{E}\left(u^{n},u^{n-1},z^{n},z^{n-1} \right),
\end{align}
where
\begin{align*}
&\bar{E}\left(u^{n+1},u^{n},z^{n+1},z^{n} \right) : = \! \left[\frac{1\!-\!\alpha_0^2 \!+  \!2\beta_2\!-\!2\alpha_0\beta_0}{4(1\!-\!\alpha_0)^2} \!-\! \tilde{c}c\right]\! \left[ \frac{1}{2}\left( \mathcal{L}u^{n+1}, u^{n+1}\right) + \left(z^{n+1}\right)^2\right]  \\
+ & \!\! \left[\frac{2\beta_2\!+\!\alpha_0^2\!-\!1}{4(1\!-\!\alpha_0)^2} \!- \!\tilde{c}c\right]\!\! \left[ \frac{1}{2}\!\left( \mathcal{L}u^{n}, u^{n}\right)\! +\! \left( z^{n}\right)^2 \right] \!\! + \! \!\left[\frac{1}{2}\!+\!\frac{\alpha_0\beta_0\!-\!\beta_2}{(1\!-\!\alpha_0)^2} \!+\!2\tilde{c}c\!\right] \!\! \left[ \frac{1}{2}\!\left( \mathcal{L}u^{n+1}, u^{n}\right) \!+ \!z^{n+1}z^{n} \right]\!.
\end{align*}
\end{theorem}
\begin{proof}
The proofs of  three inequalities \eqref{3.1.9}-\eqref{3.1.10*} are similar so that only   the inequality \eqref{3.1.10} is proved here to avoid repetition.
 It is worth emphasizing that some different  identities
from \eqref{3.1.7} and \eqref{3.1.8*} are also presented in \ref{Appx2},
so that different unconditionally modified-energy inequalities from \eqref{3.1.9}
and \eqref{3.1.10*} can  be established for some choices of  three parameters $\alpha_0$,
$\beta_0$, $\beta_2$ in the GLTDs, e.g.
  $\{\alpha_0 =\beta_0 =0$, $\beta_2> \frac{1}{2}\}$ and
  $\{\beta_2 \neq \frac{1+\alpha_0}{2}+\beta_0$,   $|\alpha_0|+|\beta_0|\neq 0\}$.

Taking the $L^2$ inner product of the first and second equations in \eqref{3.1.6} with $\mu^{n+\kappa}$ and $\frac{1}{1-\alpha_0} u^{n+1} - \frac{1+\alpha_0}{1-\alpha_0}  u^{n} + \frac{\alpha_0}{1-\alpha_0} u^{n-1} $, respectively, yields
\begin{align}   \label{3.1.11}
\left( \frac{1}{1\!-\!\alpha_0} u^{n+1} - \frac{1\!+\!\alpha_0}{1\!-\!\alpha_0}  u^{n} + \frac{\alpha_0}{1\!-\!\alpha_0} u^{n-1}, \mu^{n+\kappa}\right) = \tau \left(\mathcal{G}\mu^{n+\kappa}, \mu^{n+\kappa}\right),
\end{align}
and
\begin{align}   \label{3.1.12}
& \left( \frac{1}{1\!-\!\alpha_0} u^{n+1} - \frac{1\!+\!\alpha_0}{1\!-\!\alpha_0}  u^{n} + \frac{\alpha_0}{1\!-\!\alpha_0} u^{n-1}, \mu^{n+\kappa}\right) = \left(\mathcal{L}u^{n+\kappa}, \frac{1}{1\!-\!\alpha_0} u^{n+1} - \frac{1\!+\!\alpha_0}{1\!-\!\alpha_0}  u^{n} + \frac{\alpha_0}{1\!-\!\alpha_0} u^{n-1} \right) \nonumber \\
& \hspace{3cm} +  z^{n+\kappa}  \left(  W(\bar{u}^{n+\kappa}), \frac{1}{1\!-\!\alpha_0} u^{n+1} - \frac{1\!+\!\alpha_0}{1\!-\!\alpha_0}  u^{n} + \frac{\alpha_0}{1\!-\!\alpha_0} u^{n-1}\right).
\end{align}
According to the identity \eqref{3.1.8}, one can deduce
\begin{align}    \label{3.1.13}
& \left(\mathcal{L} u^{n+\kappa}, \frac{1}{1\!-\!\alpha_0} u^{n+1} - \frac{1\!+\!\alpha_0}{1\!-\!\alpha_0}  u^{n} + \frac{\alpha_0}{1\!-\!\alpha_0} u^{n-1} \right)  \nonumber \\
=\;& \frac{2\!+\!\alpha_0\!-\!\alpha_0^2 \!+  \!2\beta_0(1\!-\!\alpha_0) }{4(1\!-\!\alpha_0)^2}\bigg[\! \left( \mathcal{L} u^{n+1}, u^{n+1} \right) - \left( \mathcal{L} u^{n}, u^{n} \right)\!\bigg]  +  \frac{\alpha_0\!+\!\alpha_0^2 \!+  \!2\beta_0(1\!-\!\alpha_0) }{4(1\!-\!\alpha_0)^2} \bigg[\! \left( \mathcal{L} u^{n}, u^{n} \right) \nonumber \\
& - \left( \mathcal{L} u^{n-1}, u^{n-1} \right) \!\bigg]
 + \frac{(\alpha_0\!-\!1)(2\beta_0\!+\!\alpha_0\!-\!1) \!-\!(\alpha_0\!+\!1)} {2(1\!-\!\alpha_0)^2} \bigg[\! \left( \mathcal{L} u^{n+1}, u^{n} \right) - \left( \mathcal{L} u^{n}, u^{n-1} \right) \!\bigg]   \nonumber \\
& +\frac{(1\!+\!\alpha_0)(2\beta_0\!+\!\alpha_0)}{4(1\!-\!\alpha_0)^2} \bigg(\mathcal{L}\left[ u^{n+1}-2u^n + u^{n-1}\right], u^{n+1}-2u^n + u^{n-1} \bigg),
\end{align}
and
\begin{align}    \label{3.1.14}
& z^{n+\kappa}\left( \frac{1}{1\!-\!\alpha_0} z^{n+1} - \frac{1\!+\!\alpha_0}{1\!-\!\alpha_0}  z^{n} + \frac{\alpha_0}{1\!-\!\alpha_0} z^{n-1} \right)    \nonumber \\
=\; & \frac{2\!+\!\alpha_0\!-\!\alpha_0^2 \!+  \!2\beta_0(1\!-\!\alpha_0) }{4(1\!-\!\alpha_0)^2}\bigg[ \!\left( z^{n+1} \right)^2 -\left( z^{n} \right)^2 \!\bigg]     +  \frac{\alpha_0\!+\!\alpha_0^2 \!+  \!2\beta_0(1\!-\!\alpha_0) }{4(1\!-\!\alpha_0)^2} \bigg[\! \left( z^{n} \right)^2 - \left( z^{n-1} \right)^2\!\bigg]    \nonumber \\
& + \frac{(\alpha_0\!-\!1)(2\beta_0\!+\!\alpha_0\!-\!1) \!-\!(\alpha_0\!+\!1)} {2(1\!-\!\alpha_0)^2} \bigg[ z^{n+1} z^{n} - z^{n} z^{n-1}  \bigg]
\nonumber \\
& + \frac{(1\!+\!\alpha_0)(2\beta_0\!+\!\alpha_0)}{4(1\!-\!\alpha_0)^2} \left( z^{n+1} - 2z^n +z^{n-1} \right)^2.
\end{align}
Multiplying the third equation in \eqref{3.1.6} with $z^{n+\kappa}$ and using \eqref{3.1.14} give
\begin{align}   \label{3.1.15}
& \frac{1}{2}z^{n+\kappa}\left(  W(\bar{u}^{n+\kappa}), \frac{1}{1\!-\!\alpha_0} u^{n+1} - \frac{1\!+\!\alpha_0}{1\!-\!\alpha_0}  u^{n} + \frac{\alpha_0}{1\!-\!\alpha_0} u^{n-1} \right)   \nonumber \\
=\; & \frac{2\!+\!\alpha_0\!-\!\alpha_0^2 \!+  \!2\beta_0(1\!-\!\alpha_0) }{4(1\!-\!\alpha_0)^2}\bigg[ \!\left( z^{n+1} \right)^2 -\left( z^{n} \right)^2 \!\bigg]   +  \frac{\alpha_0\!+\!\alpha_0^2 \!+  \!2\beta_0(1\!-\!\alpha_0) }{4(1\!-\!\alpha_0)^2} \bigg[\! \left( z^{n} \right)^2 - \left( z^{n-1} \right)^2\!\bigg]    \nonumber \\
& + \frac{(\alpha_0\!-\!1)(2\beta_0\!+\!\alpha_0\!-\!1) \!-\!(\alpha_0\!+\!1)} {2(1\!-\!\alpha_0)^2} \bigg[ z^{n+1} z^{n} - z^{n} z^{n-1}  \bigg]
\nonumber \\
& + \frac{(1\!+\!\alpha_0)(2\beta_0\!+\!\alpha_0)}{4(1\!-\!\alpha_0)^2} \left( z^{n+1} - 2z^n +z^{n-1} \right)^2.
\end{align}
Substituting \eqref{3.1.15} and \eqref{3.1.13} into \eqref{3.1.12} and using \eqref{3.1.11} lead to
\begin{align}   \label{3.1.16}
& E\left(u^{n+1},u^{n},z^{n+1},z^{n} \right) - E\left(u^{n},u^{n-1},z^{n},z^{n-1} \right)   \nonumber \\
=\;&  \tau \bigg(\mathcal{G}\mu^{n+\kappa}, \mu^{n+\kappa}\bigg) - \frac{(1\!+\!\alpha_0)(2\beta_0\!+\!\alpha_0)}{2(1\!-\!\alpha_0)^2} \left( z^{n+1} - 2z^n +z^{n-1} \right)^2     \nonumber \\
& -  \frac{(1\!+\!\alpha_0)(2\beta_0\!+\!\alpha_0)}{4(1\!-\!\alpha_0)^2} \bigg(\mathcal{L}\left[ u^{n+1}-2u^n + u^{n-1}\right], u^{n+1}-2u^n + u^{n-1} \bigg).
\end{align}
Since the operator $\mathcal{L}$ is positive, $\mathcal{G}$ is negative, and
the parameters $\alpha_0$ and $\beta_0$ satisfy $-1\le \alpha_0 < 1$ and $2\beta_0+\alpha_0\ge0$,
 one can conclude from \eqref{3.1.16} that the inequality \eqref{3.1.10} holds. Hence, the proof is completed.
\end{proof}

\begin{remark}
If taking $\alpha_0 = \frac{1}{3}, \beta_0 = 0$ and $\beta_2 = \frac{2}{3}$, then \eqref{3.1.6} becomes the SAV-BDF2 scheme in \cite{ShenJ19}, and  \eqref{3.1.8} reduces to the identity used in \cite{ShenJ19} to derive the modified-energy stability of the SAV-BDF2 scheme.
If taking $\alpha_0 \!=\! \frac{2\theta-1}{2\theta+1}, \beta_0 \!=\! - \frac{(2\theta-1)(\theta-1)}{2\theta+1}$ and $\beta_2 \!=\! - \frac{2\theta^2-5\theta+1}{2\theta+1}$ with $\frac{1}{2}\le \theta \le \frac{3}{2}$, then \eqref{3.1.6} reduces to the  scheme in \cite{YangZ19}  for the Cahn-Hilliard equation, where the modified-energy stability is derived by using  a special case of \eqref{3.1.8}.
\end{remark}

\begin{remark}
The semi-discrete scheme \eqref{3.1.6} can be written into the form of the SAV-GL scheme in \cite{TanZQ22}. In fact, one can first compute the stage value $\left(U_{n,1}, Z_{n,1} \right)$ from
\begin{align*}
\begin{cases}
U_{n,1} \!=\! \tau\beta_2 \dot{U}_{n,1} \!+\! \frac{\beta_1+\beta_2(1+\alpha_0)}{1-\alpha_0} u^n \!+ \! \frac{\beta_0-\alpha_0\beta_2}{1-\alpha_0} u^{n-1},~~ \dot{U}_{n,1} \!= \! \mathcal{G}\mu_{n,1},~\mu_{n,1} \!=\! \mathcal{L} U_{n,1}  \!+ \!Z_{n,1} W(\bar{u}^{n+\kappa}) \\
Z_{n,1} = \tau \beta_2 \dot{Z}_{n,1} + \frac{\beta_1+\beta_2(1+\alpha_0)}{1-\alpha_0} z^n + \frac{\beta_0-\alpha_0\beta_2}{1-\alpha_0} z^{n-1},~~~ \dot{Z}_{n,1} = \frac{1}{2}\left( W(\bar{u}^{n+\kappa}),  \dot{U}_{n,1}\right),
\end{cases}
\end{align*}
and then derive the numerical solution $\left(u^{n+1}, z^{n+1} \right)$ by
\begin{align*}
u^{n+1} = \tau \dot{U}_{n,1} + (1+\alpha_0) u^n -\alpha_0 u^{n-1},~~~
z^{n+1} = \tau \dot{Z}_{n,1} + (1+\alpha_0) z^n -\alpha_0 z^{n-1}.
\end{align*}
Thus, when the previous GLTDs with three parameters are algebraically stable, the modified-energy stability of the SAV-GL scheme  \eqref{3.1.6} can also be {proved} by using the theoretical framework in \cite{TanZQ22}.
It should be emphasized that the established modified-energy
inequalities in Theorem \ref{Thm3.1} are more general and applicable for some GLTDs without the algebraical stability. For example,  in the
case (ii), i.e. $\beta_2 = \frac{1+\alpha_0}{2} + \beta_0$ and $|\alpha_0|+|\beta_0| \neq 0$,
the GLTDs with $1\le \alpha_0 <1$ and $2\beta_0+\alpha_0 = 0$ are not algebraically stable (see \ref{Appx1}),  but
the modified-energy stability of the corresponding SAV-GL schemes can be gotten by using the identity \eqref{3.1.8}. Moreover, we can also find that those energy inequalities
may  not be unique for some $\alpha_0$,
$\beta_0$, and $\beta_2$.
\end{remark}

\begin{remark}  \label{remk3.6}
 Theorem \ref{Thm3.1} tells us that \eqref{3.1.6} is unconditionally modified-energy stable with choosing appropriate parameters. However,
the original-energy $\mathcal{E}(u^{n})$ of
the gradient flow \eqref{1.1}  may be monotonically decreasing
conditionally.
It will be confirmed by combining numerical experiments in Section \ref{sec:6}
for the Allen-Cahn, the Cahn-Hilliard and the phase field crystal models
with  studying the stability regions of the semi-implicit time discretizations based on \eqref{3.1.3}-\eqref{3.1.5} studied in
 \ref{Appx3}.
\end{remark}


\begin{remark}\label{rem:2.4}
There exist some variants of the original SAV approach. For example,  instead of  the square root function,  the exponential function \cite{LiuZ20,LiuZ21}, the monotone polynomials, and the tanh function \cite{Cheng20} could be used  to extend  the original SAV approach. One can combine those extended SAV approaches with the time-discretizations \eqref{3.1.3}-\eqref{3.1.5} and use Lemma \ref{lem3.1} to obtain  corresponding modified-energy stabilities. Moreover, applying the relaxation technique to \eqref{3.1.6} can derive  the SAV schemes with relaxation (R-SAV)
\cite{JiangM22}, which may improve the accuracy and consistency of the introduced SAV.
Besides, there still exists an interesting extension  of the original SAV  approach, namely the generalized SAV (G-SAV) approach  \cite{HuangF20,HuangF21}. The proof of its modified-energy
stability may do not require Lemma \ref{lem3.1}.
%
If defining a shifted free energy by
$ \widetilde{\mathcal{E}}(u) = \mathcal{E}(u) + \tilde{C}_0$ and introducing a new SAV $R(t):=\widetilde{\mathcal{E}}(u)$,
where $\tilde{C}_0$ is a chosen non-negative constant such that $\widetilde{\mathcal{E}}(u)$ is always positive,
then  the  gradient flow model \eqref{1.1} can be reformulated
as follows
\begin{align}
\begin{aligned}  \label{3.2.1}
& \frac{\partial u}{\partial t} = \mathcal{G} \mu,~~~ \mu = \mathcal{L} u +  V(u),   ~~~V(u):=
\frac{\delta \mathcal{E}_1}{\delta u},   \\[1 \jot]
&  \frac{d R}{dt} = \eta \left(\mu, \mathcal{G} \mu \right),
\end{aligned}
\end{align}
where $\eta(t) = \frac{R(t)}{\widetilde{\mathcal{E}}(u)} \equiv 1$ at the continuous level.
After applying the time discretizations \eqref{3.1.3}-\eqref{3.1.5} to  \eqref{3.2.1}, one has the following semi-discrete-in-time G-SAV-GL scheme
\begin{align}
\label{3.2.3}
& \frac{1}{1\!-\!\alpha_0}u^{n+1}-\frac{1\!+\!\alpha_0}{1 \!-\!\alpha_0}u^n + \frac{\alpha_0}{1\!-\!\alpha_0} u^{n-1} = \tau\mathcal{G} \hat{\mu}^{n+\kappa},  \\
\label{3.2.4}
& \qquad \hat{\mu}^{n+\kappa} = \mathcal{L} \left[ \frac{\beta_2}{1\!-\!\alpha_0} u^{n+1}+ \frac{\beta_1}{1\!-\!\alpha_0}u^{n} + \frac{\beta_0}{1\!-\!\alpha_0}u^{n-1}\right] +  V(\bar{u}^{n+\kappa}),  \\
 \label{3.2.6}
&\frac{R^{n+1}-R^n}{\tau} = \eta^{n+1} \left(\mu^{n+1},  \mathcal{G} \mu^{n+1} \right),
\end{align}
where  $ \eta^{n+1} = \frac{R^{n+1}}{\widetilde{\mathcal{E}}(u^{n+1})}$, $\mu^{n+1} = \mathcal{L} u^{n+1} +  V(u^{n+1})$.
They imply
\begin{align*}  
\left( \!\frac{1}{1\!-\!\alpha_0} \!-\! \frac{\tau\beta_2}{1\!-\!\alpha_0}\mathcal{G}\mathcal{L}\!\right) \! u^{n+1} \!=\! \frac{1\!+\!\alpha_0}{1 \!-\!\alpha_0}u^n \!- \! \frac{\alpha_0}{1\!-\!\alpha_0} u^{n-1} \!+\! \tau \mathcal{G}\mathcal{L}  \!\left[ \!\frac{\beta_1}{1\!-\!\alpha_0}u^{n} \!+ \! \frac{\beta_0}{1\!-\!\alpha_0}u^{n-1}\!\right]\!\! + \!\tau \mathcal{G} V(\bar{u}^{n+\kappa}),
\end{align*}
which is a equation of $u^{n+1}\approx u(\cdot,t_{n+1})$. Moreover, 
for given $R^n> 0$,  $R^{n+1}$ and $\eta^{n+1}$ are positive  and   \eqref{3.2.3}-\eqref{3.2.6} is unconditionally modified-energy stable in the sense that
\begin{align*}  
R^{n+1} - R^n = \tau \eta^{n+1} \left(\mu^{n+1}, \mathcal{G} \mu^{n+1} \right) \le 0.
\end{align*}
Note that   the G-SAV scheme \eqref{3.2.3}-\eqref{3.2.6} is different from that in \cite{HuangF20,HuangF21}, the main difference between them is
that no special control factor, e.g. $\xi^{n+1} = 1-(1-\eta^{n+1})^3$, is introduced in \eqref{3.2.3}-\eqref{3.2.6}
so that the numerical solution $u^{n+1}$ is totally derived by  corresponding semi-implicit scheme.
When the time stepsize is large,
$\xi^{n+1}$ may become a bad approximation to one  so that numerical solutions  are not accurate. Hence, such difference allows us to choose a larger time stepsize  when applying \eqref{3.2.3}-\eqref{3.2.6} to the gradient flow  \eqref{1.1}. The scheme  \eqref{3.2.3}-\eqref{3.2.6}
will be compared to the SAV-GL scheme \eqref{3.1.6} in our numerical experiments, see Section \ref{sec:6}.
\end{remark}

\section{Numerical experiments}   \label{sec:6}

This section applies the  SAV-GL scheme \eqref{3.1.6}
in comparison to the G-SAV-GL scheme \eqref{3.2.3}-\eqref{3.2.6}
to the Allen-Cahn, the Cahn-Hilliard and the phase field crystal models with the periodic boundary conditions
in order to demonstrate their modified-energy stability and check
their original-energy stability.  For such purpose,
the Fourier pseudo-spectral spatial discretization \cite{TanZQ22} is still employed
for \eqref{3.1.6} and \eqref{3.2.3}-\eqref{3.2.6}.
The theoretical results in Section \ref{sec:2} could be straightforwardly extended to such fully discrete schemes.
The readers are referred to \cite{TanZQ22} about the fully discrete SAV-GL methods  and   \cite{Gottlieb77,ShenJ11,Gottlieb12a,ChenWB12,JuL18} for more detailed descriptions of the spectral methods.
Our fully discrete  SAV-GL schemes are implemented in MATLAB  and call  both  {\tt fft} and  {\tt ifft} functions directly for the discrete Fourier and  inverse  Fourier transforms so that their implementation is very simple and efficient. It should be noted that the FFT of   the nonlinear term may always produce the aliasing errors, see e.g. \cite{CHQZ1988,Tadmor87}.
The effect
of the aliasing error and the   de-aliasing by zero-padding  provided in \ref{sec:5} on our numerical results are investigated. 
For simplicity, the subsequent numerical results
will be given only for several special values of three parameters
{$(\alpha_0, \beta_0, \beta_2)$}  in \eqref{3.1.3}-\eqref{3.1.5},
see Table \ref{tab6.1},
 and  corresponding fully-discrete SAV-GL and G-SAV-GL schemes
will also be abbreviated as in Table \ref{tab6.1}.

\begin{table}[!htbp]
\begin{center}
\caption{Choices of {$(\alpha_0, \beta_0, \beta_2)$}
and abbreviations of corresponding fully discrete schemes.}
\setlength{\tabcolsep}{1mm}{\begin{tabular*}{\textwidth}{@ {\extracolsep{\fill}}cccccccc}
\hline
{} & SAV-GL schemes & G-SAV-GL schemes    \\
\hline
$(\alpha_0,\beta_0,\beta_2) = (0,0,1)$               &{\tt SAV-M(1)}      &{\tt G-SAV-M(1)}          \\
$(\alpha_0,\beta_0,\beta_2) = (-1/3,5/12,3/4)$       &{\tt SAV-M(2)}    &{\tt G-SAV-M(2)}             \\
$(\alpha_0,\beta_0,\beta_2) = (1/3,0,2/3)$           &{\tt SAV-M(3)}    &{\tt G-SAV-M(3)}        \\
$(\alpha_0,\beta_0,\beta_2) = (1/3,-1/6,1/2)$        &{\tt SAV-M(4)}     &{\tt G-SAV-M(4)}         \\
\hline
\end{tabular*}} \label{tab6.1}
\end{center}
\end{table}

\subsection{Allen-Cahn model}  \label{sec:6.1}

The Allen-Cahn model
\begin{align}   \label{6.1.1}
\frac{\partial u}{\partial t} = \epsilon^2 \Delta u - u^3 + u,
~~~\vec{x}\in \Omega,~t>0,
\end{align}
was introduced to describe the motion of anti-phase interfaces in crystalline solids \cite{Allen79} and can be derived from the $L^2$ gradient flow of the following free energy
\begin{align}   \label{6.1.3}
\mathcal{E}(u) = \int_{\Omega} \frac{\epsilon^2}{2} |\nabla u|^2 + \frac{1}{4}(u^2-1)^2 dx,
\end{align}
where $0<\epsilon< 1$ denotes the diffuse interface thickness.

In order to apply {\tt SAV-M(1)}$\sim${\tt SAV-M(4)} and {\tt G-SAV-M(1)}$\sim${\tt G-SAV-M(4)} for \eqref{6.1.1}, one takes
\[  \mathcal{L} = - \epsilon^2 \Delta, ~~~\mathcal{G} = -1,~~~ \mathcal{E}_1(u) = \int_{\Omega} \frac{1}{4}\left(u^2 - 1 \right)^2 dx. \]

\begin{example} \label{Exp6.1.0}
This example is used to check the  effectiveness  of the de-aliasing by zero-padding for the Allen-Cahn equation \eqref{6.1.1} with $\epsilon = 0.1$ and   $u(x,y,0) = 0.05\sin(x)\sin(y)$.
The domain $\Omega = (0,2\pi)\times (0,2\pi)$ is  partitioned with  $N = 128$ or $256$,
and {\tt SAV-M(3)} is used.

Figure \ref{fig6.1} presents the contour lines and cut lines of two numerical solutions at $t = 200$ computed by {\tt SAV-M(3)} with or without de-aliasing by zero-padding.
It is obvious that they are different when $N = 128$,
but  are quite similar when $N = 256$.
Figure \ref{fig6.2} further shows the
snapshots of the numerical solutions with $N = 256$
at $t = 80$, $84$, and  $88$ computed by {\tt SAV-M(3)} with or without the de-aliasing.
It can be seen that those numerical solutions have some slight differences,
which do not effect the motion  of anti-phase interfaces essentially.
Those results are also consistent with those shown in Figure \ref{fig6.3}, which gives the cut lines of numerical solutions at $t = 80, 84, 88$, and  $92$.
\end{example}

\begin{figure}
\centering
{\includegraphics[width=7.5cm]{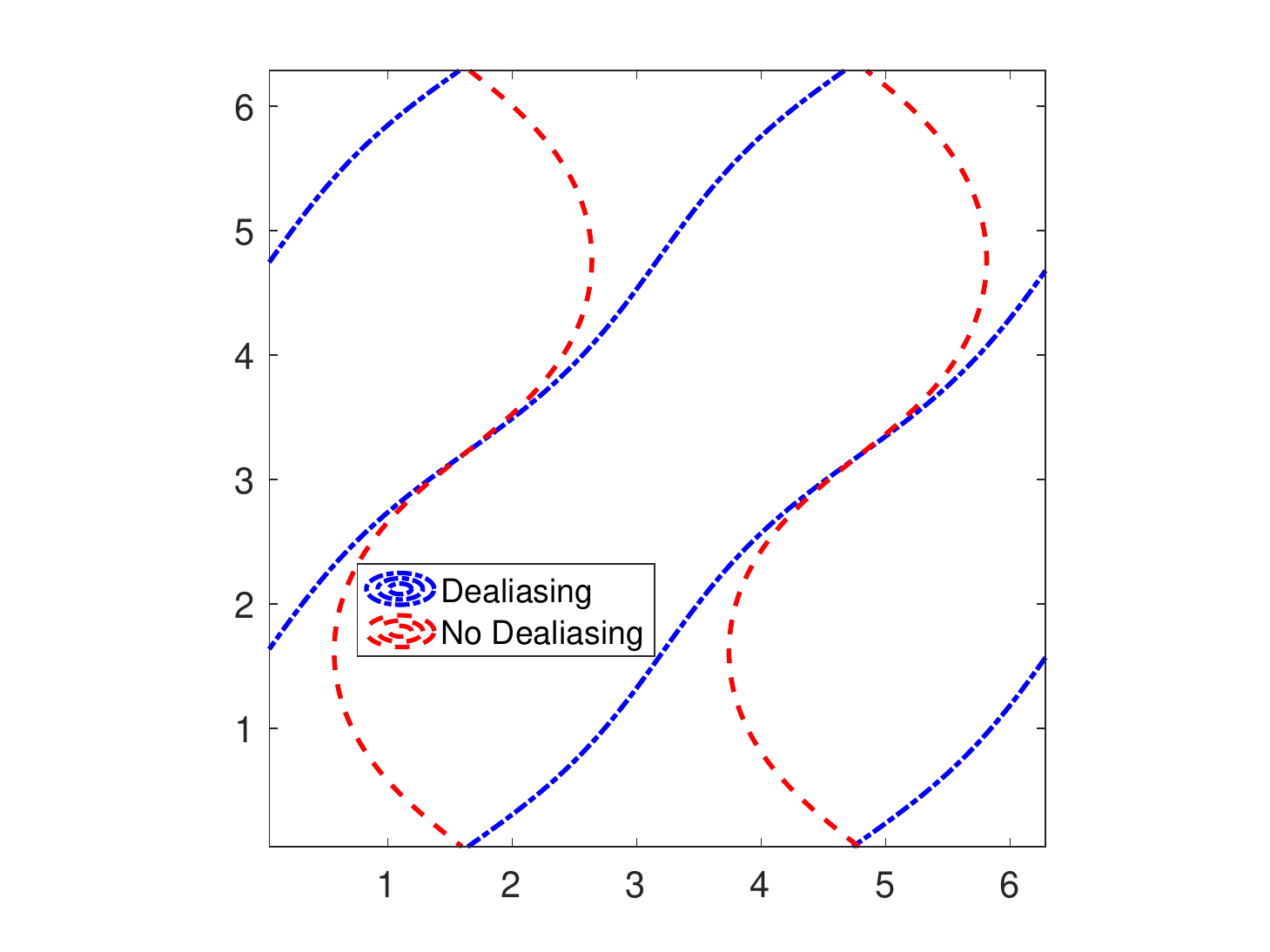}}
{\includegraphics[width=7.5cm]{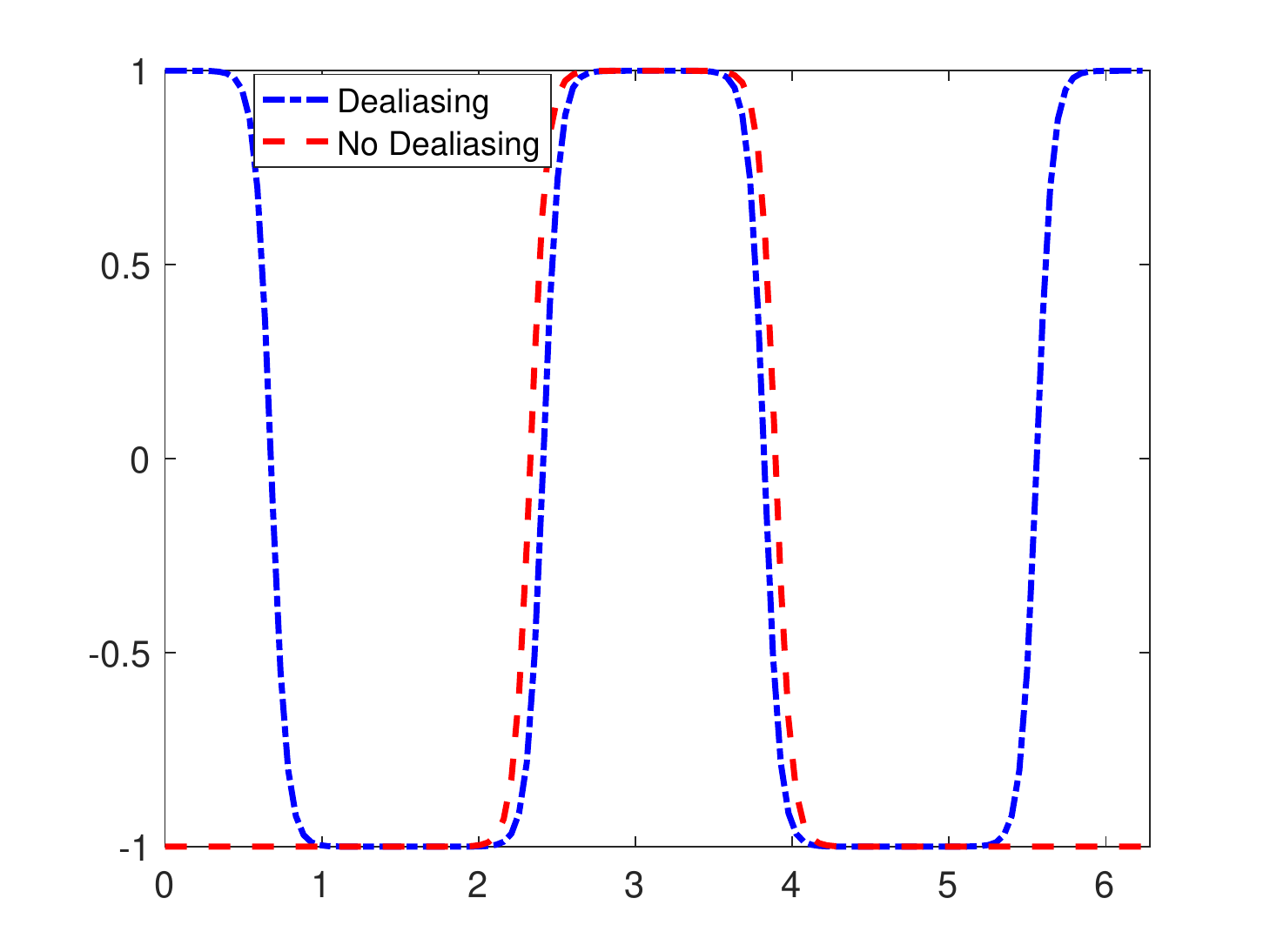}}

{\includegraphics[width=7.5cm]{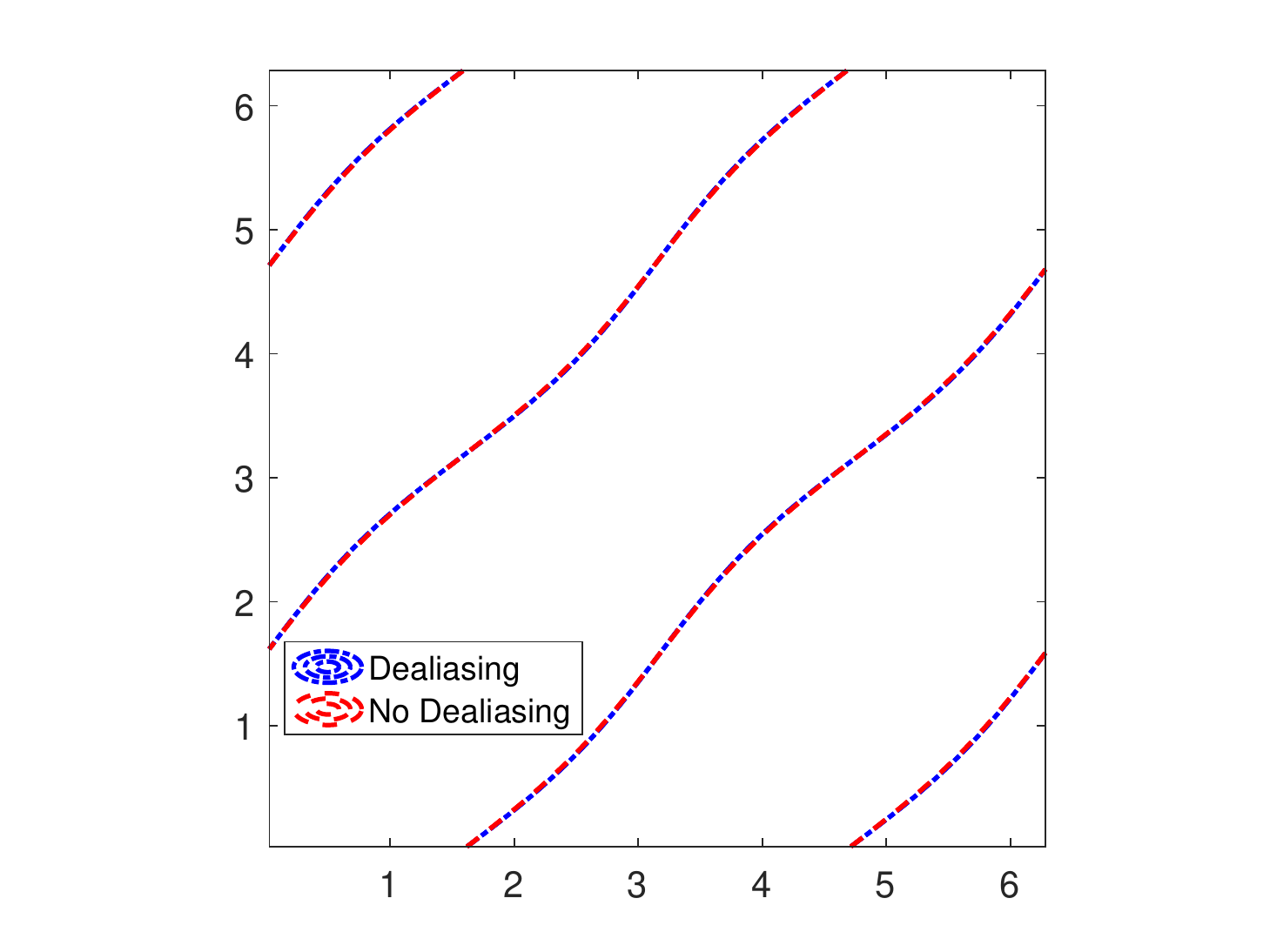}}
{\includegraphics[width=7.5cm]{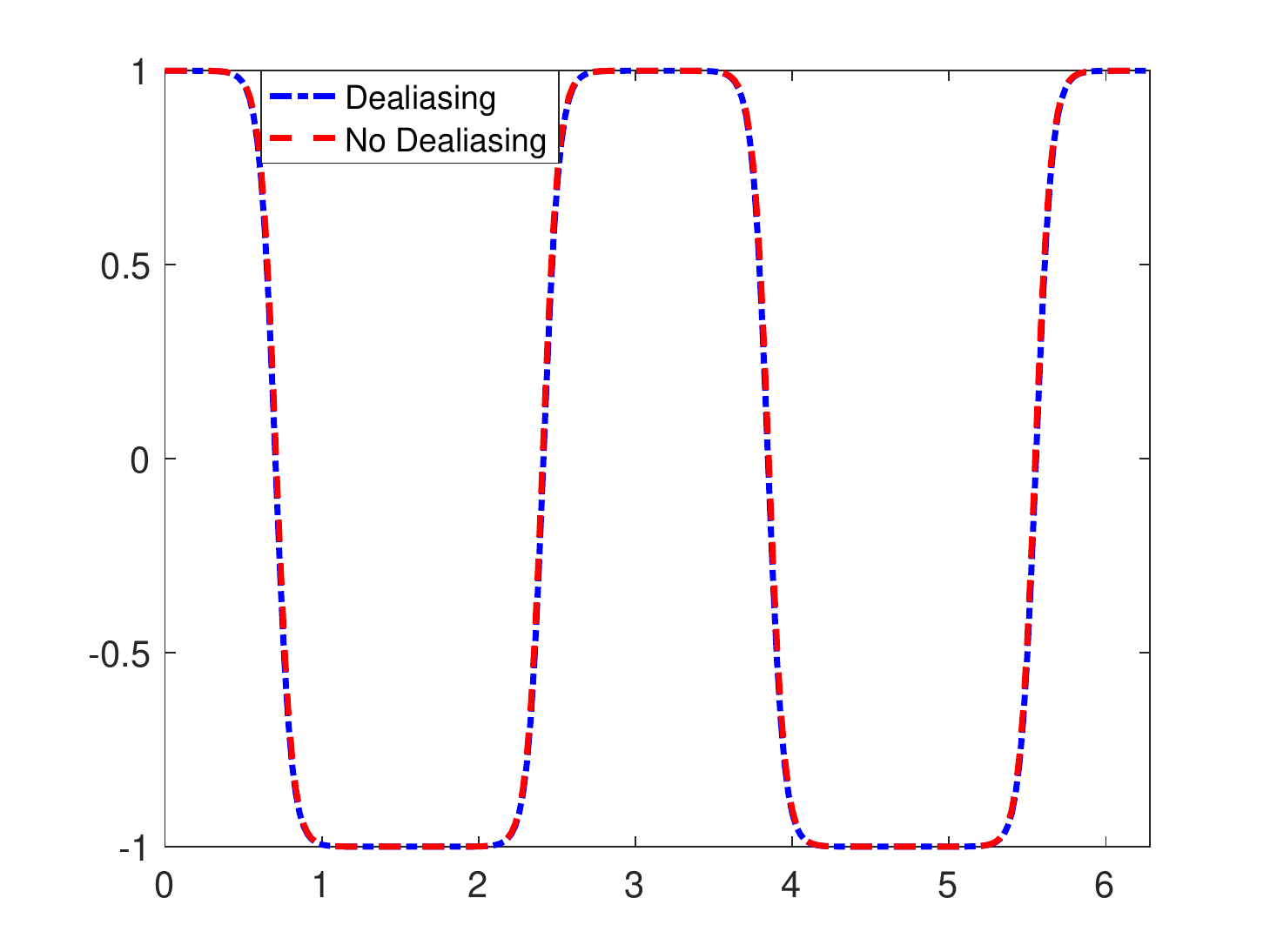}}
\caption{Example \ref{Exp6.1.0}.
Left: contour lines of  $u$ with the value of $-0.1$;
right: cut lines of the numerical solutions along $y=2\pi-x$ with $x\in[0,2\pi]$.  Top: $N=128$; bottom: $N=256$.}
\label{fig6.1}
\end{figure}

\begin{figure}
\includegraphics[width=16cm,height=9cm]{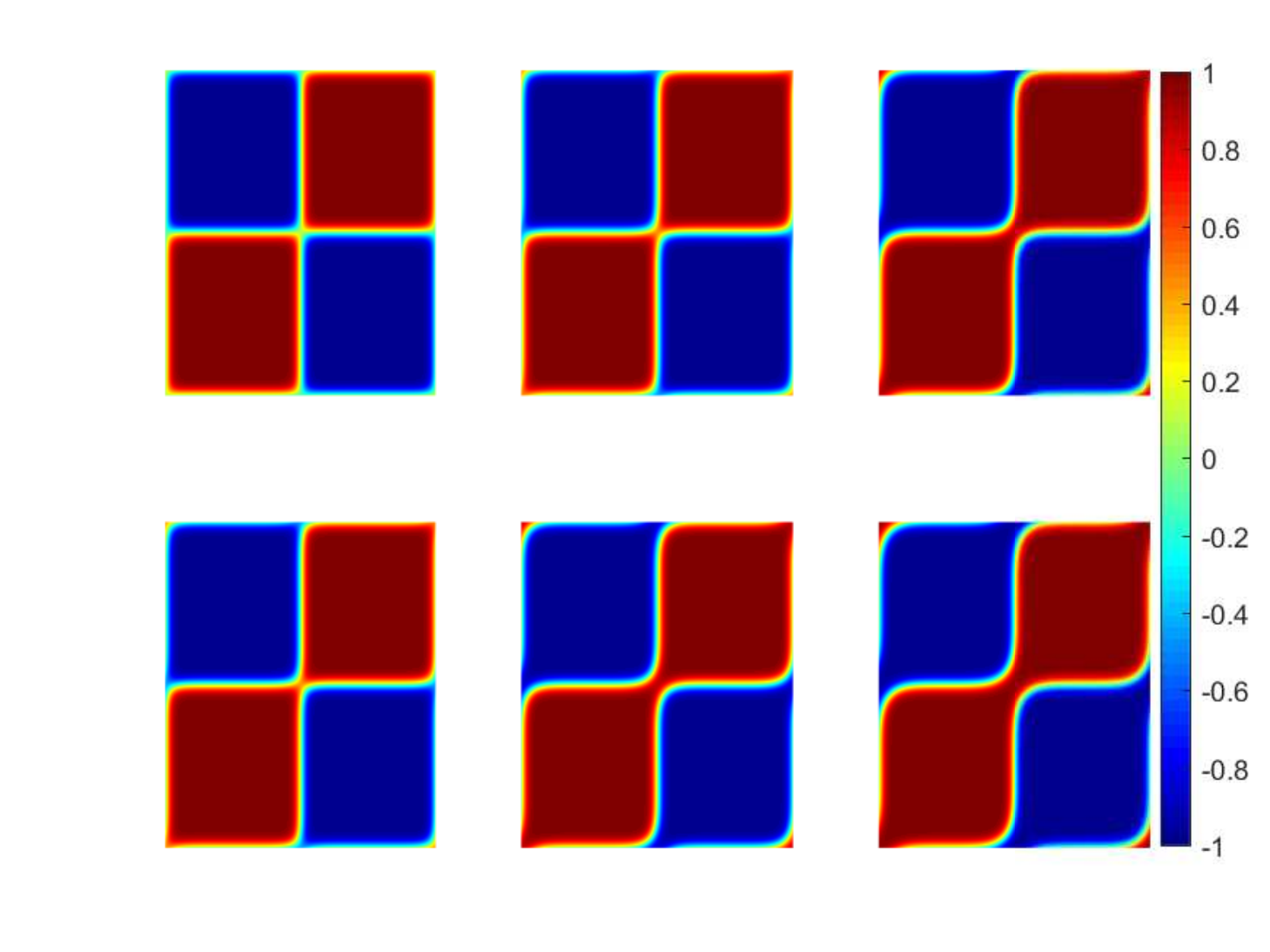}
\caption{Example \ref{Exp6.1.0}.   Snapshots of the numerical solutions at $t = 80$, $84$, and  $88$ derived by {\tt SAV-M(3)} with  (Top) and without (Bottom) the de-aliasing.}
\label{fig6.2}
\end{figure}

\begin{figure}
\centering
{\includegraphics[width=7.5cm]{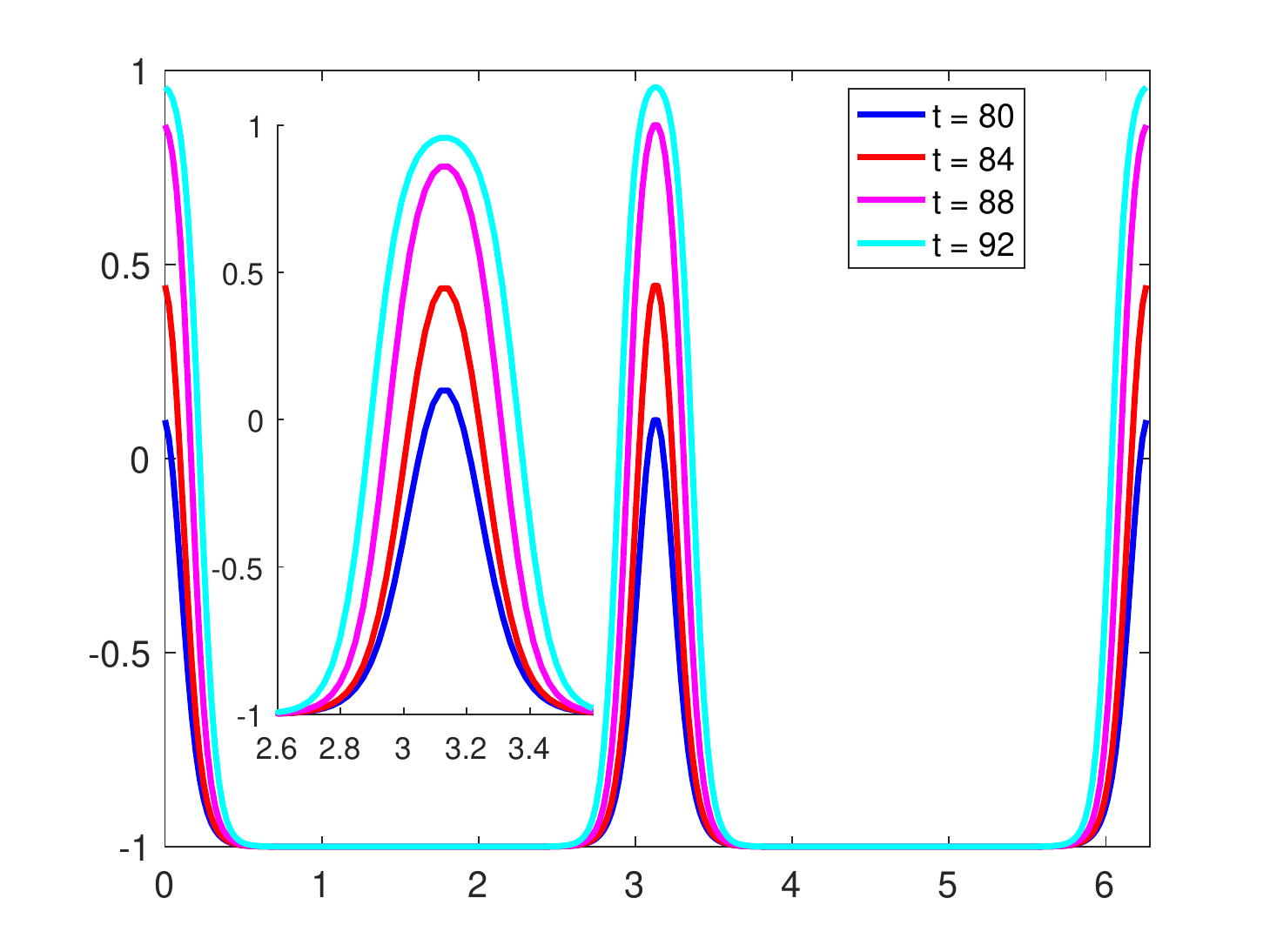}}
{\includegraphics[width=7.5cm]{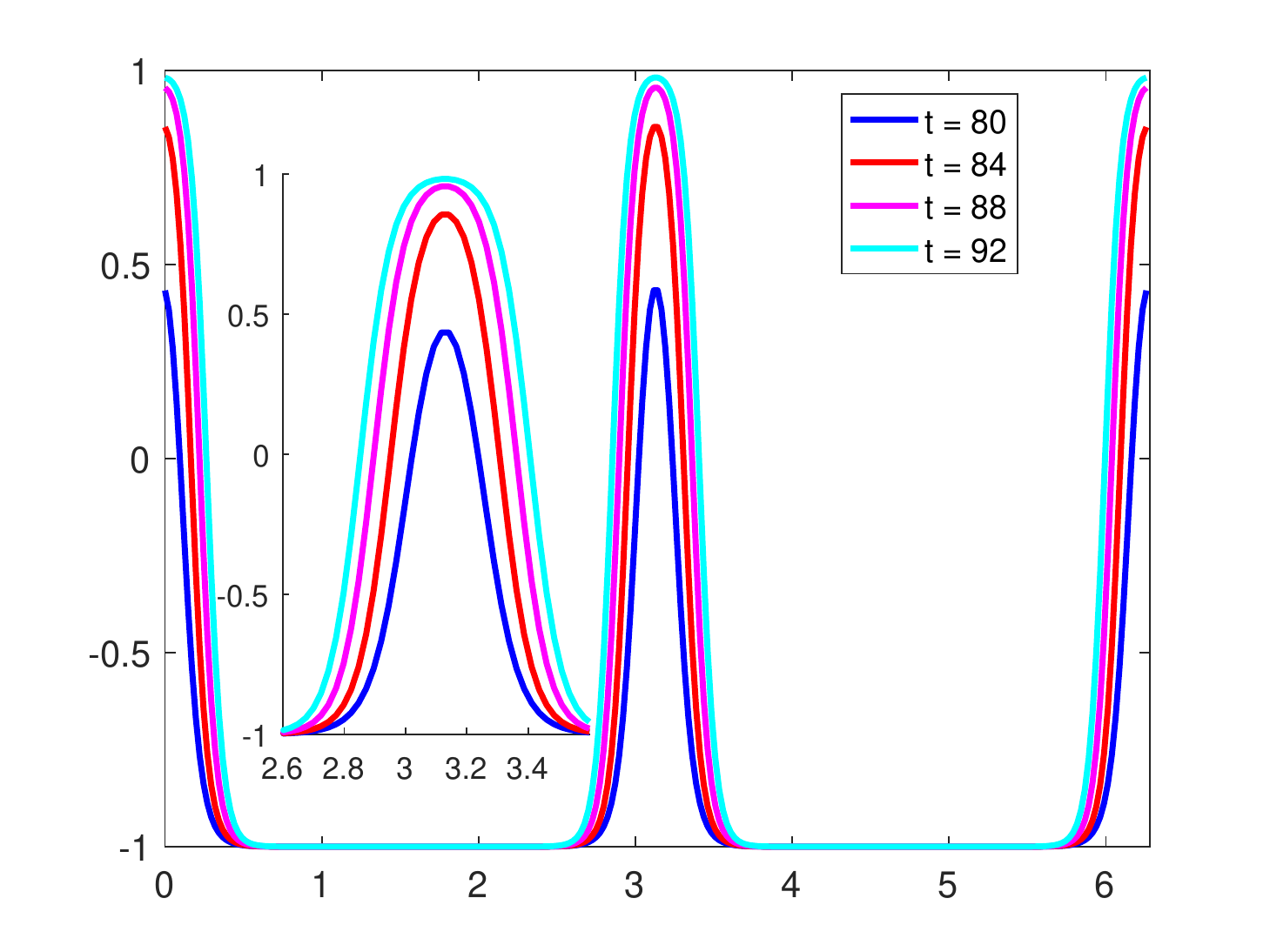}}
\caption{Example \ref{Exp6.1.0}.
Cut lines of the numerical solutions along $y=2\pi-x$, $x\in[0,2\pi]$, derived by {\tt SAV-M(3)}
with  (Left) and  without (Right) the de-aliasing.}
\label{fig6.3}
\end{figure}

\begin{example}   \label{Exp6.1.1}
This example is used to discuss the modified- and original-energy stabilities of {\tt SAV-M(1)}$\sim${\tt SAV-M(4)} and {\tt G-SAV-M(1)}$\sim${\tt G-SAV-M(4)} for the Allen-Cahn model \eqref{6.1.1}. For this purpose,  the domain $\Omega = (0,2\pi)\times (0,2\pi)$ is uniformly partitioned with  $N = 128$, the parameter $\epsilon=0.1$, and the initial value is chosen as  $u(x,y,0) = 0.1\times \mbox{\tt rand}(x,y) - 0.05$, where $\mbox{\tt rand}(\cdot,\cdot)$ generates a random number between $0$ and $1$.

Figure \ref{fig5.1} presents the discrete total modified-energy curves of {\tt SAV-M(1)}$\sim${\tt SAV-M(4)} and {\tt G-SAV-M(1)}$\sim${\tt G-SAV-M(4)} defined respectively in Theorem \ref{Thm3.1} and Remark \ref{rem:2.4}. One can see that all those modified-energy curves are  monotonically decreasing and consistent with the theoretical results.
Figure \ref{fig5.2} provides the discrete total original-energy curves of {\tt SAV-M(1)}$\sim${\tt SAV-M(4)} and {\tt G-SAV-M(1)}$\sim${\tt G-SAV-M(4)}, and Figure \ref{fig5.3} presents the numerical solution at $t = 200$ derived by {\tt G-SAV-M(4)} with $\tau = 2$ and $1$.
Those results show that the numerical solution shown in Figure \ref{fig5.3} with $\tau = 1$ is quite similar to that in \cite{TanZQ22}, but when $\tau = 2$, the solution is inaccurate or non-physical  and the original energy is not monotonically decreasing as shown in Figure \ref{fig5.2}. It indicates that  some time stepsize constraints are necessary to ensure the original-energy decay.    Remark \ref{remk6.1} will discuss  the time stepsize constraints of {\tt SAV-M(1)}$\sim${\tt SAV-M(4)} and {\tt G-SAV-M(1)}$\sim${\tt G-SAV-M(4)} for  the Allen-Cahn model \eqref{6.1.1} by using the stability regions of our SAV-GL schemes for the test equation.
\end{example}

\begin{figure}
\centering
{\includegraphics[width=7.5cm]{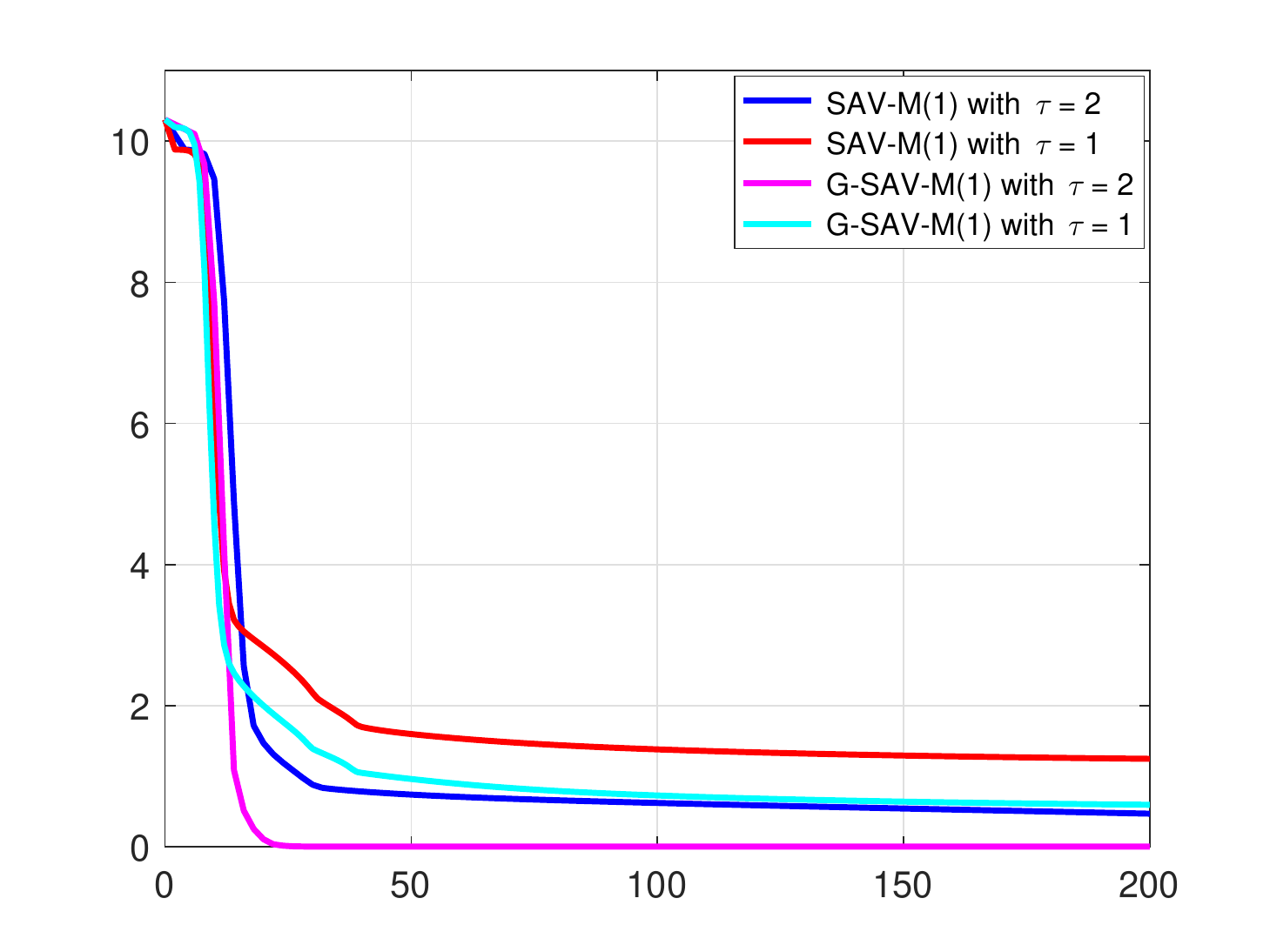} }
{\includegraphics[width=7.5cm]{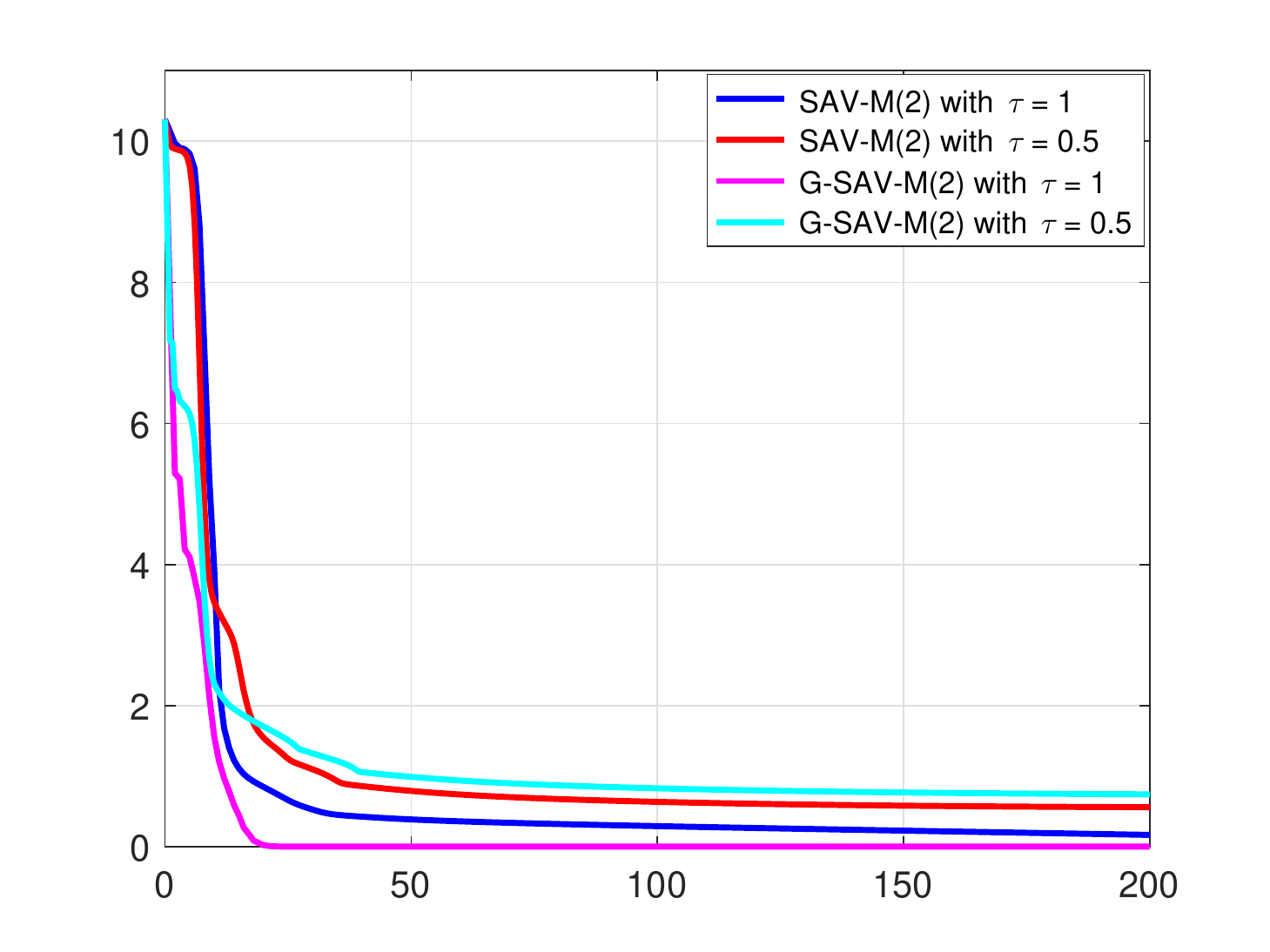} }

{\includegraphics[width=7.5cm]{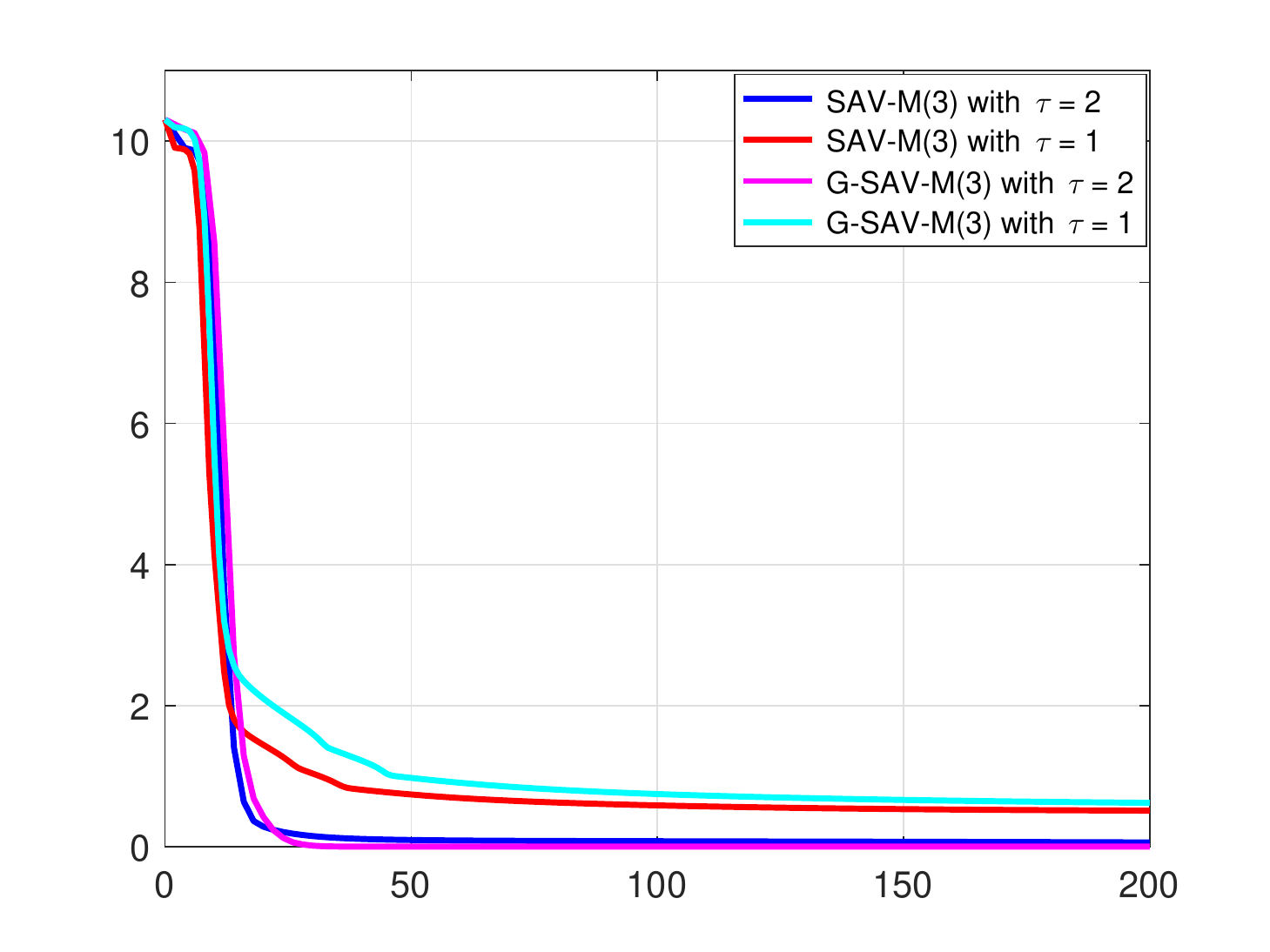} }
{\includegraphics[width=7.5cm]{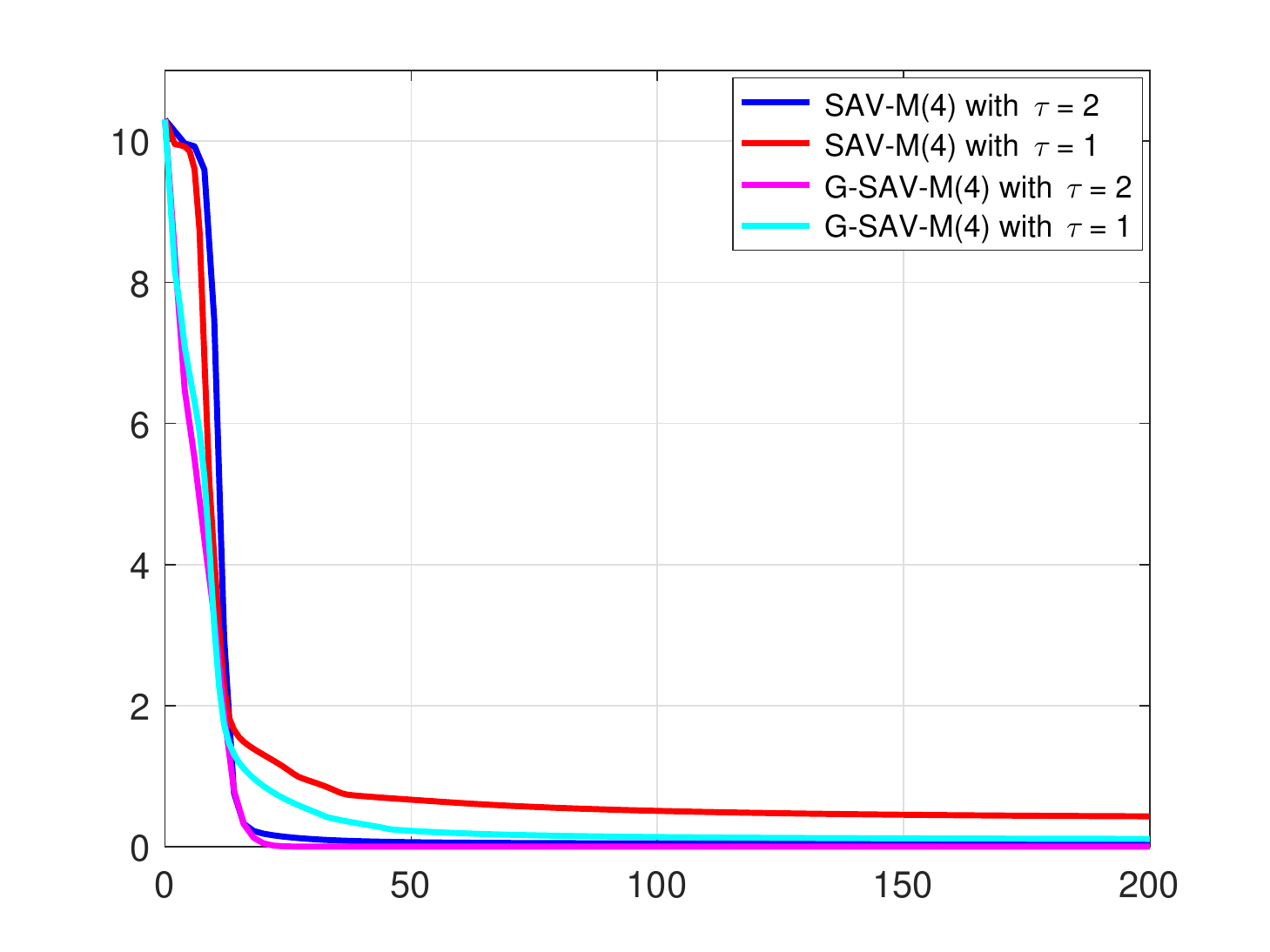} }
\caption{Example \ref{Exp6.1.1}. The time evolution of the total modified-energies of {\tt SAV-M(1)}$\sim${\tt SAV-M(4)} and {\tt G-SAV-M(1)}$\sim${\tt G-SAV-M(4)} for the Allen-Cahn model \eqref{6.1.1}.}
\label{fig5.1}
\end{figure}

\begin{figure}
\centering
{\includegraphics[width=7.5cm]{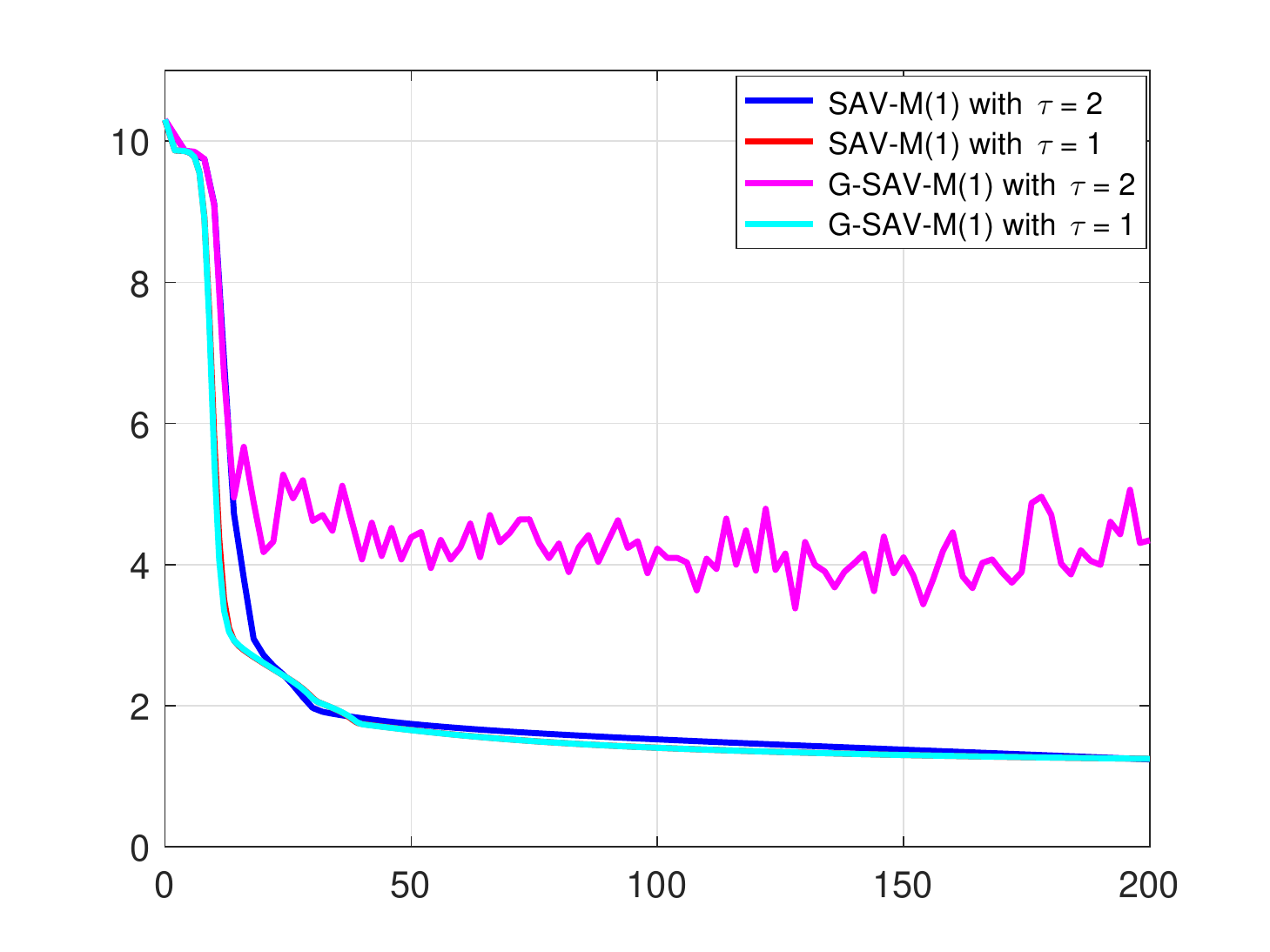} }
{\includegraphics[width=7.5cm]{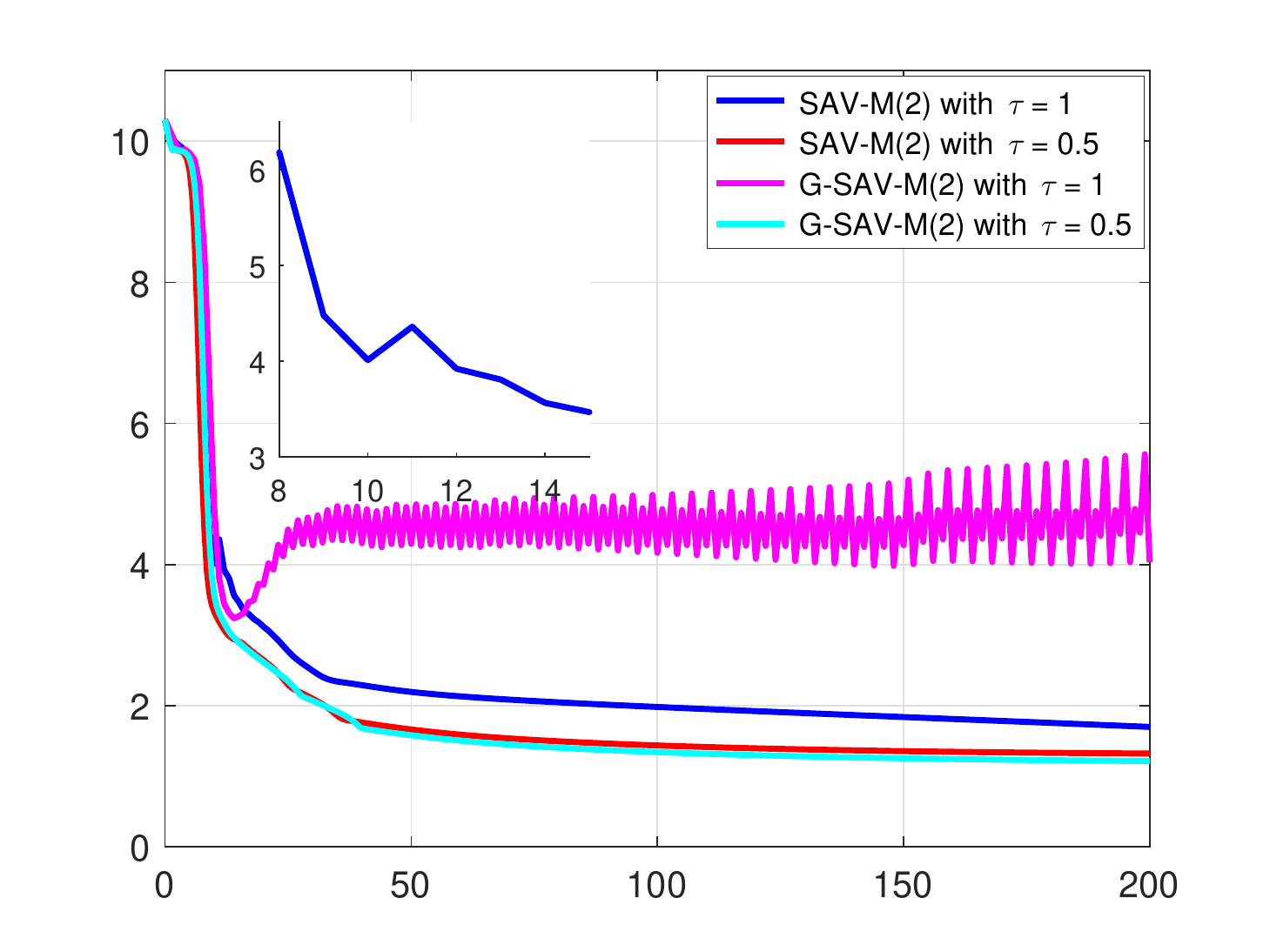} }

{\includegraphics[width=7.5cm]{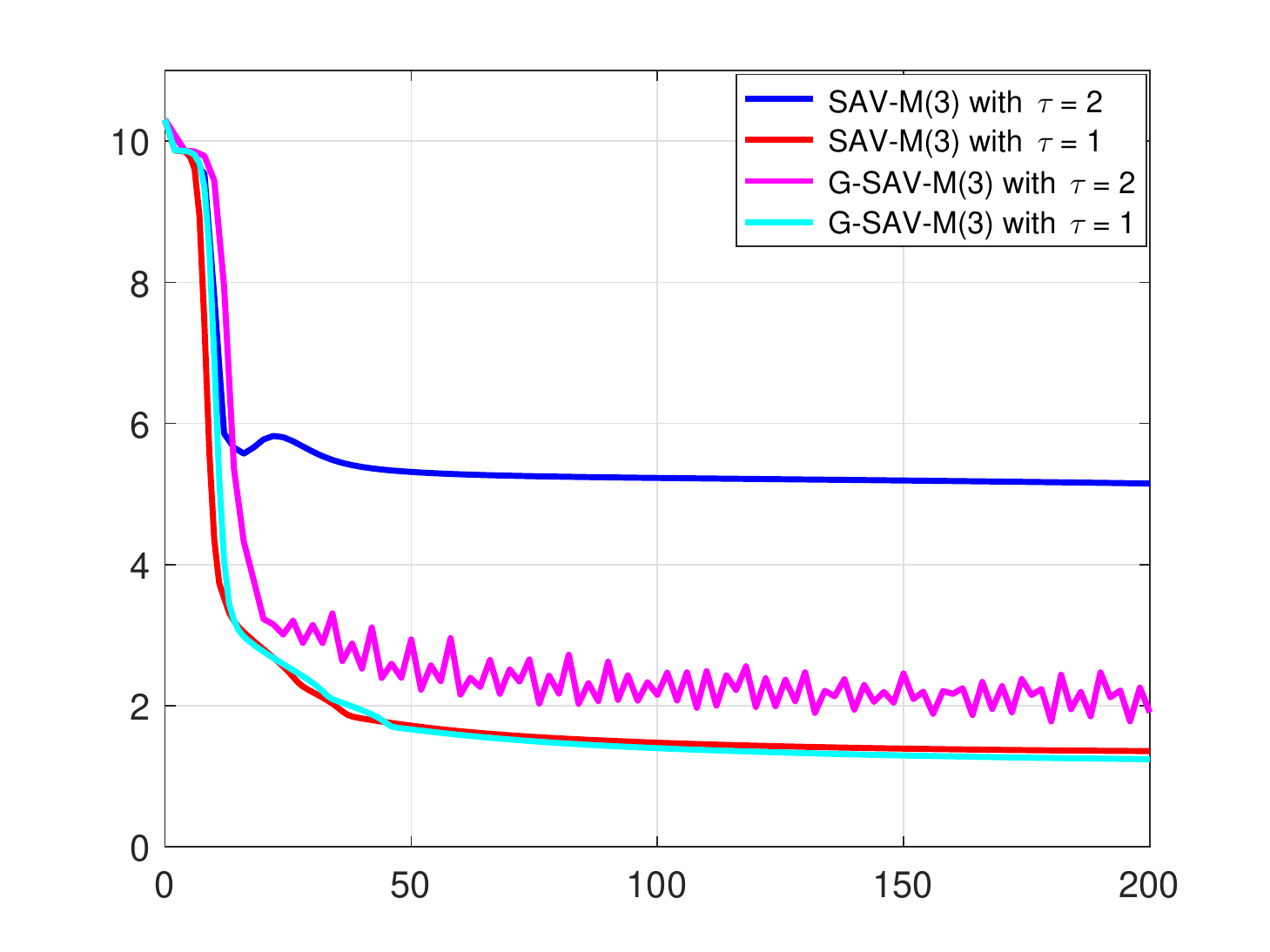} }
{\includegraphics[width=7.5cm]{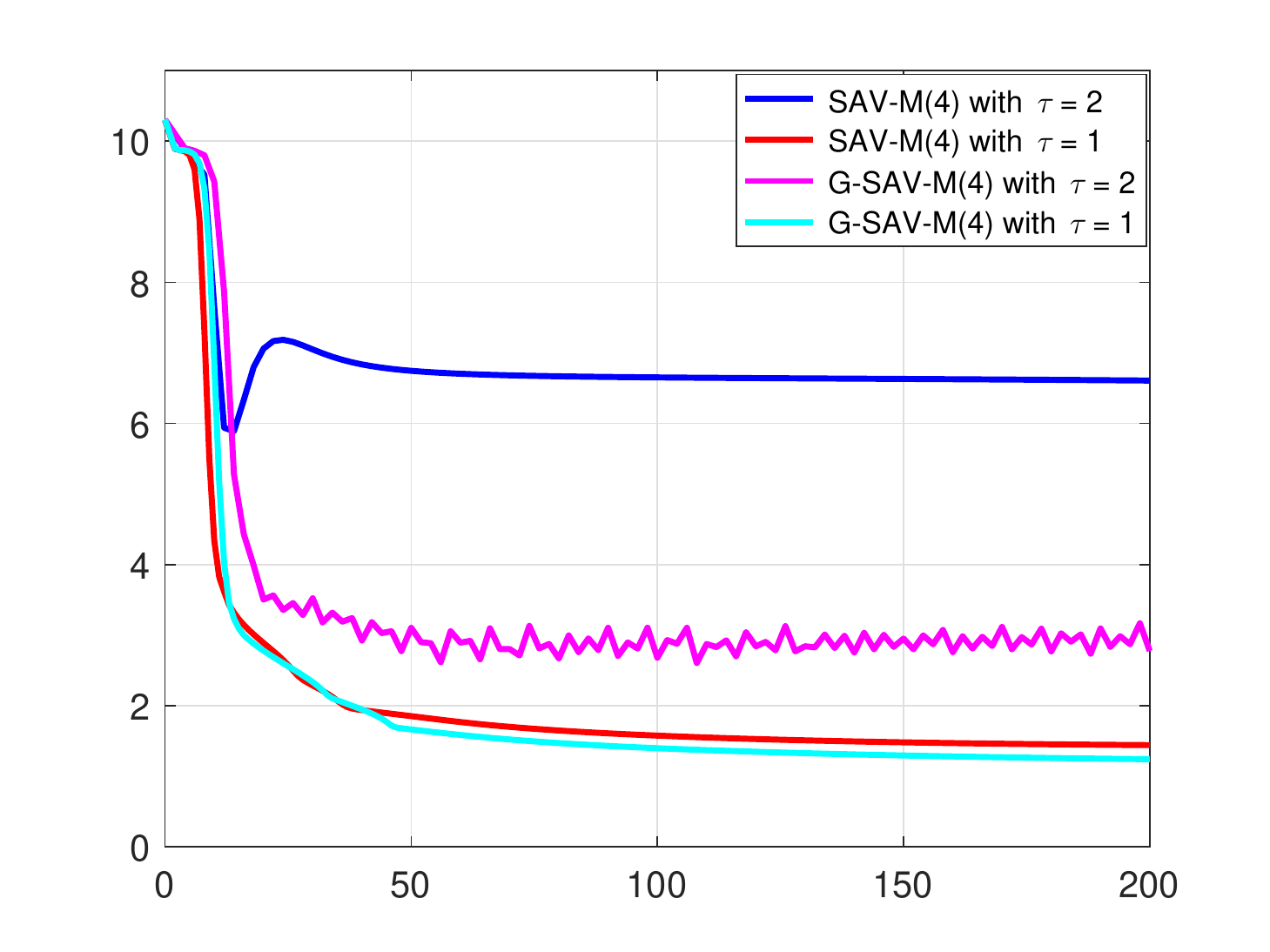} }
\caption{Same as Figure \ref{fig5.1}, except for the discrete total original-energies.
}
\label{fig5.2}
\end{figure}

\begin{figure}
\begin{minipage}{0.48\linewidth}
  \centerline{\includegraphics[width=7cm,height=5cm]{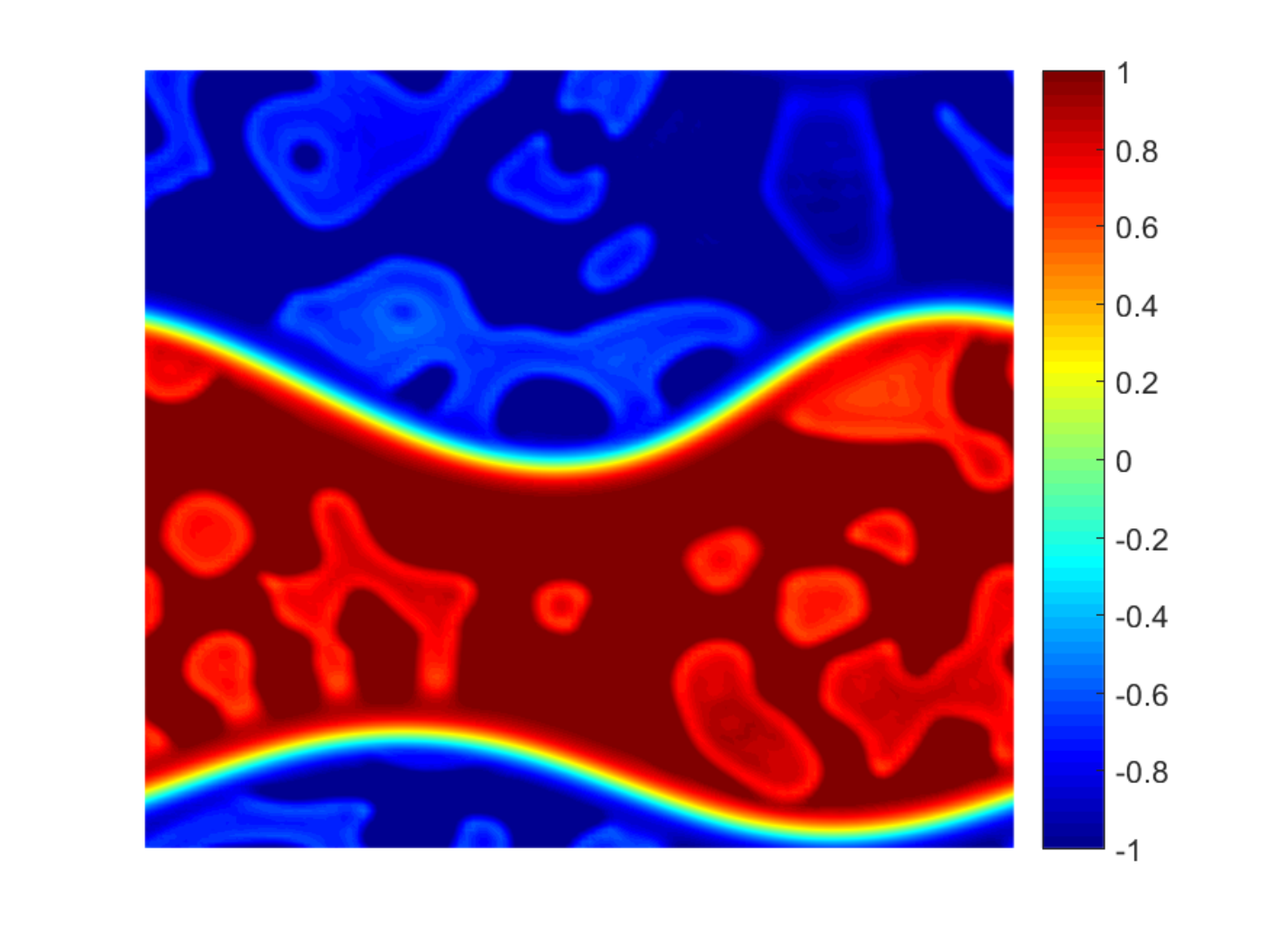}}
\end{minipage}
\hfill
\begin{minipage}{0.48\linewidth}
  \centerline{\includegraphics[width=7cm,height=5cm]{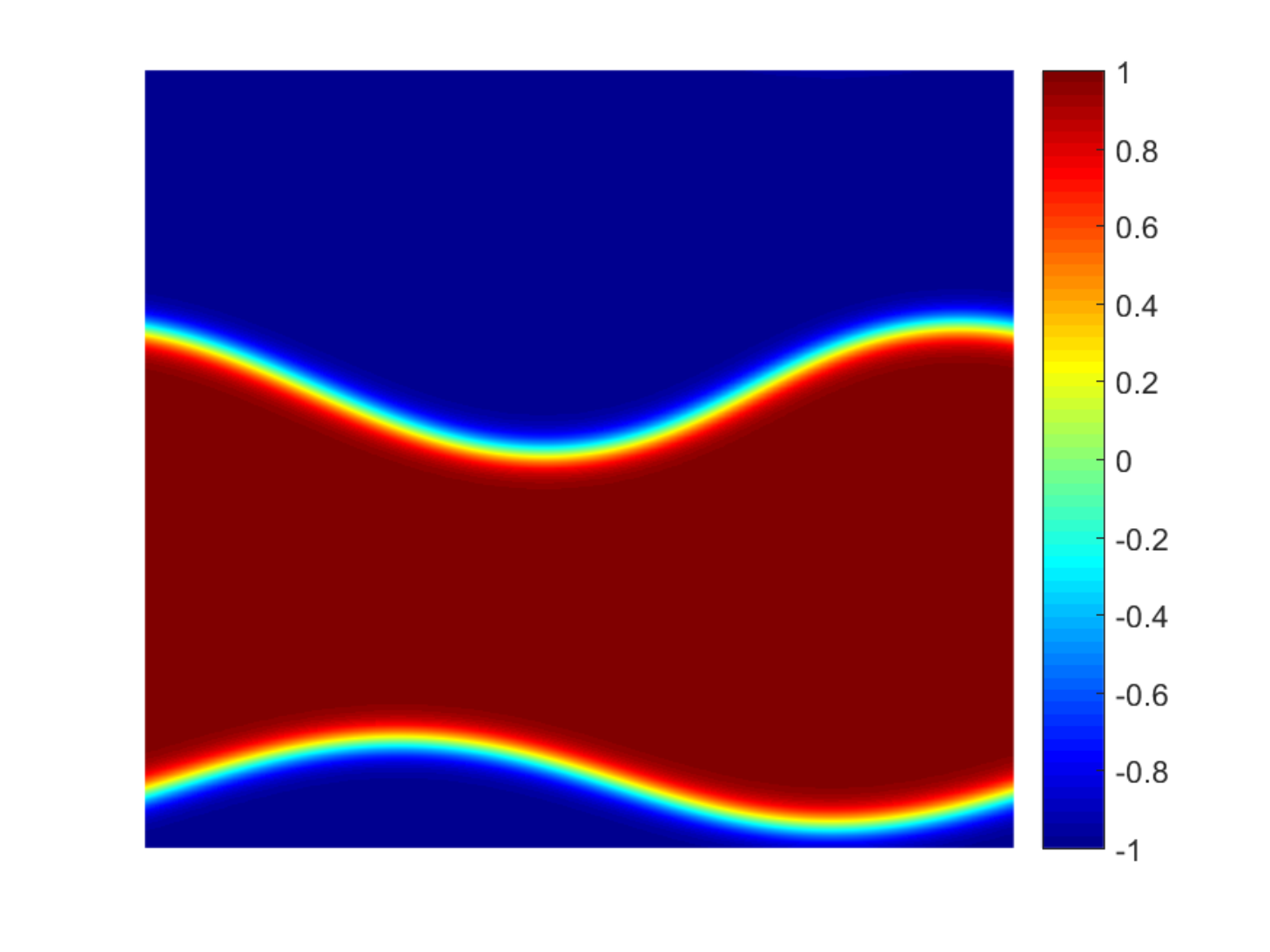}}
\end{minipage}
\caption{Example \ref{Exp6.1.1}. Numerical solutions at $t=200$ computed by {\tt G-SAV-M(4)} with $\tau = 2$ (Left) and $1$ (Right), respectively.}
\label{fig5.3}
\end{figure}

\begin{remark}  \label{remk6.1}
Applying the Fourier pseudo-spectral method to the Allen-Cahn model \eqref{6.1.1} yields the ODE system 
\begin{align}   \label{6.1.4}
\frac{d \hat{u}_{k,l}}{d t } = - \frac{4\pi^2\epsilon^2}{L^2}\left(k^2 + l^2 \right) \hat{u}_{k,l} + \hat{u}_{k,l} - \hat{w}_{k,l},~~~~ (k,l) \in \widehat{\mathbb{S}}_N,
\end{align}
where {$\widehat{\mathbb{S}}_N = \left\{ (k,l)\in\mathbb{Z}^2 | -\frac{N}{2}+1\le k,l \le \frac{N}{2}\right\}$,} $\{\hat{w}_{k,l}\}$ are the discrete Fourier coefficients of the cubic term $u^3$ and given by
\begin{align}   \label{6.1.5}
\widehat{w}_{kl} = \frac{1}{N^4}\sum_{(m,n),(p,q)\in\widehat{\mathbb{S}}_N}  \widehat{u}_{mn}\widehat{u}_{pq} \widehat{u}_{k-m-p,l-n-q}.
\end{align}
The system \eqref{6.1.4} may be viewed as  the test equation \eqref{8.3.1} with $\xi = - \epsilon^2\left(k^2 + l^2 \right)$ and 
\begin{align*}
\zeta = 1 - \frac{3}{N^4} \sum_{(m,n)\in\widehat{\mathbb{S}}_N} \hat{u}_{m,n} \hat{u}_{-m,-n} = 1 - \frac{3}{N^4} \sum_{(m,n)\in\widehat{\mathbb{S}}_N} |\hat{u}_{m,n}|^2 = 1 - \frac{3}{N^2} \sum_{(i,j)\in\mathbb{S}_N} |u_{i,j}|^2,
\end{align*}
where $\mathbb{S}_N = \left\{ (i,j)\in\mathbb{Z}^2 | 1\le i,j \le N\right\}$, $\hat{u}_{-m,-n} = \bar{\hat{u}}_{m,n}$ with $\bar{\hat{u}}_{m,n}$ being the complex conjugate of $\hat{u}_{-m,-n}$ and  Parseval's theorem have been used.
For Example \ref{Exp6.1.1}, Figure \ref{fig5.4} plots the curve $\psi^n = \frac{3}{N^2} \sum\limits_{(i,j)\in\mathbb{S}_N} |u_{i,j}^n|^2$ derived by {\tt SAV-M(1)}$\sim${\tt SAV-M(4)} and {\tt G-SAV-M(1)}$\sim${\tt G-SAV-M(4)}. It shows that $\psi^n \lesssim 2.7$ so that $\zeta\gtrsim -1.7$.
Thus, one can take $\zeta \approx -1.7$ and then {estimate} the time stepsize  according to  \ref{Appx3}.
Specifically,  when the parameters $(\alpha_0,\beta_0,\beta_2) = (0,0,1)$,  $$\tau  < \min\left\{\frac{2}{(2\beta_2-1)\xi-(2\beta_2+1)\zeta}: \zeta < \frac{2\beta_2-1}{2\beta_2+1}\xi \right\} = \left\{ \frac{2}{\max(\xi-3\zeta)}: \zeta < \frac{1}{3}\xi\right\},$$ which implies $\tau \lesssim 0.3922$ since $\max\{\xi-3\zeta: \zeta < \frac{1}{3}\xi \} = -3\zeta$;
when  $(\alpha_0,\beta_0,\beta_2) = (-1/3,3/12,3/4)$,
 $$\tau < \min\left\{\frac{1+\alpha_0}{(2\beta_0+\alpha_0)\xi-\zeta}: \zeta < (2\beta_0+\alpha_0)\xi \right\}  = \left\{ \frac{4}{3\max(\xi-2\zeta)}: \zeta < \frac{1}{2}\xi\right\},$$
 which  gives $\tau  \lesssim 0.3922$ by using $\max\{\xi-2\zeta: \zeta < \frac{1}{2}\xi \} = -2\zeta$;
when  $(\alpha_0,\beta_0,\beta_2) = (1/3,0,2/3)$,
$$\tau < \min\left\{\frac{1+\alpha_0}{(2\beta_0+\alpha_0)\xi-\zeta}: \zeta < (2\beta_0+\alpha_0)\xi \right\} = \left\{ \frac{4}{\max(\xi-3\zeta)}: \zeta < \frac{1}{3}\xi\right\}, $$
so that $\tau < -\frac{4}{3\zeta} \lesssim 0.7843$;
and when $(\alpha_0,\beta_0,\beta_2) = (1/3,-1/6,1/2)$,
$$\tau < - \frac{1+\alpha_0}{\zeta} \lesssim 0.7843.$$
Note that the above  time {stepsize} constraints for {\tt SAV-M(1)}$\sim${\tt SAV-M(4)} and {\tt G-SAV-M(1)}$\sim${\tt G-SAV-M(4)} are sufficient and slightly more severer than
 them used in the numerical experiments on ensuring the original-energy decay of  Example \ref{Exp6.1.1}; and  although the time discretization with $(\alpha_0,\beta_0,\beta_2) = (1/3,-1/6,1/2)$ is not algebraically stable,  the time stepsizes for  {\tt SAV-M(4)} and {\tt G-SAV-M(4)} are comparable and both two schemes can provide good numerical results of \eqref{6.1.1}.  It is worth noting that for the Allen-Cahn model \eqref{6.1.1}, one can use the maximum principle to give the estimation $\zeta \approx -2$, and then use  \ref{Appx3} to get certain time stepsize conditions, which are also sufficient and have no big difference from the above estimations.
\end{remark}
\begin{figure}
\begin{minipage}{0.48\linewidth}
  \centerline{\includegraphics[width=7cm,height=5cm]{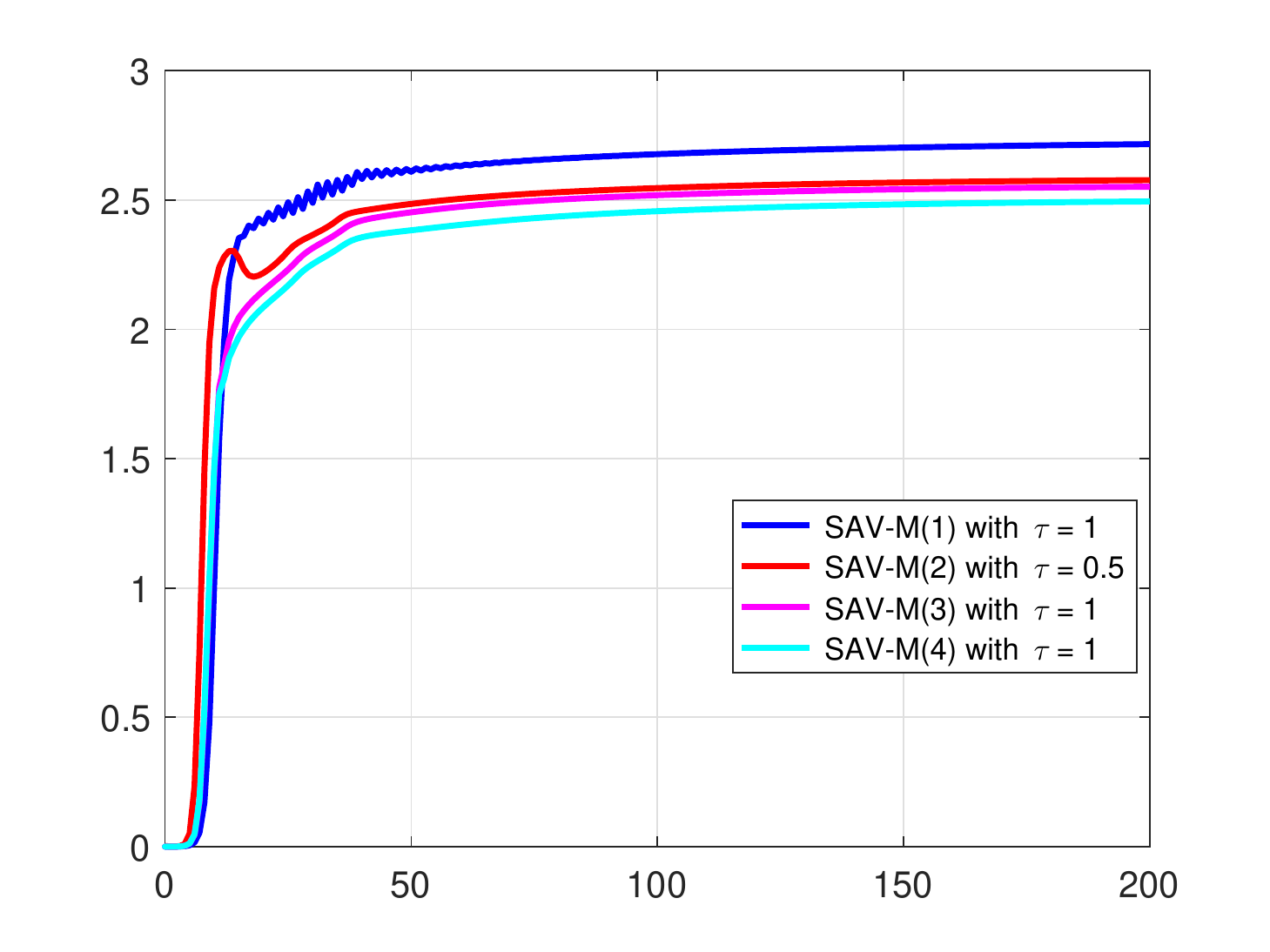}}
\end{minipage}
\hfill
\begin{minipage}{0.48\linewidth}
  \centerline{\includegraphics[width=7cm,height=5cm]{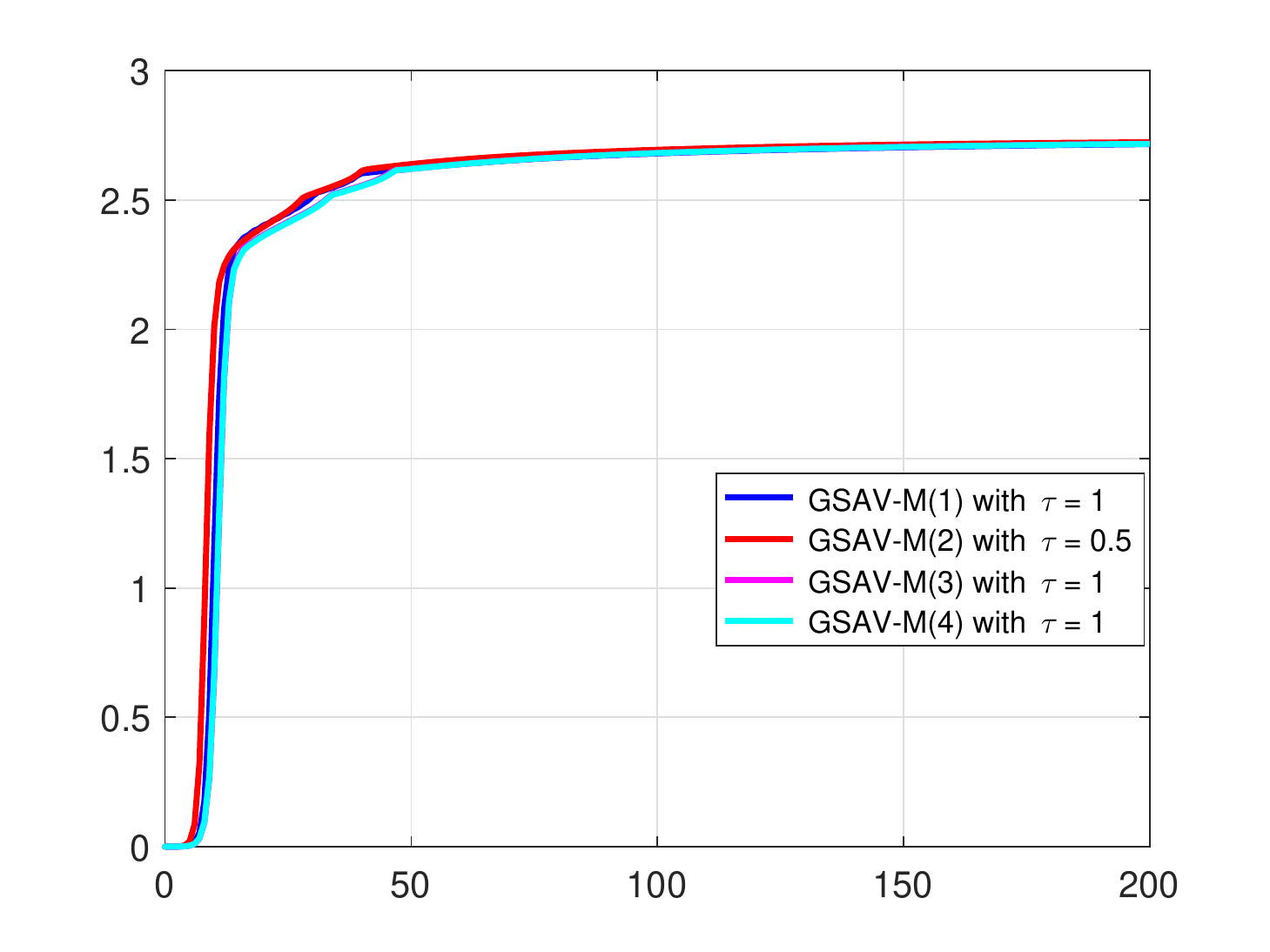}}
\end{minipage}
\caption{Example \ref{Exp6.1.1}.  $\psi^n$ derived by {\tt SAV-M(1)}$\sim${\tt SAV-M(4)} (Left) and {\tt G-SAV-M(1)}$\sim${\tt G-SAV-M(4)} (Right).}
\label{fig5.4}
\end{figure}

\begin{remark}
For the SAV-GL scheme \eqref{3.1.6}, the term $\bar{\psi}^n = \frac{z^{n+\kappa}}{\sqrt{\mathcal{E}_1(\bar{u}^{n+\kappa})+C_0}}$ should  be precisely considered in discussing the time stepsize constraints, theoretically. 
However, unfortunately, it is difficult to   estimate exactly  $\bar{\psi}^n$, even if it is equal to one at the continuous level.
Figure \ref{fig5.1.5} plots $\bar{\psi}^n$ derived by {\tt SAV-M(1)}$\sim${\tt SAV-M(4)} with $\tau = 1$ and $0.1$, from which one can observe $\bar{\psi}^n\lesssim 1$.
This is the reason why we  take $\bar{\psi}^n \approx 1$  for convenience and derive the time stepsize constraints for {\tt SAV-M(1)}$\sim${\tt SAV-M(4)} in Remark \ref{remk6.1}.

\end{remark}

\begin{figure}[!h]
\begin{minipage}{0.48\linewidth}
  \centerline{\includegraphics[width=7cm,height=5cm]{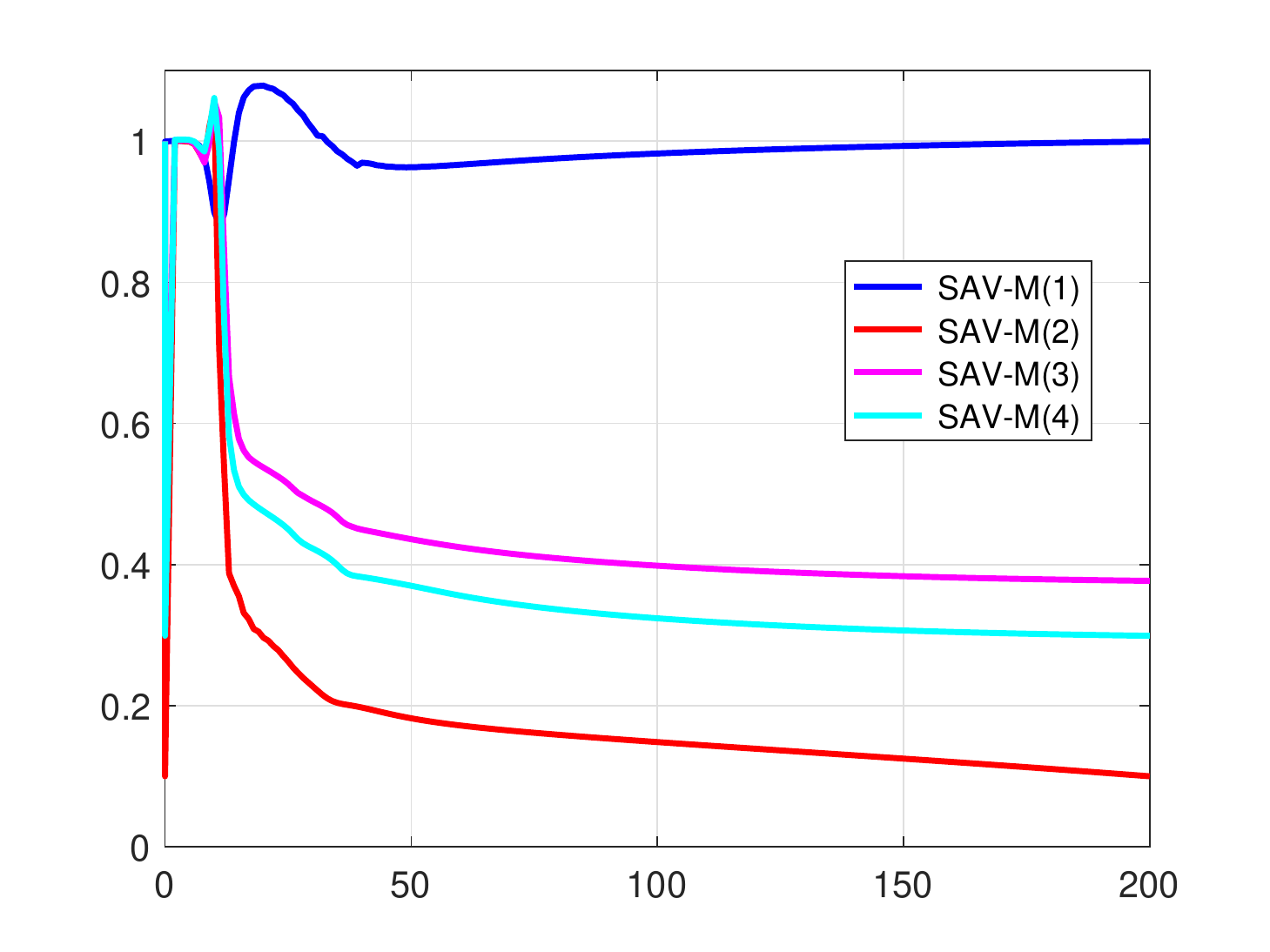}}
\end{minipage}
\hfill
\begin{minipage}{0.48\linewidth}
  \centerline{\includegraphics[width=7cm,height=5cm]{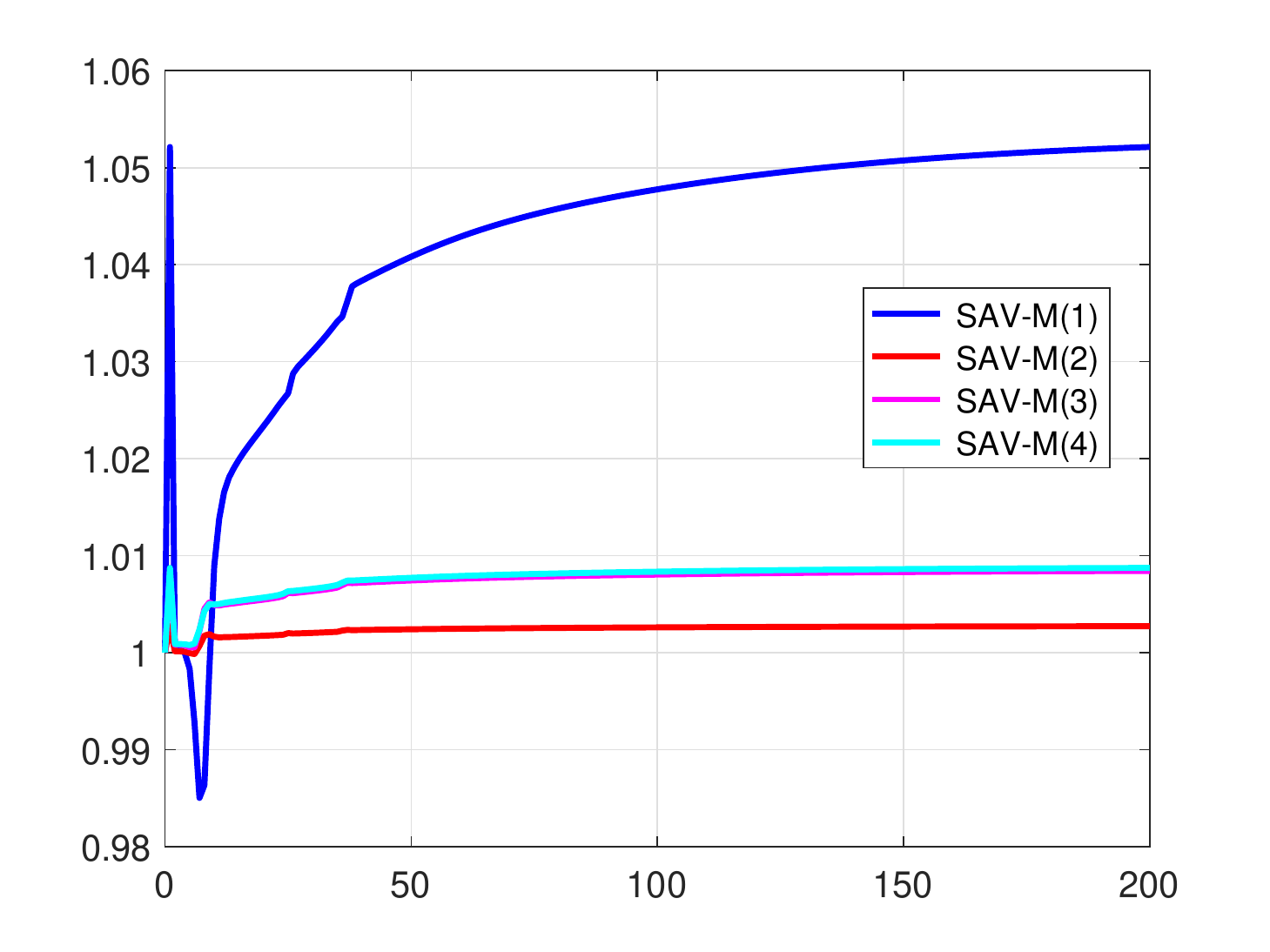}}
\end{minipage}
\caption{Example \ref{Exp6.1.1}.   $\bar{\psi}^n$   derived by {\tt SAV-M(1)}$\sim${\tt SAV-M(4)} with $\tau = 1$ (Left) and $0.1$ (Right), respectively.}
\label{fig5.1.5}
\end{figure}

\subsection{Cahn-Hilliard model}

The Cahn-Hilliard  model
\begin{align}   \label{6.2.1}
\frac{\partial u}{\partial t} = \Delta \left( -\epsilon^2\Delta u + u^3 - u \right),~~~\vec{x}\in \Omega,~t>0,
\end{align}
is derived from the $H^{-1}$ gradient flow of the free energy \eqref{6.1.3}, and  describes the complicated phase separation and coarsening phenomena \cite{Cahn58}.

In order to apply {\tt SAV-M(1)}$\sim${\tt SAV-M(4)} and {\tt G-SAV-M(1)}$\sim${\tt G-SAV-M(4)} to the Cahn-Hilliard model \eqref{6.2.1}, the operators $\mathcal{L}$, $\mathcal{G}$ and  the energy $\mathcal{E}_1(u)$ are taken as
\[ \mathcal{G} = \Delta,~~~ \mathcal{L} = - \epsilon^2 \Delta, ~~~ \mathcal{E}_1(u) = \int_{\Omega} \frac{1}{4}\left(u^2 - 1 \right)^2 dx. \]

\begin{example} \label{Exp6.2.0}
This example is used to check the  effectiveness  of the de-aliasing  by zero-padding for  \eqref{6.2.1}. We take $\epsilon = 0.1$,
and the initial data  $u(x,y,0) = 0.05\sin(x)\sin(y)$.
The domain $\Omega = (0,2\pi)\times (0,2\pi)$ is partitioned with  $N = 128$ or $256$, and {\tt SAV-M(3)} is used.

Figure \ref{fig6.4} gives the contour lines and cut lines of the numerical solutions at $t = 200$ derived by {\tt SAV-M(3)} with or without de-aliasing.  Visible difference between the numerical solutions with $N = 128$ can be observed, but  the  difference
 is indistinguishable when $N = 256$.
For $N = 256$,  Figure \ref{fig6.5} presents the
snapshots of the numerical solutions at $t = 7.5$, $8$, and  $8.5$, while Figure \ref{fig6.6} shows the cut lines of numerical solutions at $t = 7.5, 8, 8.5$, and  $9$.
 It is shown that there are some slight differences between those numerical solutions.
\end{example}

\begin{figure}
\centering
{\includegraphics[width=7.5cm]{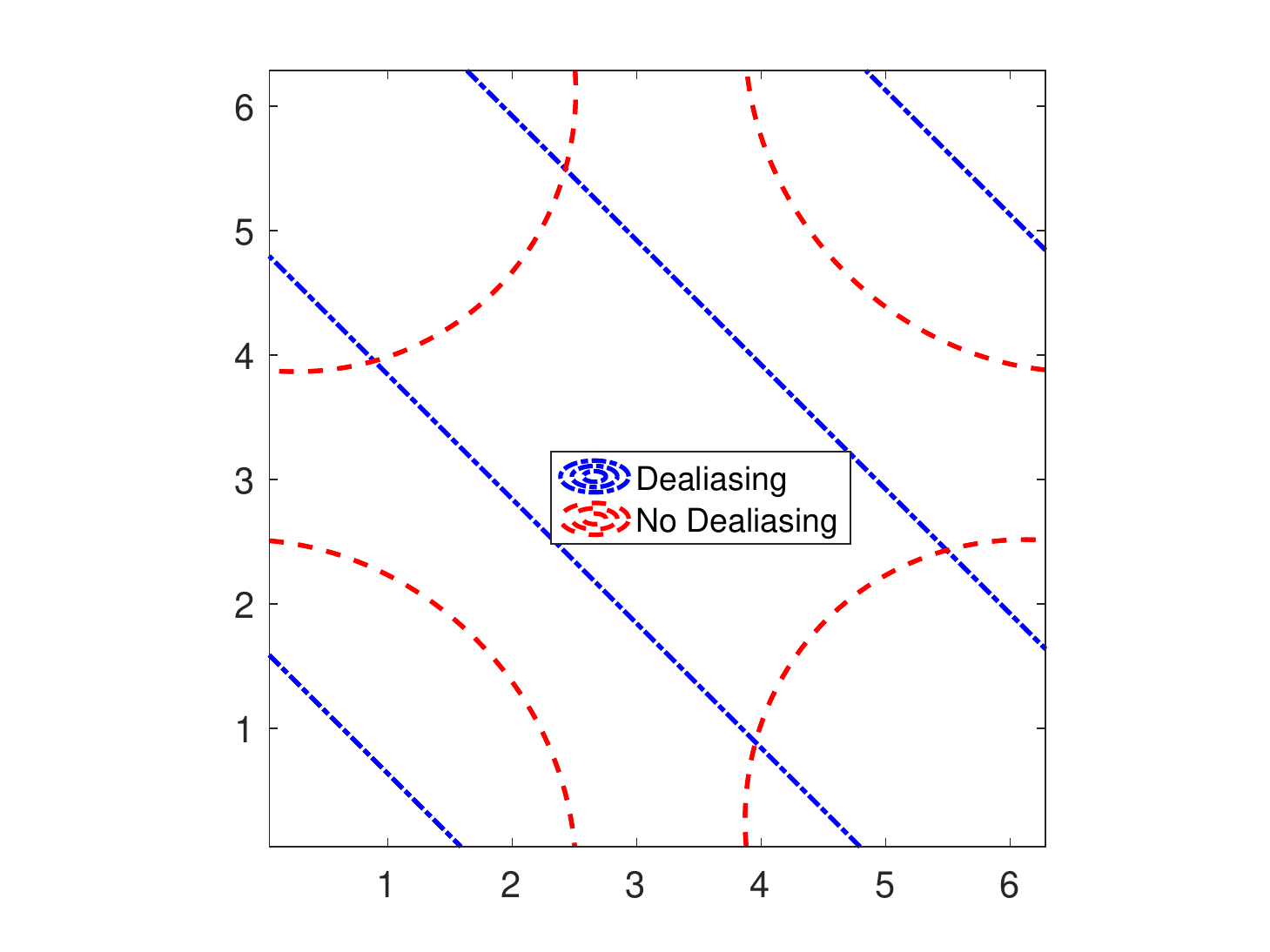}}
{\includegraphics[width=7.5cm]{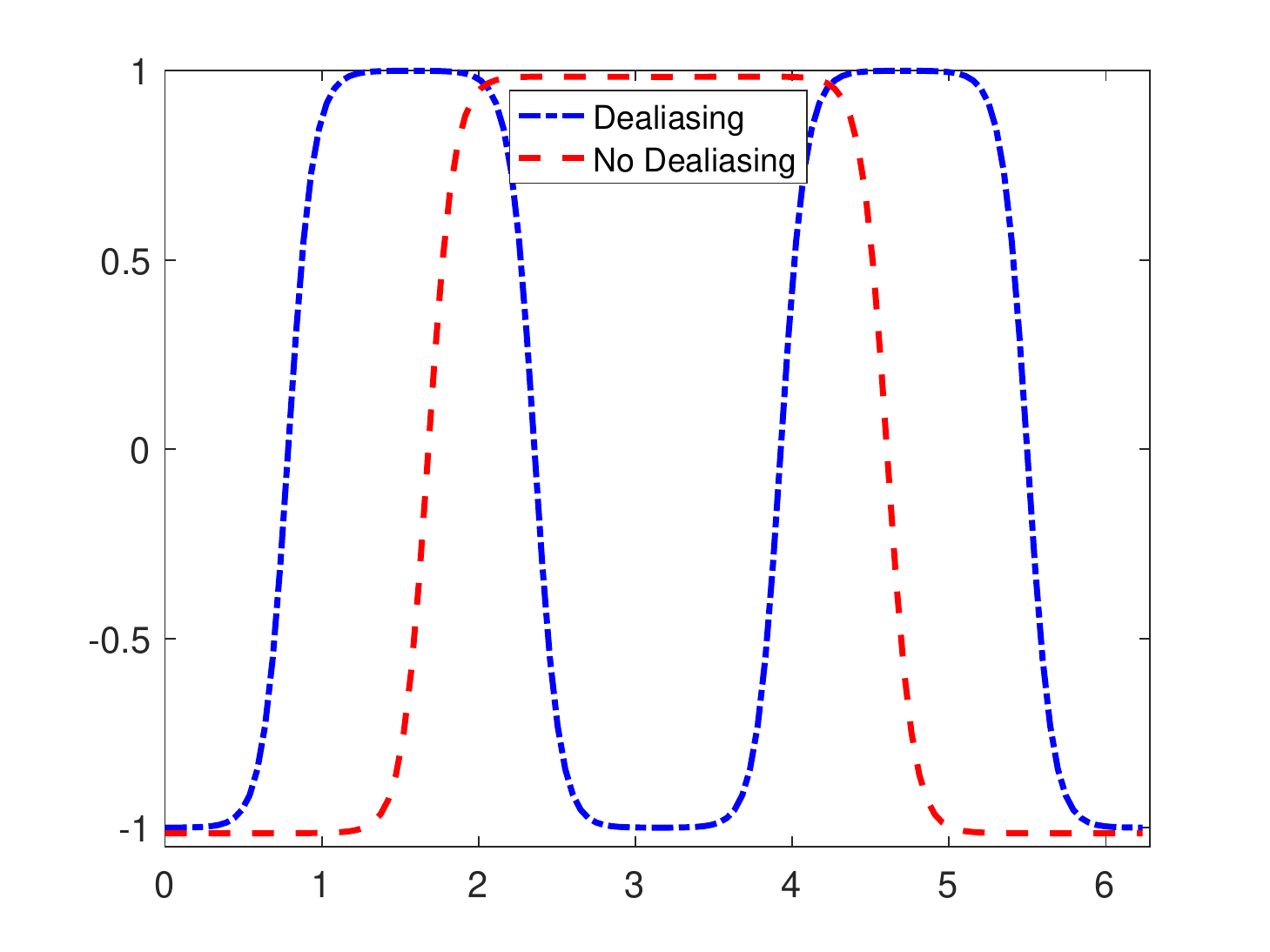}}

{\includegraphics[width=7.5cm]{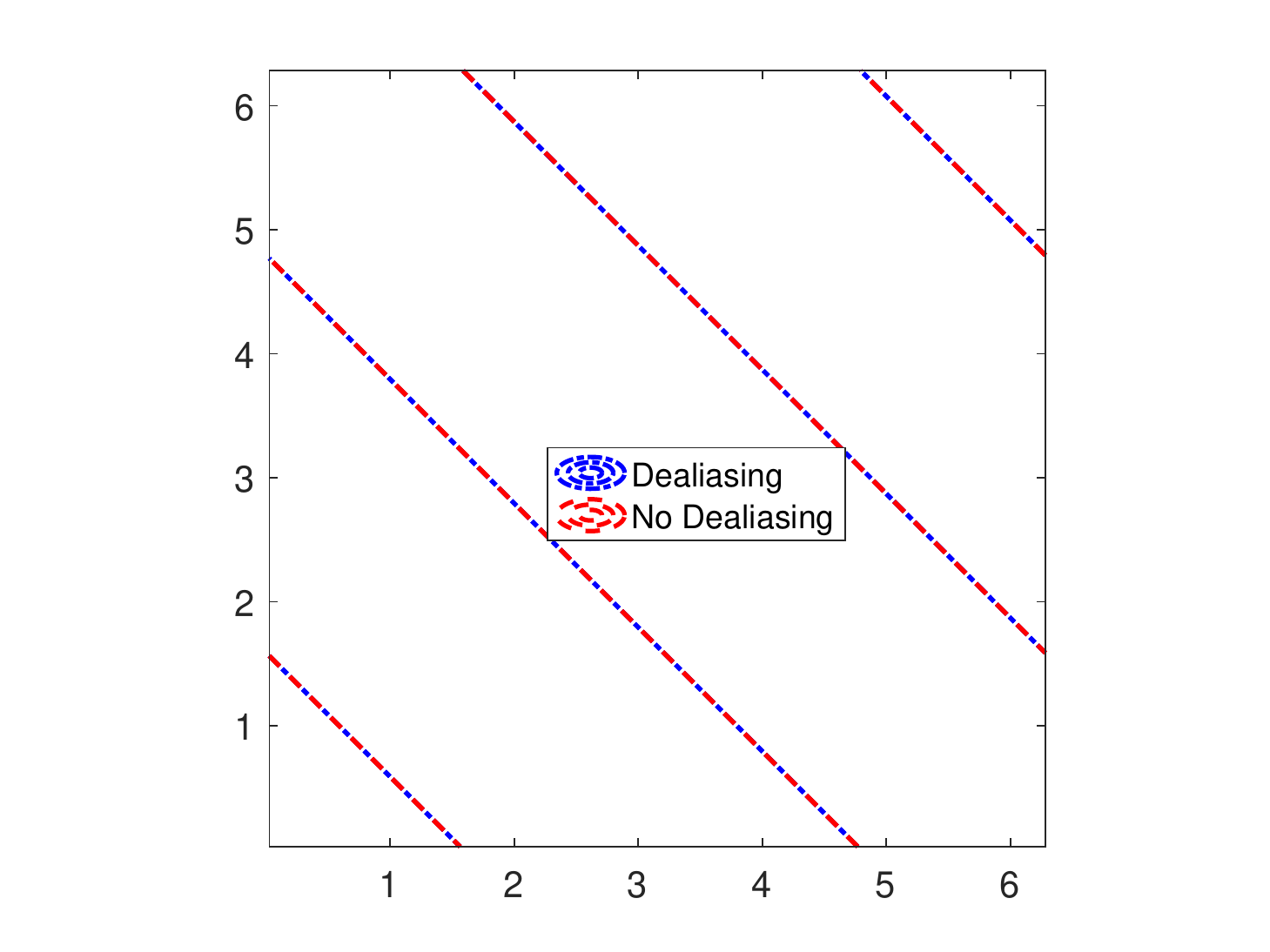}}
{\includegraphics[width=7.5cm]{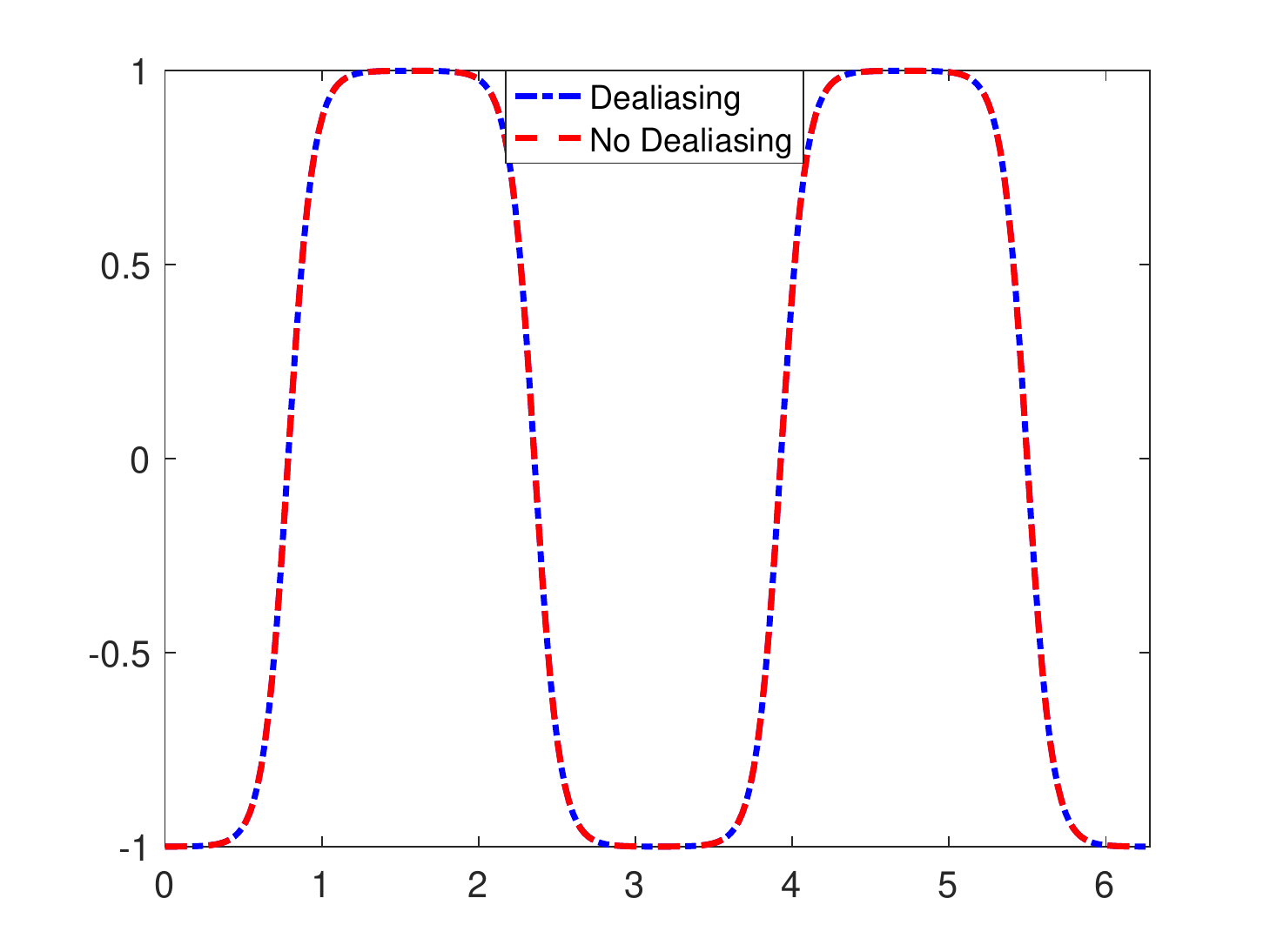}}
\caption{Example \ref{Exp6.2.0}.
Left: contour lines of  $u$ with the value of $-0.1$;
right: cut lines of the numerical solutions along $y=x$, $x\in[0,2\pi]$.  Top: $N=128$; bottom: $N=256$.}
\label{fig6.4}
\end{figure}

\begin{figure}
\includegraphics[width=16cm,height=9cm]{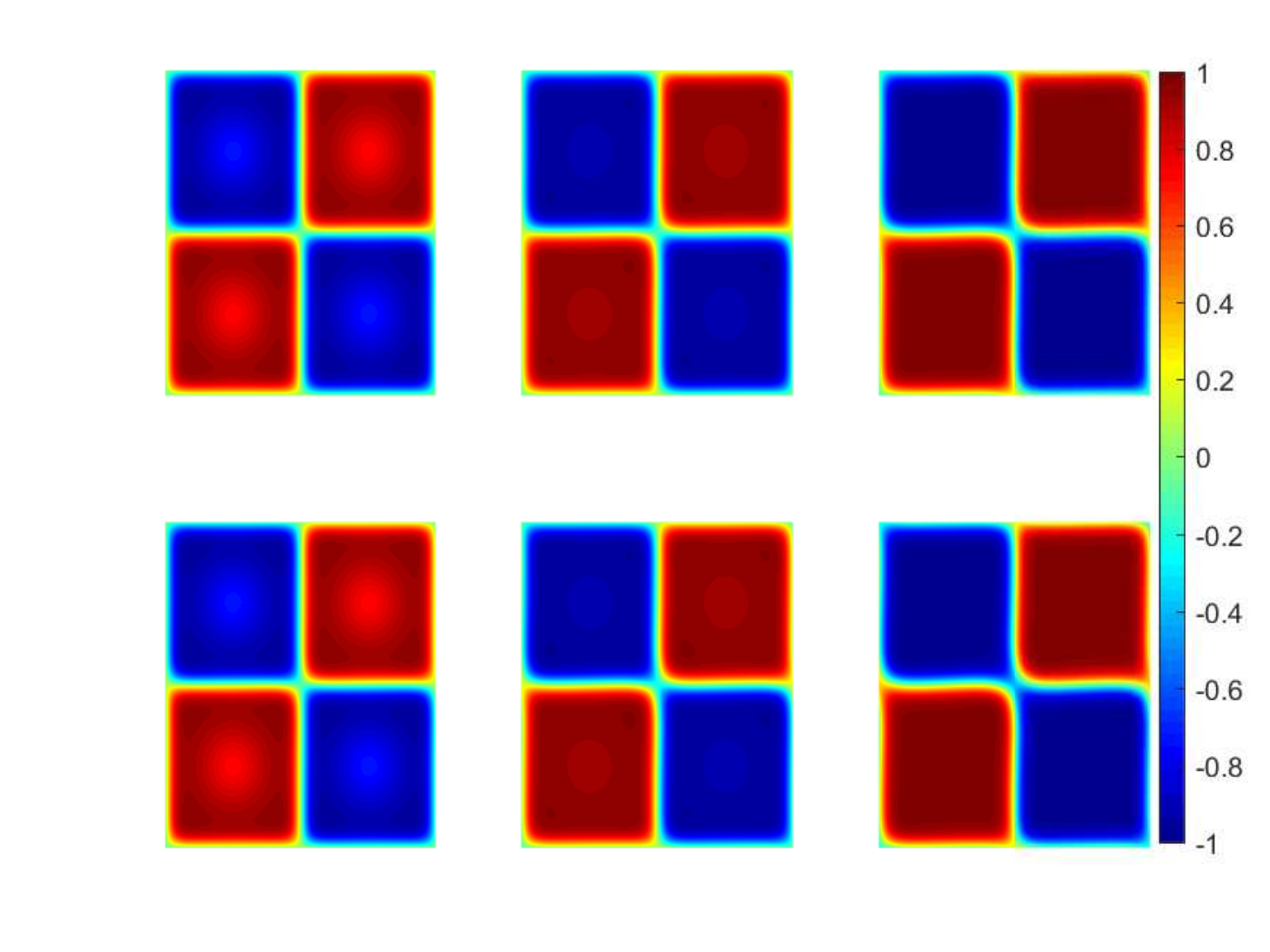}
\caption{Example \ref{Exp6.2.0}.   Snapshots of the numerical solutions at $t = 7.5, 8$, and  $8.5$ derived by {\tt SAV-M(3)} with (Top) and without the de-aliasing (Bottom).}
\label{fig6.5}
\end{figure}

\begin{figure}
\centering
{\includegraphics[width=7.5cm]{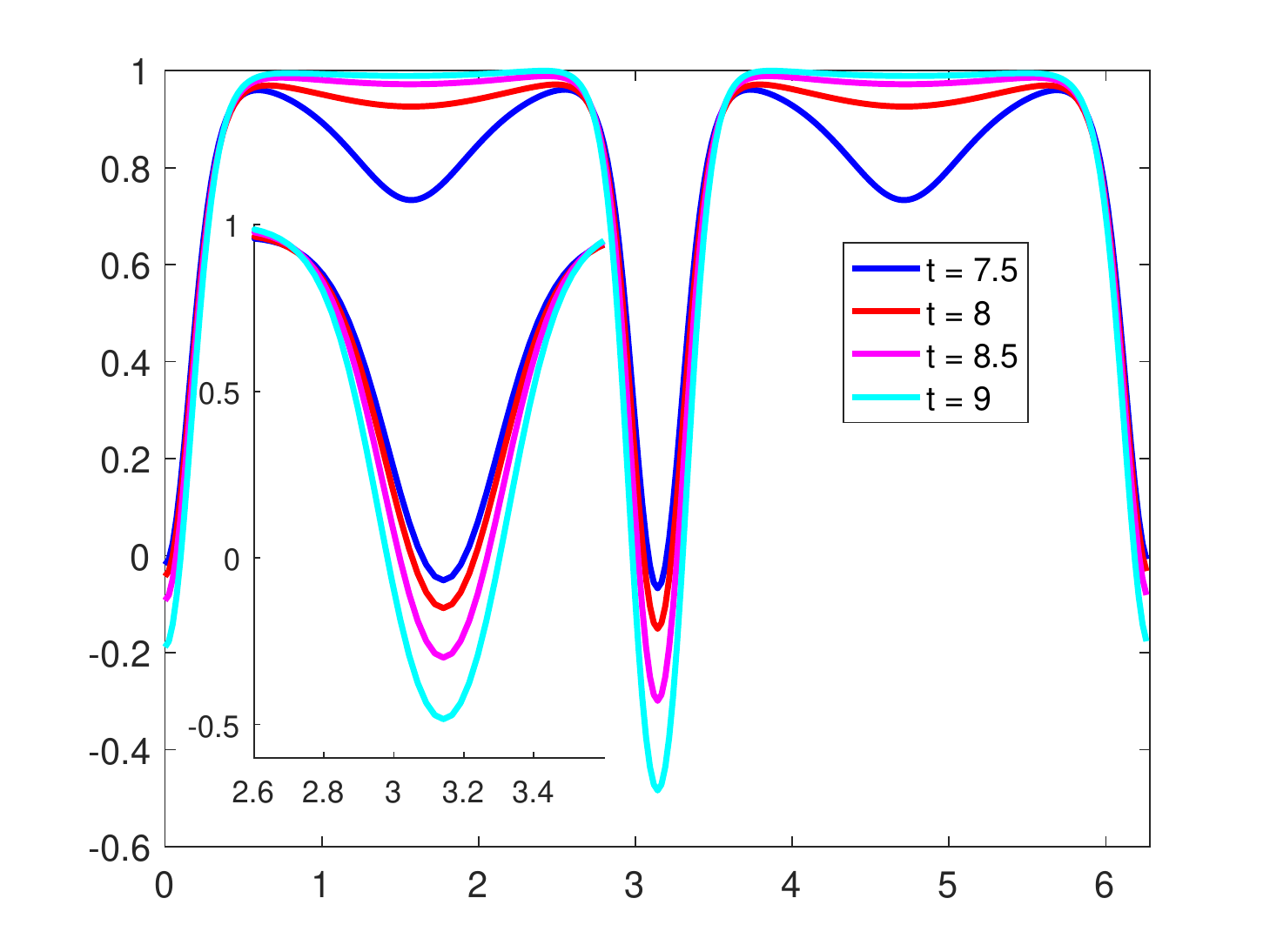}}
{\includegraphics[width=7.5cm]{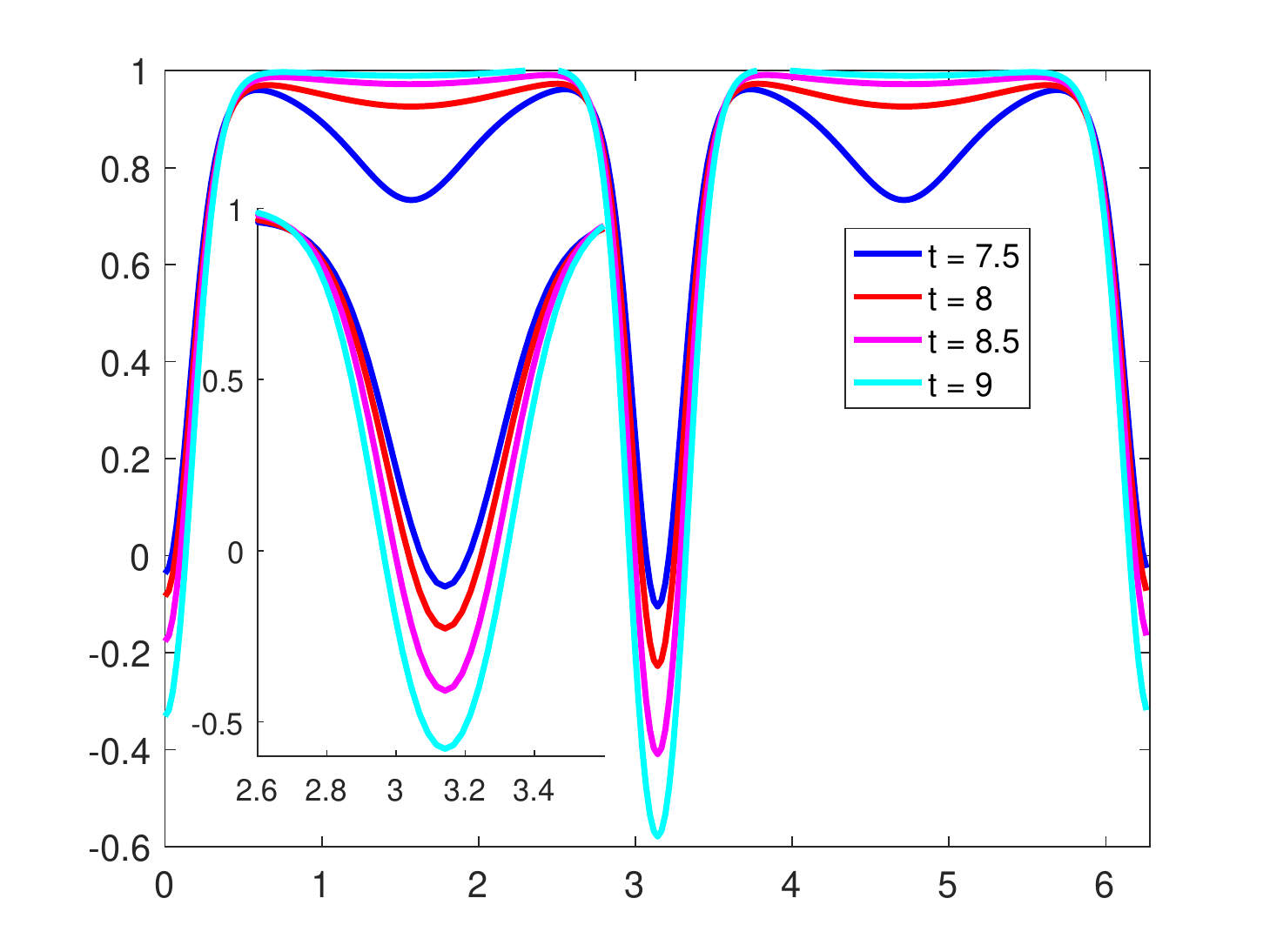}}
\caption{Example \ref{Exp6.2.0}.
Cut lines of the numerical solutions along $y=2\pi$, derived by {\tt SAV-M(3)} with (Left) and  without the de-aliasing  (Right).}
\label{fig6.6}
\end{figure}

\begin{example}   \label{Exp6.2.1}
This example is used to validate the modified-energy {stability}  and to check the original-energy stability of {\tt SAV-M(1)}$\sim${\tt SAV-M(4)} and {\tt G-SAV-M(1)}$\sim${\tt G-SAV-M(4)} for \eqref{6.2.1}. The domain $\Omega = (0,2\pi)\times (0,2\pi)$ is uniformly partitioned with  $N = 128$, the parameter $\epsilon$ is taken as $0.1$, and the initial value is chosen as  $u(x,y,0) = 0.1\times \mbox{\tt rand}(x,y) - 0.05$.

Figure \ref{fig5.5} presents the discrete total modified-energy curves of {\tt SAV-M(1)}$\sim${\tt SAV-M(4)} and {\tt G-SAV-M(1)}$\sim${\tt G-SAV-M(4)}. All those curves are  monotonically decreasing, and consistent with the theoretical results.
Figure \ref{fig5.6} plots the discrete total original-energy curves of {\tt SAV-M(1)}$\sim${\tt SAV-M(4)} and {\tt G-SAV-M(1)}$\sim${\tt G-SAV-M(4)}. The result shows that those schemes can ensure the original-energy decay only if a  suitable time stepsize is taken.
Figure \ref{fig5.7} presents the numerical solutions at $t = 200$ derived by {\tt G-SAV-M(2)} with $\tau = 0.02$ and $0.01$.
It is shown that   with a large time stepsize, the solution is inaccurate and the original-energy is not monotonically decreasing as shown in Figure \ref{fig5.6}. Remark \ref{remk6.3} will provide
  a detailed discuss on the time stepsize constraints of {\tt SAV-M(1)}$\sim${\tt SAV-M(4)} and {\tt G-SAV-M(1)}$\sim${\tt G-SAV-M(4)}  for the Cahn-Hilliard model \eqref{6.2.1}.
\end{example}

\begin{figure}
\centering
{\includegraphics[width=7.5cm]{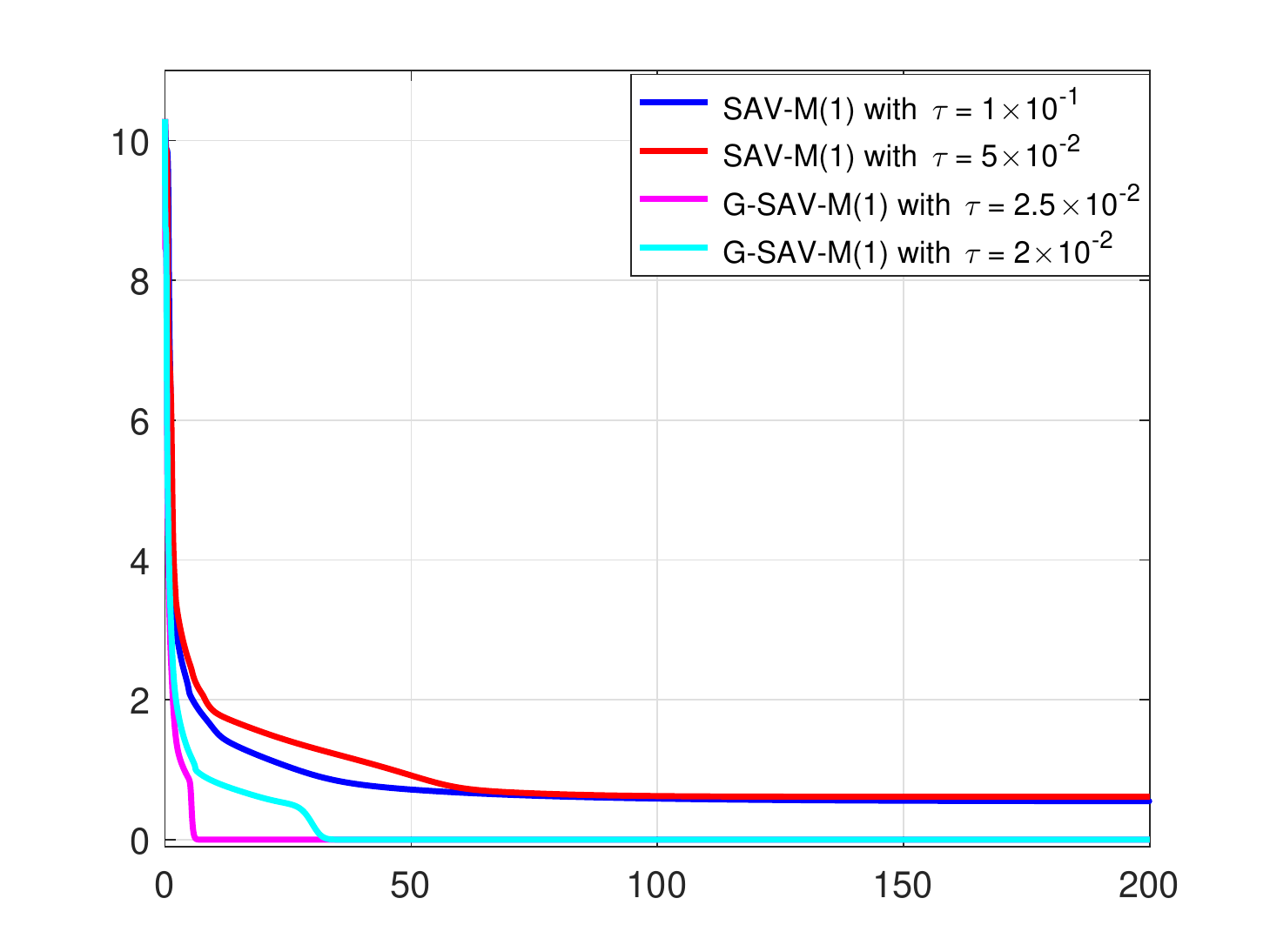} }
{\includegraphics[width=7.5cm]{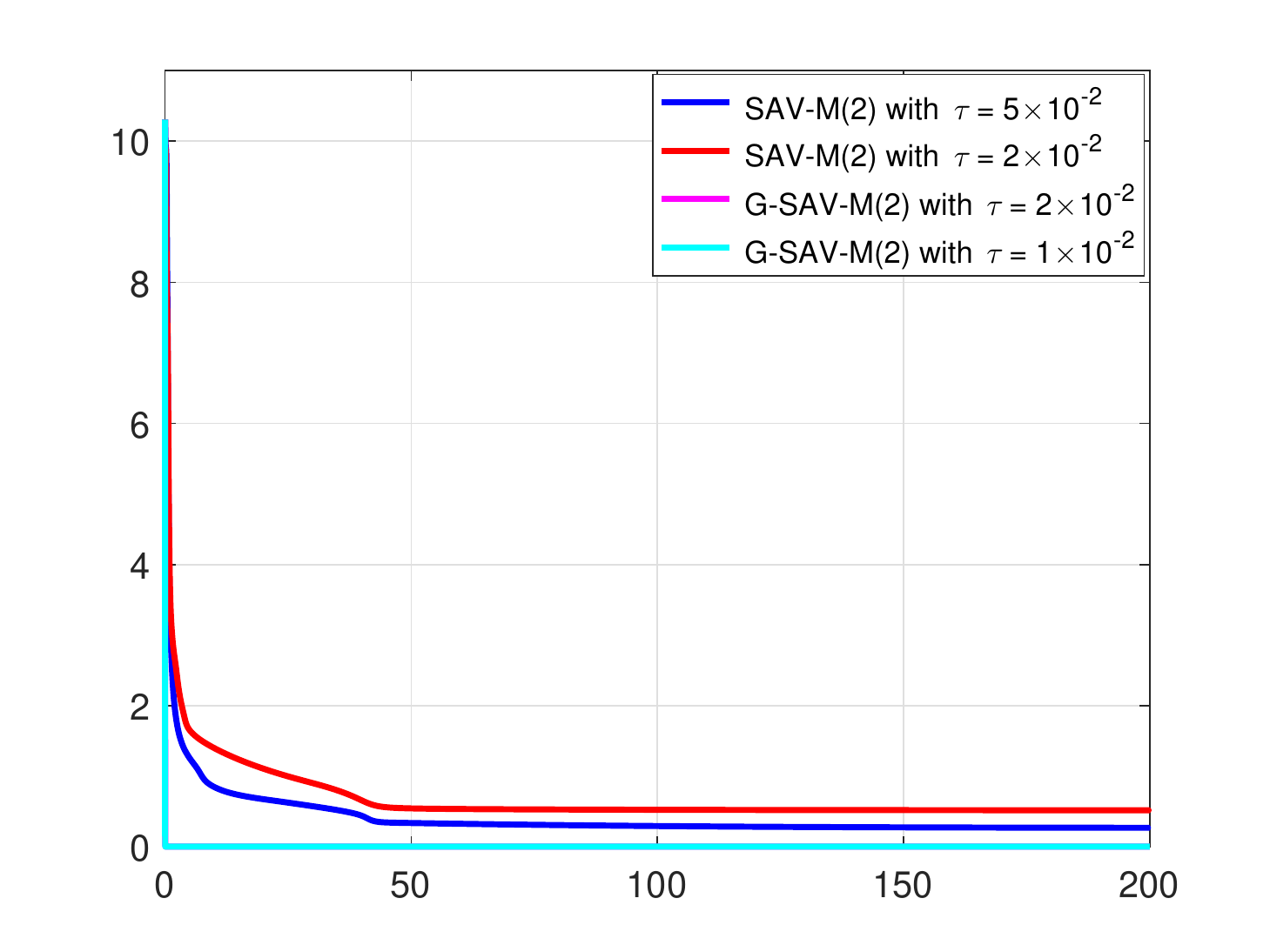} }

{\includegraphics[width=7.5cm]{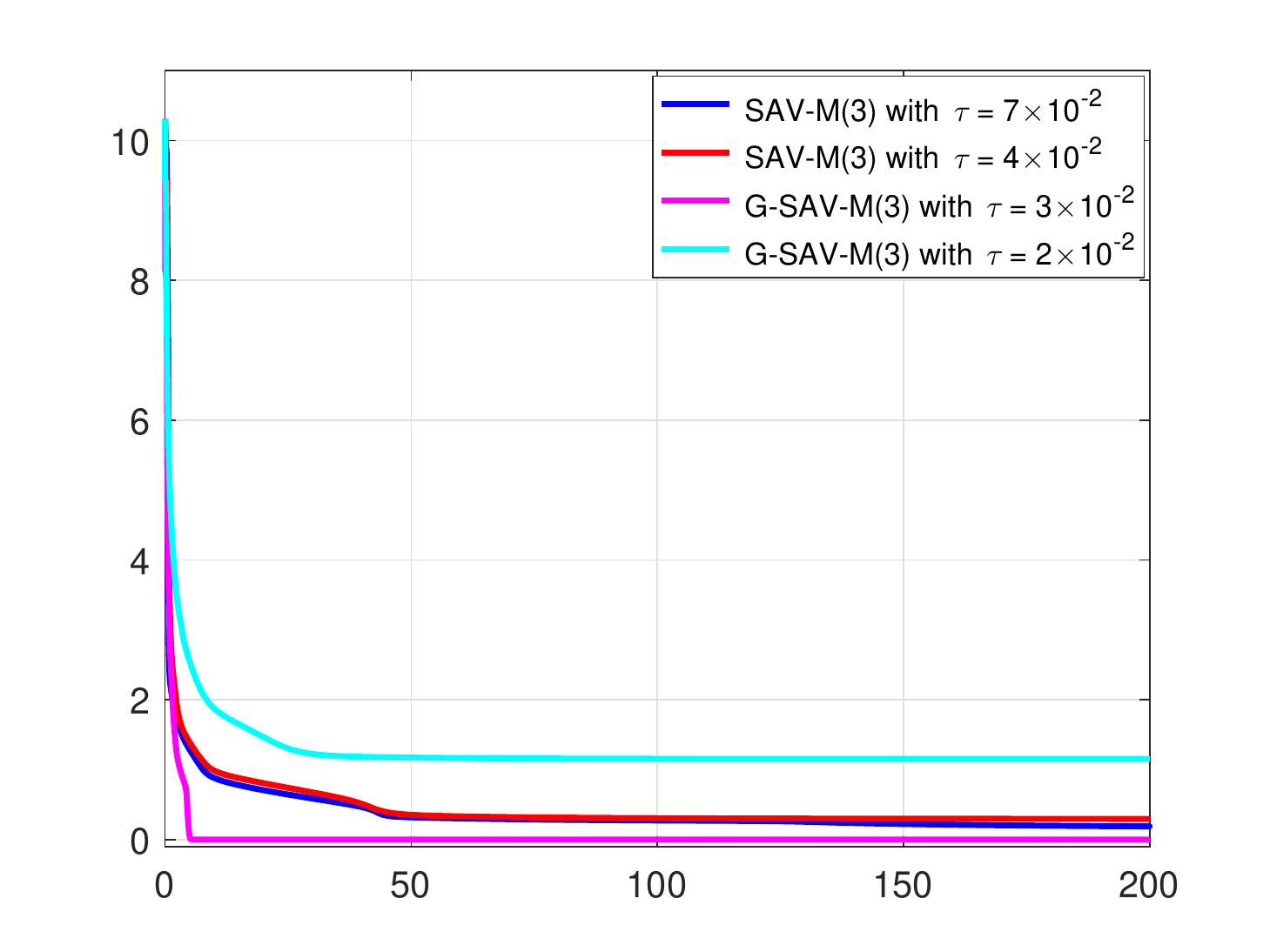} }
{\includegraphics[width=7.5cm]{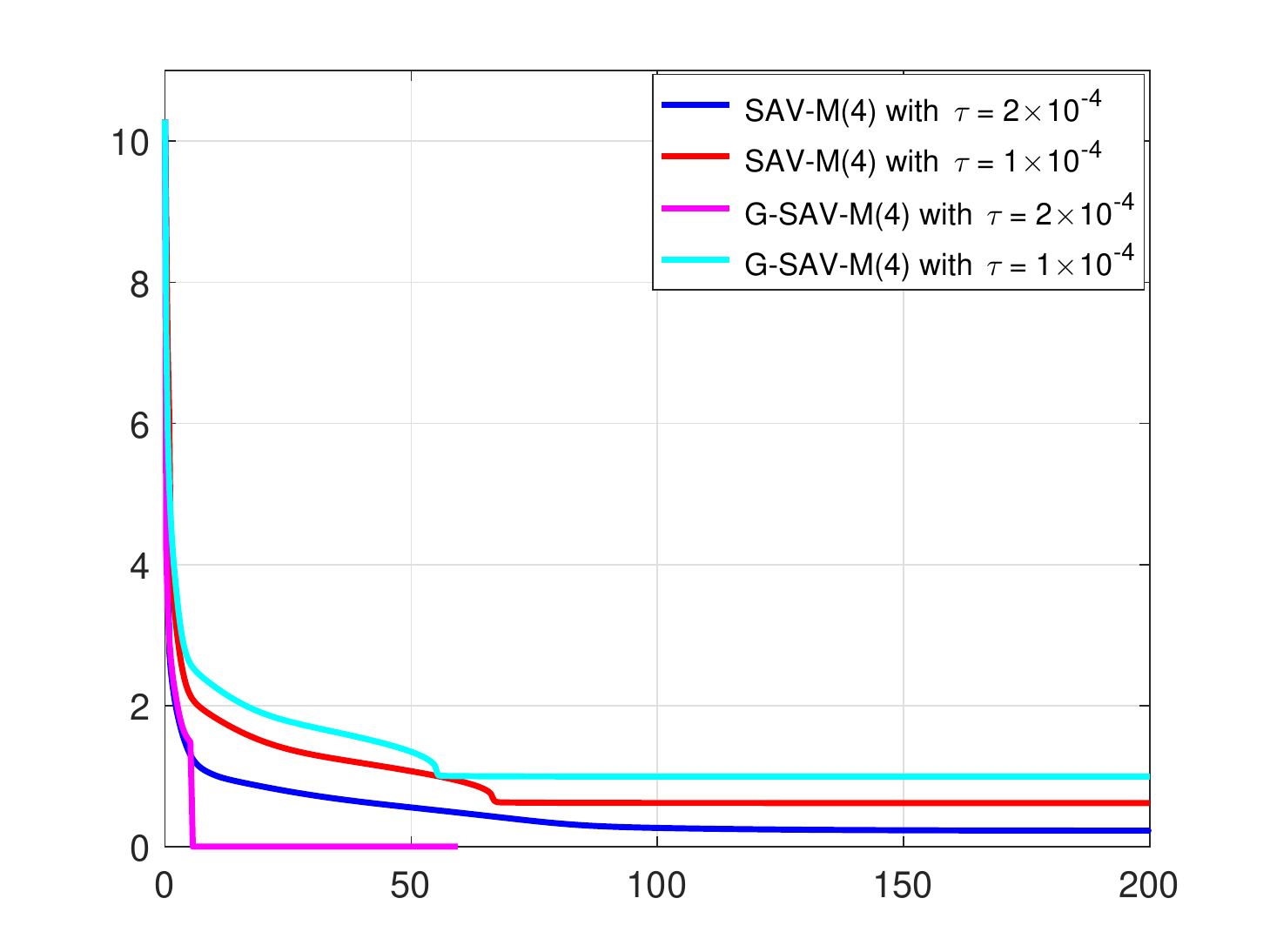} }
\caption{Example \ref{Exp6.2.1}. The time evolution of the  discrete total modified-energies of {\tt SAV-M(1)}$\sim${\tt SAV-M(4)} and {\tt G-SAV-M(1)}$\sim${\tt G-SAV-M(4)} for the Cahn-Hilliard model \eqref{6.2.1}.}
\label{fig5.5}
\end{figure}

\begin{figure}
\centering
{\includegraphics[width=7.5cm]{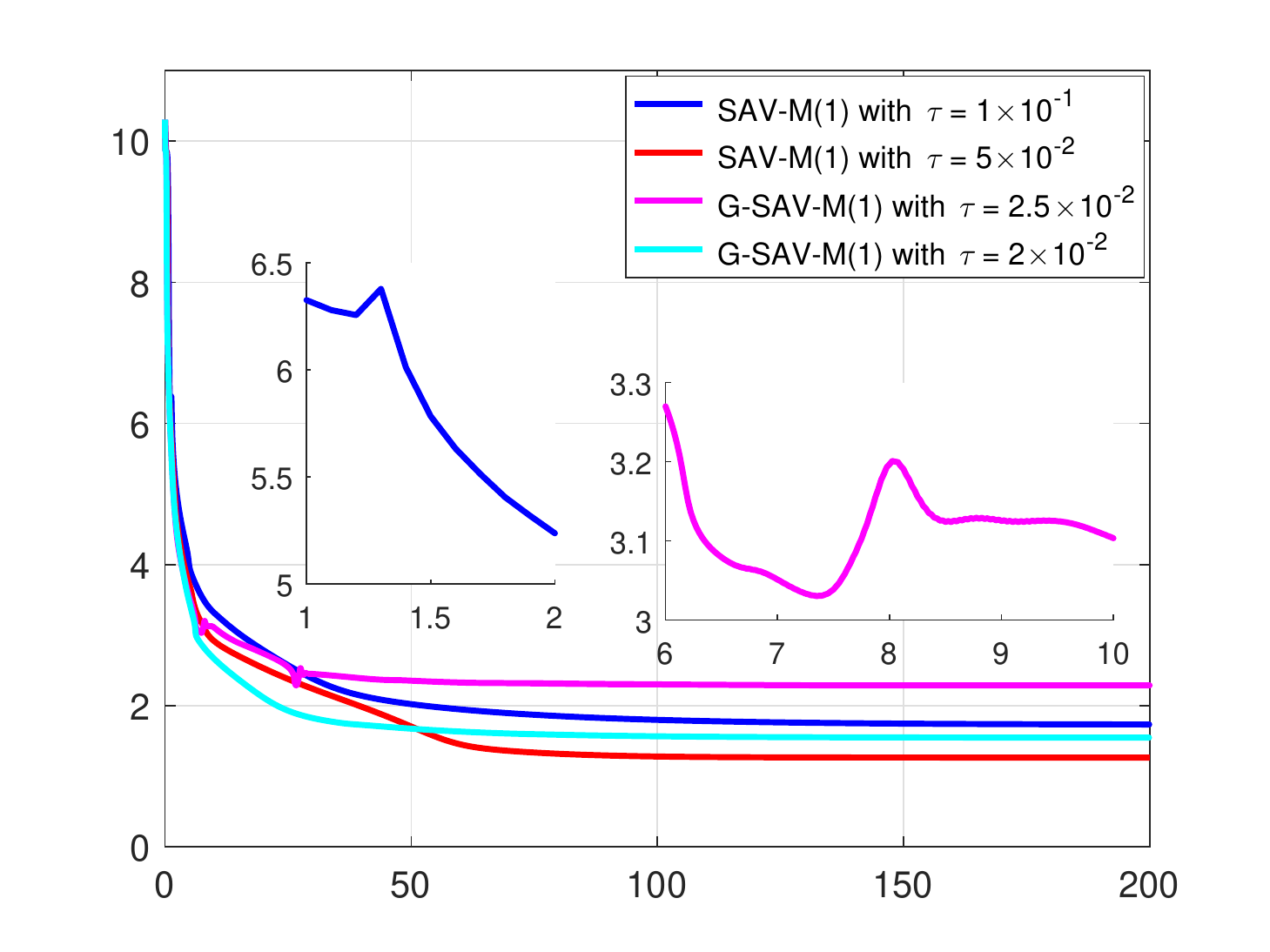} }
{\includegraphics[width=7.5cm]{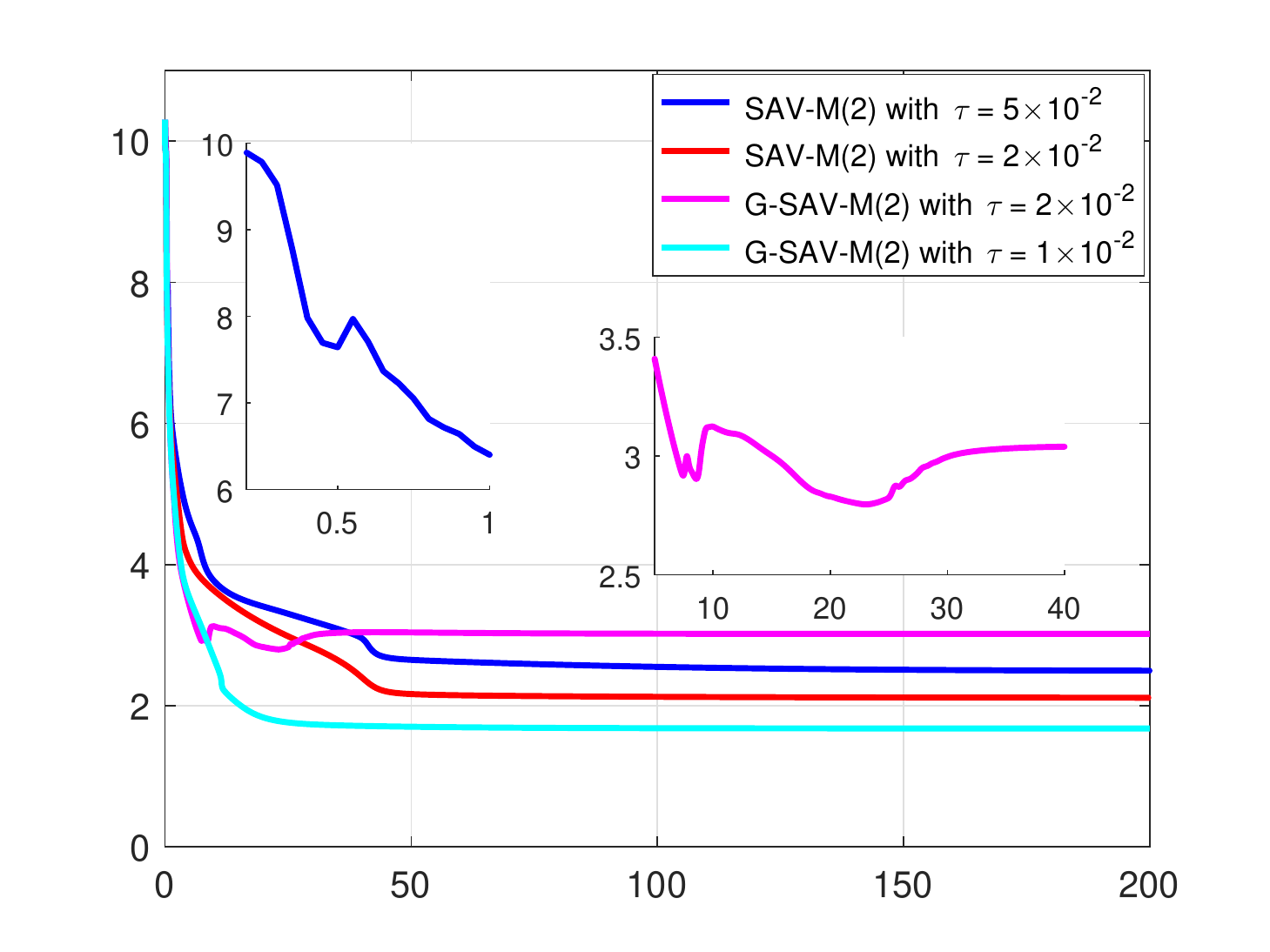} }

{\includegraphics[width=7.5cm]{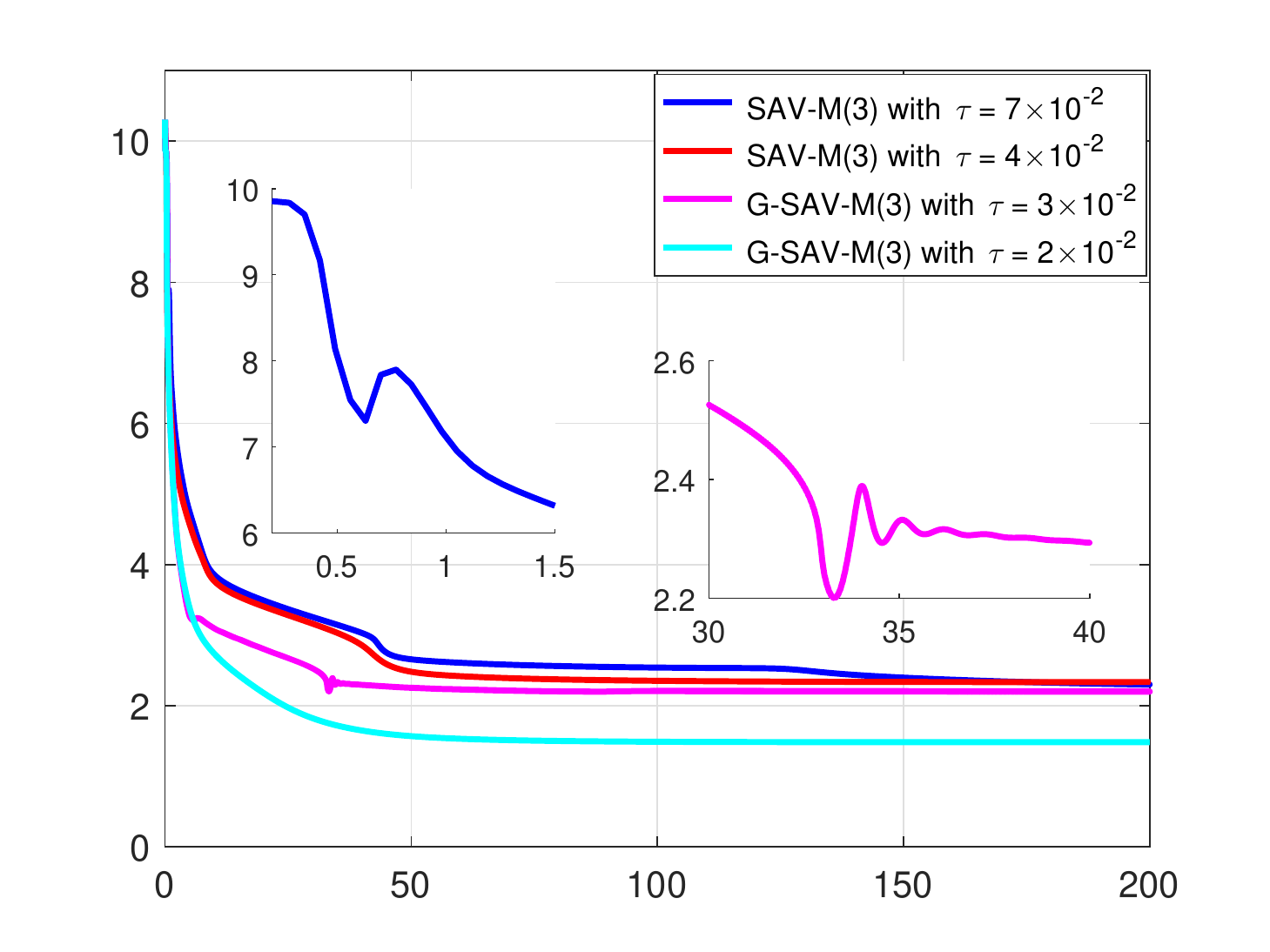} }
{\includegraphics[width=7.5cm]{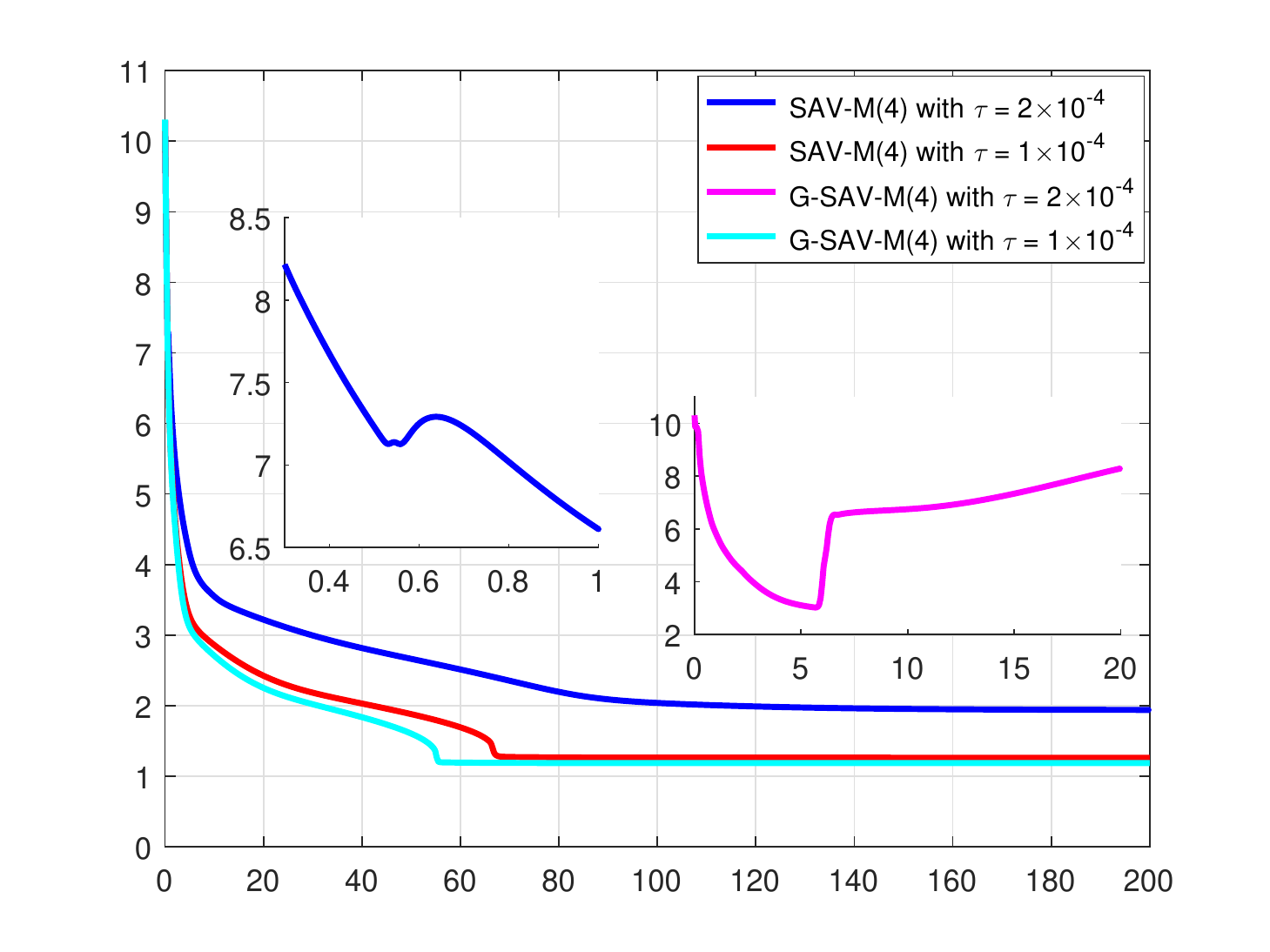} }
\caption{Same as Figure \ref{fig5.5}, except for the  original-energy.}
\label{fig5.6}
\end{figure}

\begin{figure}
\begin{minipage}{0.48\linewidth}
  \centerline{\includegraphics[width=7cm,height=5cm]{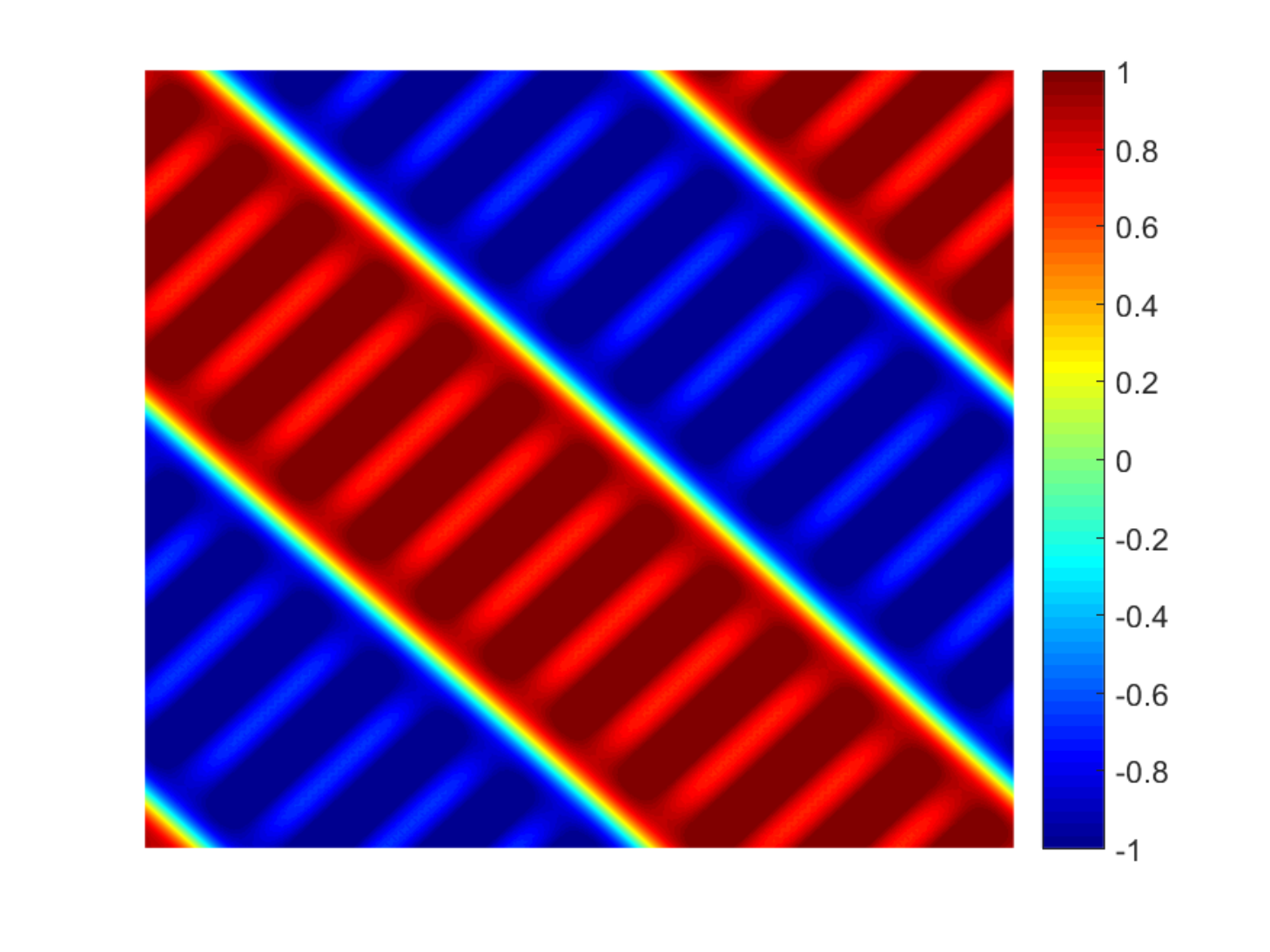}}
\end{minipage}
\hfill
\begin{minipage}{0.48\linewidth}
  \centerline{\includegraphics[width=7cm,height=5cm]{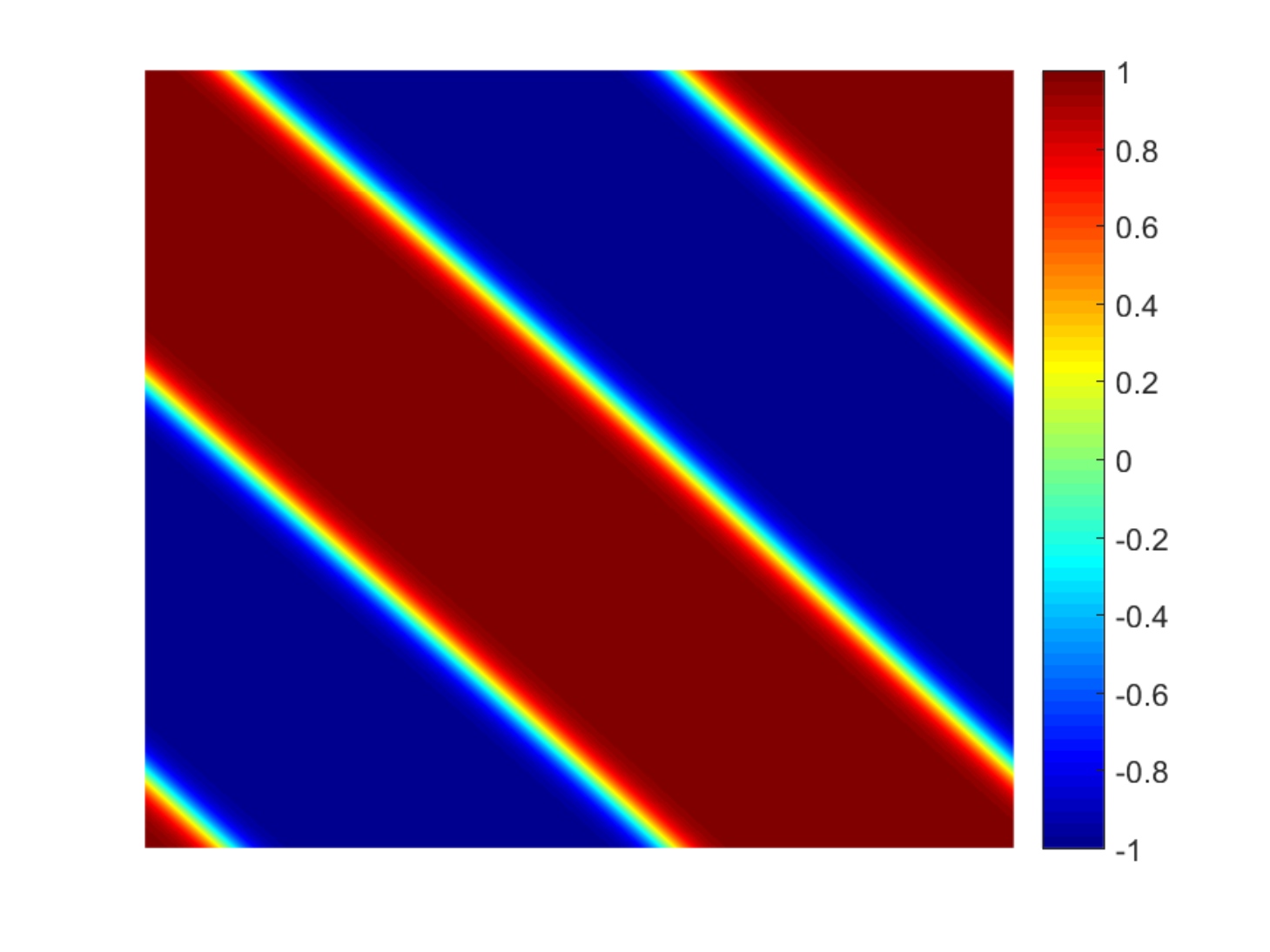}}
\end{minipage}
\caption{Example \ref{Exp6.2.1}. Numerical solutions at $t=200$ derived by {\tt G-SAV-M(2)} with $\tau = 0.02$ (Left) and $0.01$ (Right), respectively.}
\label{fig5.7}
\end{figure}

\begin{remark}   \label{remk6.3}
Applying the Fourier pseudo-spectral method to the Cahn-Hilliard model \eqref{6.2.1} yields the  ODE system
\begin{align}   \label{6.2.2}
\frac{d \hat{u}_{k,l}}{d t } = - \epsilon^2\left(k^2 + l^2 \right)^2 \hat{u}_{k,l} - (k^2+l^2)\left[ \hat{w}_{k,l} - \hat{u}_{k,l}\right],~~~~ (k,l) \in \widehat{\mathbb{S}}_N,
\end{align}
where $\{\hat{w}_{k,l}\}$ are the discrete Fourier coefficients of the cubic term $u^3$ and given by \eqref{6.1.5}.
Similarly, 
\eqref{6.2.2} can also be viewed as the test equation \eqref{8.3.1} with
\[ \xi = - \epsilon^2\left(k^2 + l^2 \right)^2,~~~ \zeta = -(k^2 + l^2)\left[ \frac{3}{N^2} \sum\limits_{(i,j)\in\mathbb{S}_N} |u_{i,j}|^2 - 1\right],~~ (k,l) \in \widehat{\mathbb{S}}_N.  \]
For Example \ref{Exp6.2.1}, the curves  of $\psi^n = \frac{3}{N^2} \sum\limits_{(i,j)\in\mathbb{S}_N} |u_{i,j}|^2$ plotted in Figure \ref{fig5.2.4} show that $\psi^n \lesssim 2.6$.
Thus, one can take $\zeta \approx -1.6(k^2 + l^2), (k,l) \in \widehat{\mathbb{S}}_N$  and then  use \ref{Appx3} to estimate the time stepsizes for {\tt SAV-M(1)}$\sim${\tt SAV-M(4)} and {\tt G-SAV-M(1)}$\sim${\tt G-SAV-M(4)}.
Specifically,  when the parameters $(\alpha_0,\beta_0,\beta_2) = (0,0,1)$, the time stepsize satisfies
$$\tau < \left\{ \frac{2}{\max(\xi-3\zeta)}: \zeta < \frac{1}{3}\xi\right\},$$
a simple calculation gives $\max\left\{\xi\!-\!3\zeta \right\} \!= \! \max\left\{- 0.01\left(k^2 \!+ \!l^2 \right)^2 \!+ \!4.8\left(k^2 \!+ \!l^2 \right), (k,l) \in \widehat{\mathbb{S}}_N\right\} \!= \! 575.99~(k\!=\!4, l \!= \!15)$ so that $\tau \lesssim 3.47\times10^{-3}$;
when  $(\alpha_0,\beta_0,\beta_2) = (-1/3,5/12,3/4)$,
$$ \tau < \left\{ \frac{4}{3\max(\xi-2\zeta)}: \zeta < \frac{1}{2}\xi\right\},$$
which is combined with the result $\max\left\{\xi\!-\!2\zeta \right\} \!=\! \max\!\left\{-0.01\left(k^2 \!+ \!l^2 \right)^2 \!+ \!3.2\left(k^2 \!+ \!l^2 \right), (k,l) \in \widehat{\mathbb{S}}_N\right\}\! = 256~(k\!=\!4, l \!= \!12)$  to yield $\tau \lesssim 5.2\times10^{-3}$;
when $(\alpha_0,\beta_0,\beta_2) = (1/3,0,2/3)$,
  $$ \tau < \left\{ \frac{4}{\max(\xi-3\zeta)}: \zeta < \frac{1}{3}\xi\right\},$$
   which gives $\tau\lesssim 7.0\times10^{-3}$;
and  when $(\alpha_0,\beta_0,\beta_2) = (1/3,-1/6,1/2)$,
$$\tau < -\frac{4}{3\min(\zeta)}, $$
 which gives $\tau\lesssim 1.02\times 10^{-4}$.
Note that 
those time stepsize estimates  for {\tt SAV-M(1)}$\sim${\tt SAV-M(4)} and {\tt G-SAV-M(1)}$\sim${\tt G-SAV-M(4)} are sufficient, and  when $(\alpha_0,\beta_0,\beta_2) = (1/3,-1/6,1/2)$, the GLTD is not algebraically stable and the time stepsize
 is constrained much severely for the Cahn-Hilliard model \eqref{6.2.1}.
\end{remark}

\begin{figure}
\begin{minipage}{0.48\linewidth}
  \centerline{\includegraphics[width=7cm,height=5cm]{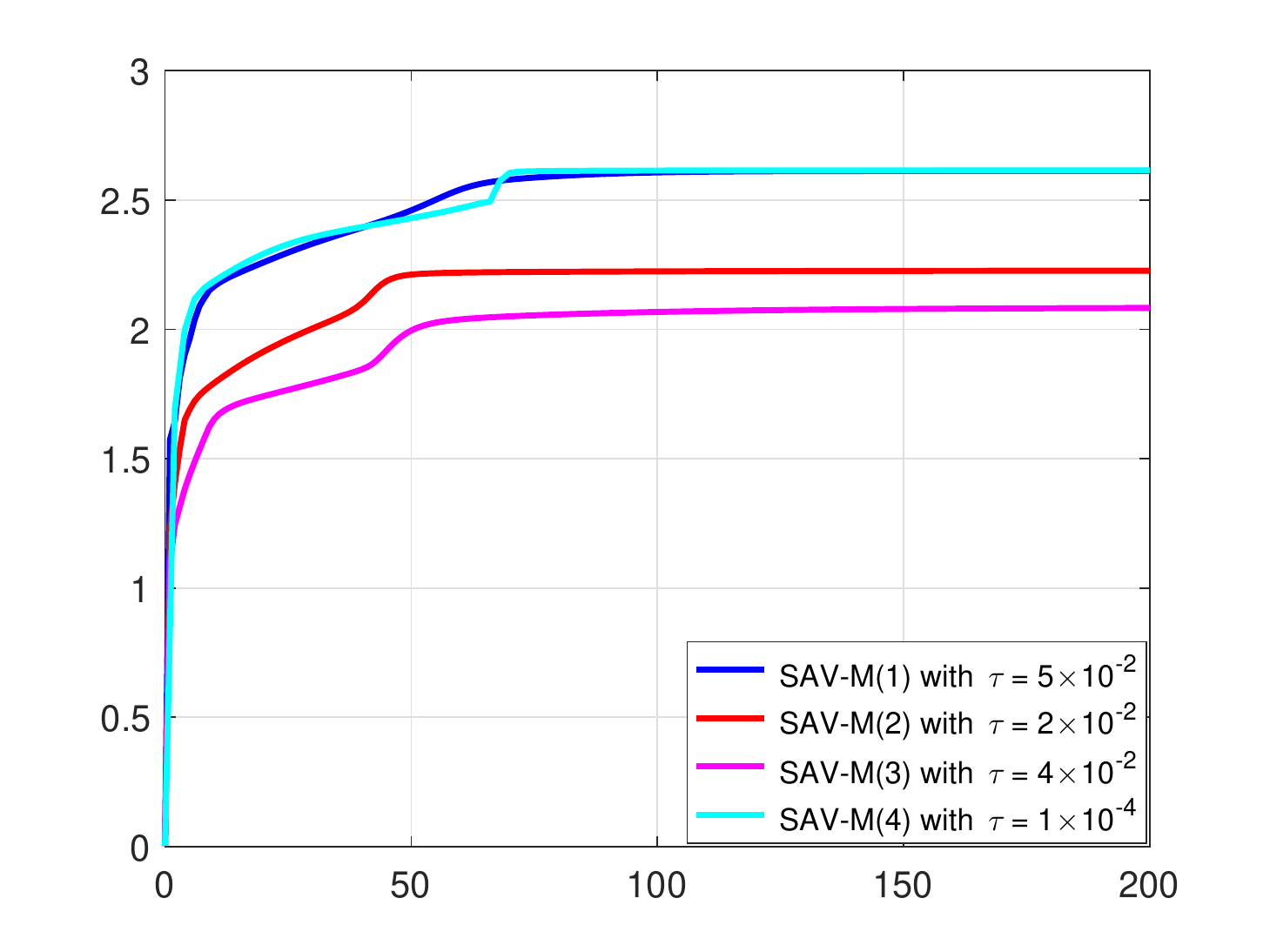}}
\end{minipage}
\hfill
\begin{minipage}{0.48\linewidth}
  \centerline{\includegraphics[width=7cm,height=5cm]{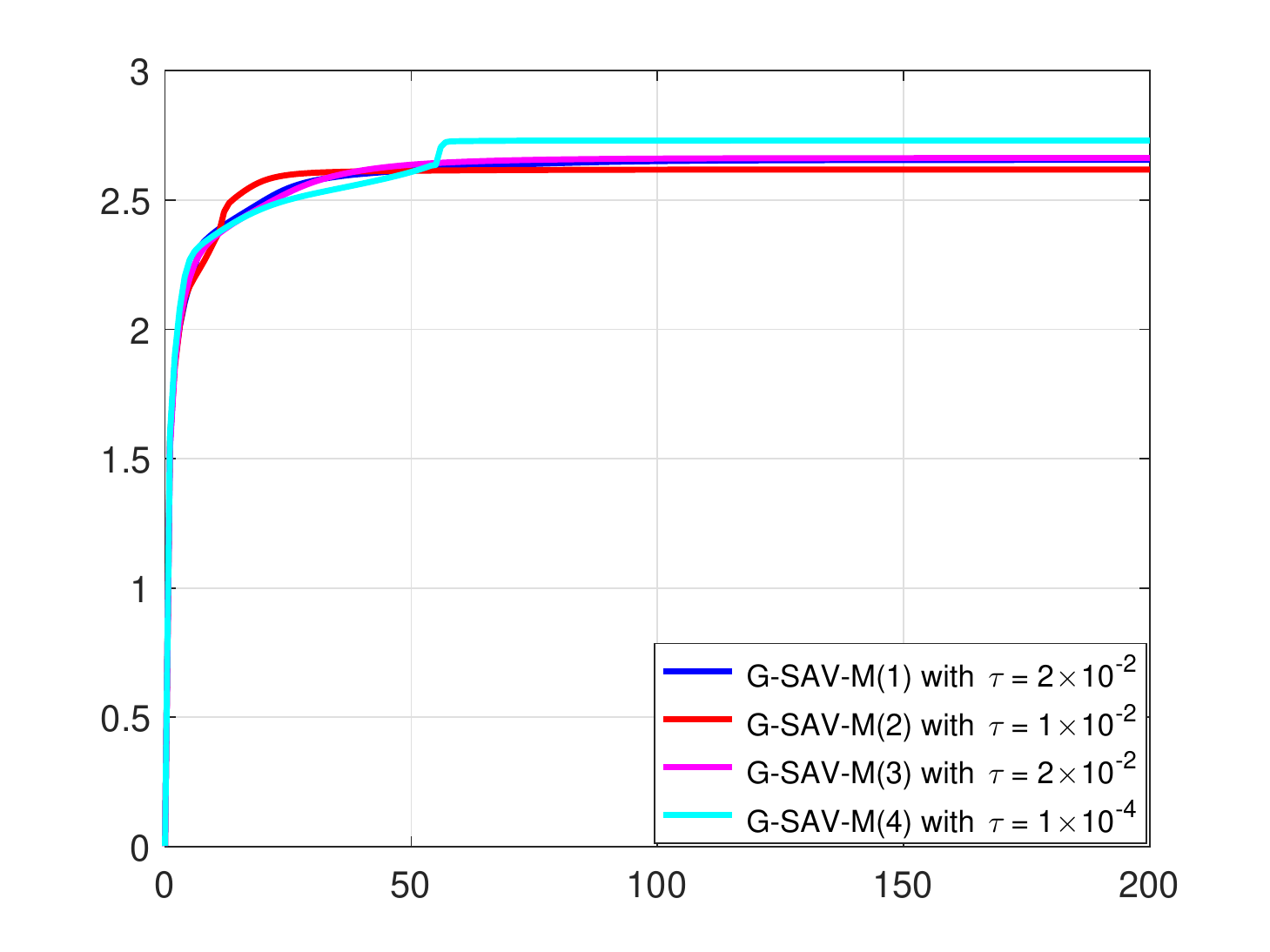}}
\end{minipage}
\caption{Example \ref{Exp6.2.1}.  $\psi^n$ derived by {\tt SAV-M(1)}$\sim${\tt SAV-M(4)} (Left) and {\tt G-SAV-M(1)}$\sim${\tt G-SAV-M(4)}.}
\label{fig5.2.4}
\end{figure}

\subsection{Phase field crystal model}

The phase field crystal  model
\begin{align}   \label{6.3.1}
\frac{\partial u}{\partial t} = \Delta\mu,~~~ \mu = u^3 + (1- \epsilon) u + 2\Delta u + \Delta^2 u, ~~~\vec{x}\in \Omega,~t>0,
\end{align}
can be derived from the $H^{-1}$ gradient flow of the   free energy
\begin{align*}
\mathcal{E}(u) = \int_{\Omega} \left[ \frac{1}{4}u^4  + \frac{1-\epsilon}{2} u^2 - |\nabla u|^2 + \frac{1}{2}(\Delta u)^2\right] dx.
\end{align*}
Such model may be used to describe many crystal phenomena such as edge dislocations \cite{Berry06}, fcc ordering \cite{WuK10}, epitaxial growth and zone refinement \cite{Elder02}, and is a sixth-order nonlinear partial differential equation.

In order to apply {\tt SAV-M(1)}$\sim${\tt SAV-M(4)} and {\tt G-SAV-M(1)}$\sim${\tt G-SAV-M(4)} to \eqref{6.3.1} successfully,  the operators $\mathcal{L}$, $\mathcal{G}$ and the energy $\mathcal{E}_1(u)$  are chosen as
\[  \mathcal{L} =  \Delta^2, ~~ \mathcal{G} = \Delta,~~ \mathcal{E}_1(u) = \int_{\Omega} \left[ \frac{1}{4}u^4 + \frac{1\!-\!\epsilon}{2} u^2 - |\nabla u|^2 \right] dx.  \]

\begin{example} \label{Exp6.3.0}
This example  applies {\tt SAV-M(3)} with or without  the de-aliasing by zero-padding to the phase field crystal model \eqref{6.3.1}.
The parameter $\epsilon$ is taken as $0.25$,
the domain $\Omega = (0,100)\times (0,100)$ is partitioned with  $N = 200$ or $400$, and the initial data is chosen as $u(x,y,0) = 0.5\sin\left(\frac{\pi x}{50}\right)\sin\left(\frac{\pi y}{50}\right)$.

Figure \ref{fig6.7} shows the contour lines and cut lines of the numerical solutions at $t = 1000$ derived by {\tt SAV-M(3)} with or without de-aliasing. Some visible differences between those numerical solutions with $N = 200$ can be observed, but the differences  are indistinguishable for $N = 400$.
Figure \ref{fig6.8} gives the snapshots of the numerical solutions at $t=1000$ computed by {\tt SAV-M(3)} with the de-aliasing.  Figure \ref{fig6.9} presents the cut lines of numerical solutions at $t = 545$ and  $670$. The results show that the numerical solutions by {\tt SAV-M(3)} with or without the de-aliasing may have some differences at intermediate times.
\end{example}

\begin{figure}
\centering
{\includegraphics[width=7.5cm]{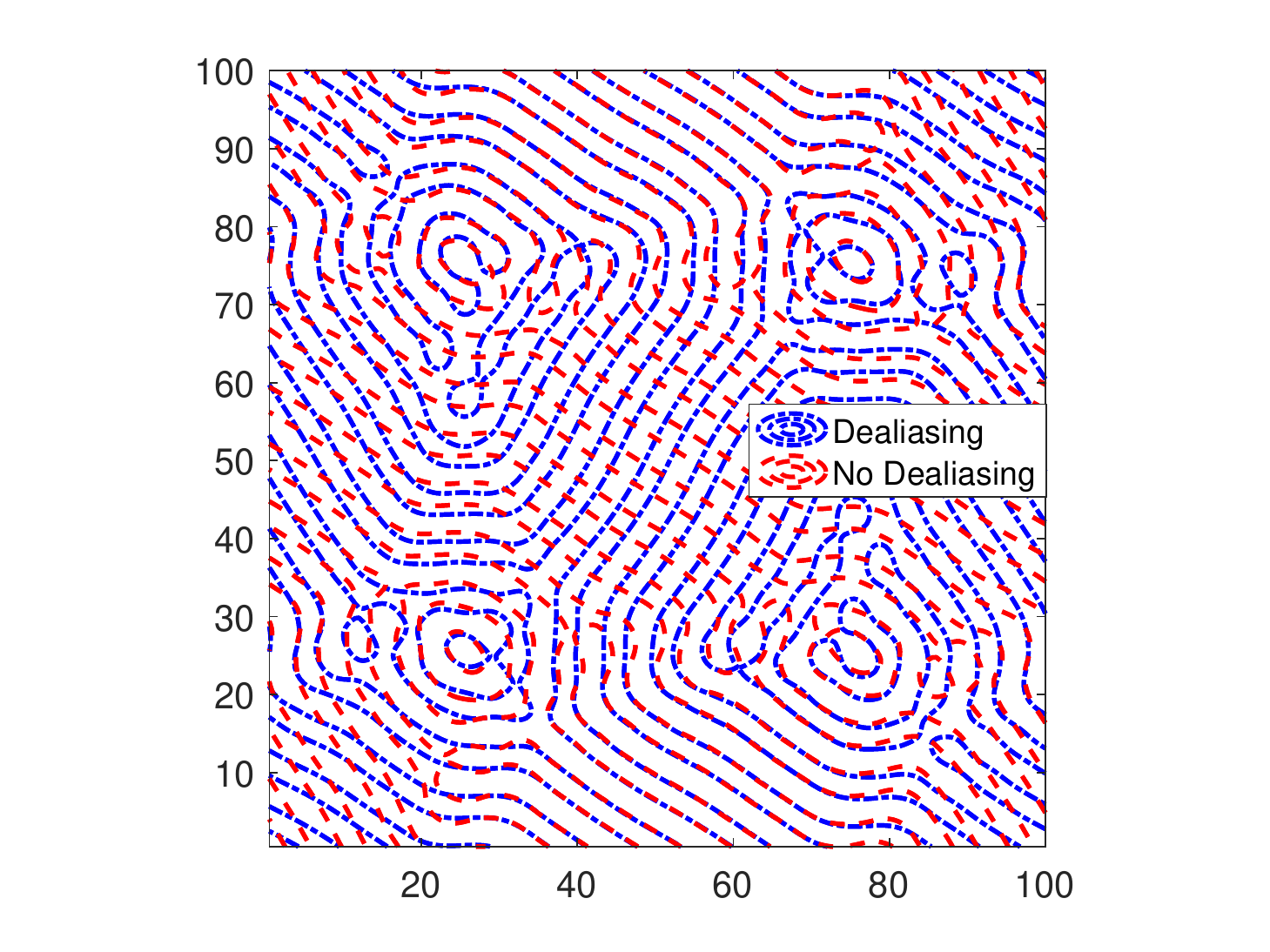}}
{\includegraphics[width=7.5cm]{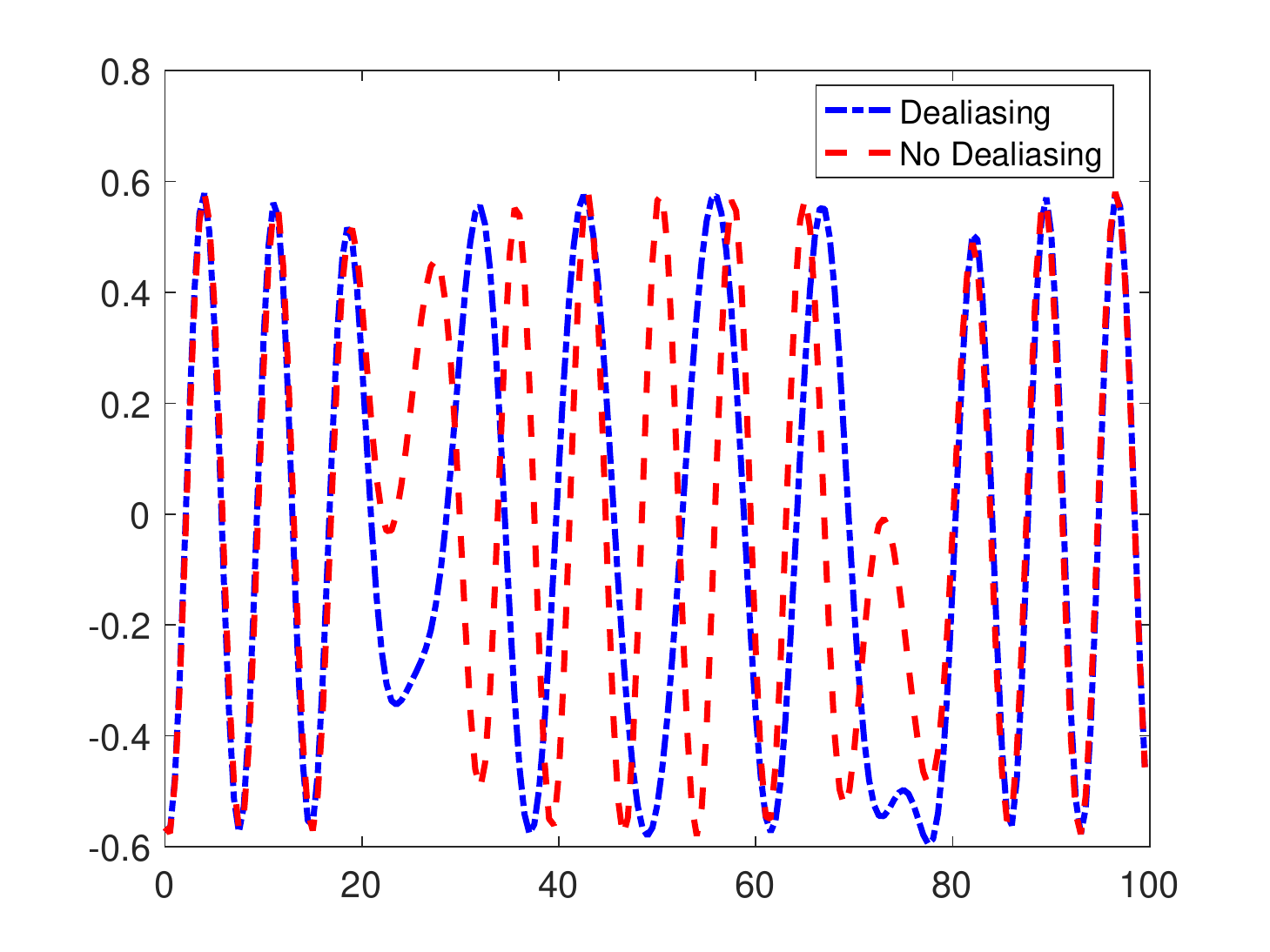}}

{\includegraphics[width=7.5cm]{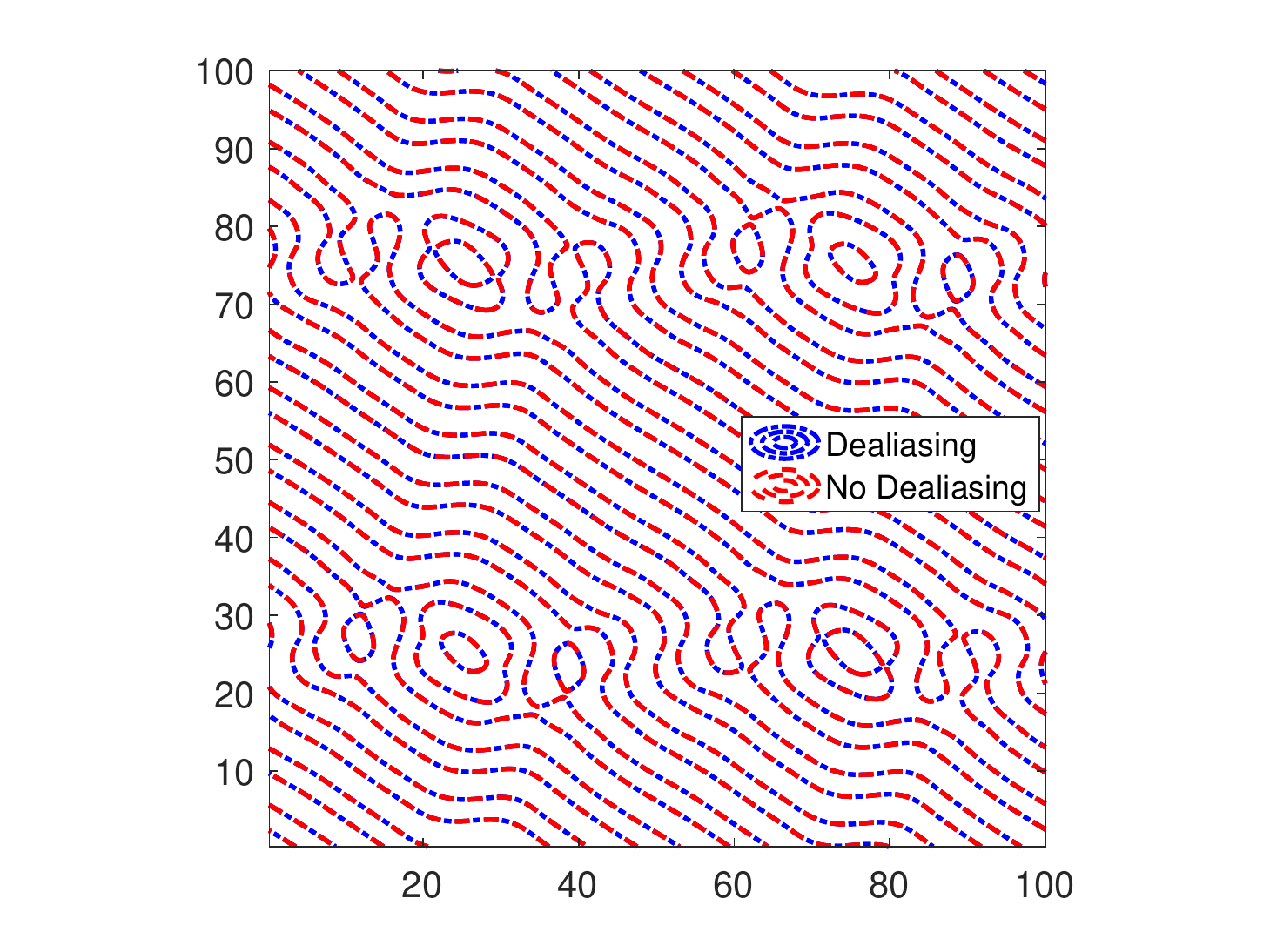}}
{\includegraphics[width=7.5cm]{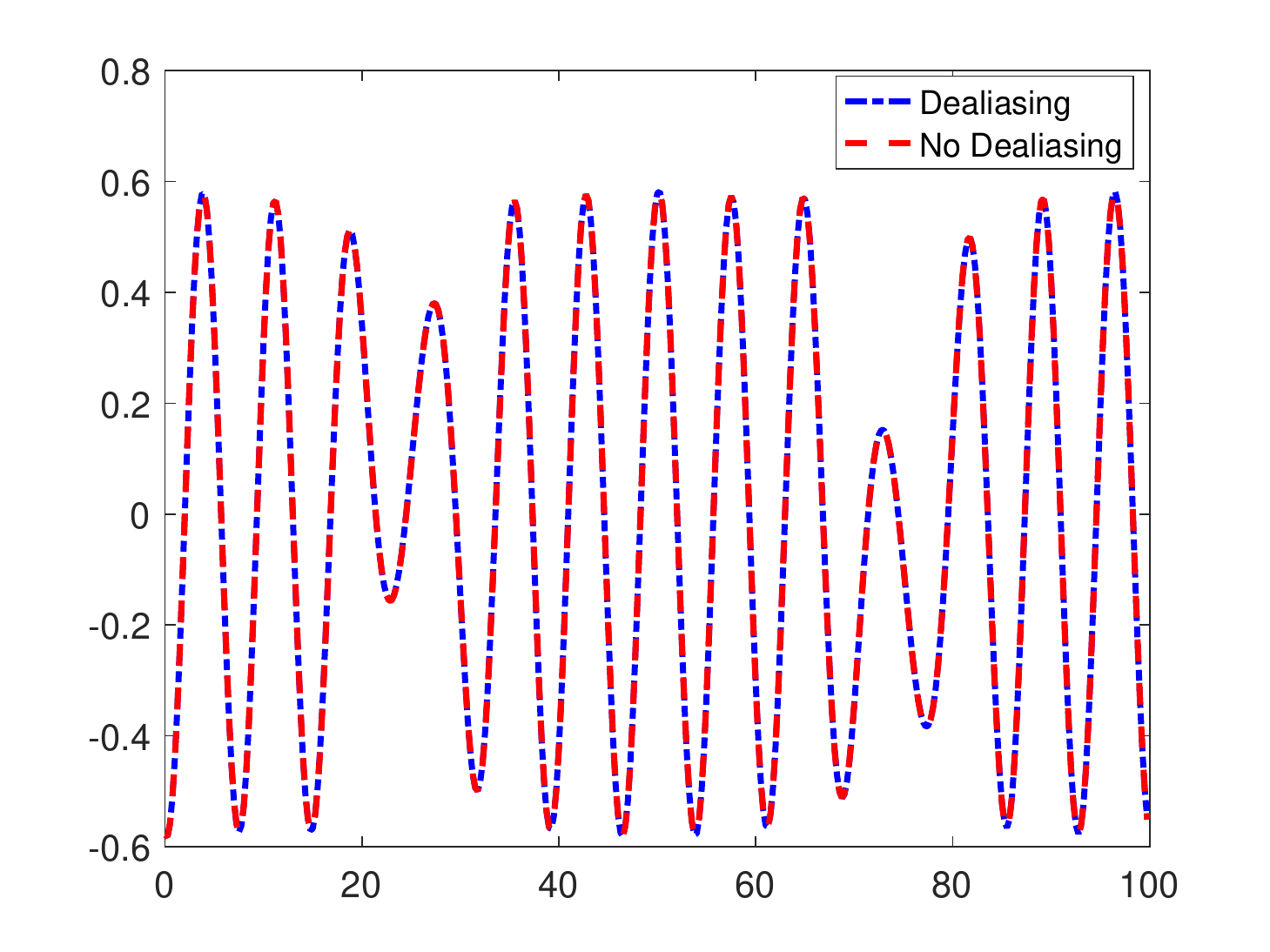}}
\caption{Example \ref{Exp6.3.0}.
Left: contour lines of  $u$ with the value of $-0.1$;
right: cut lines of the numerical solutions along $x=2\pi$.
Top: $N=200$; bottom: $N=400$.}
\label{fig6.7}
\end{figure}

\begin{figure}
\begin{minipage}{0.48\linewidth}
  \centerline{\includegraphics[width=7.5cm]{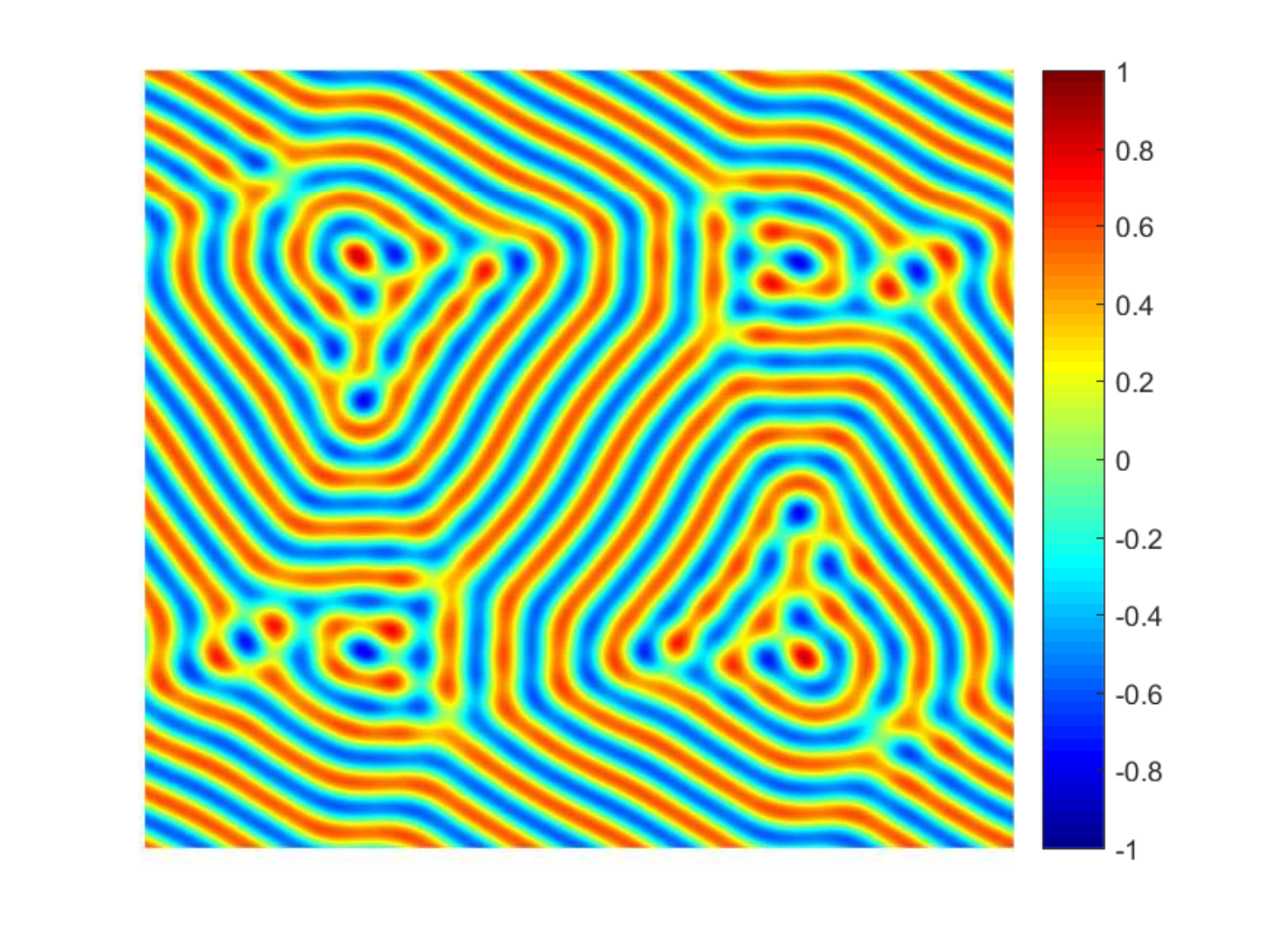}}
\end{minipage}
\hfill
\begin{minipage}{0.48\linewidth}
  \centerline{\includegraphics[width=7.5cm]{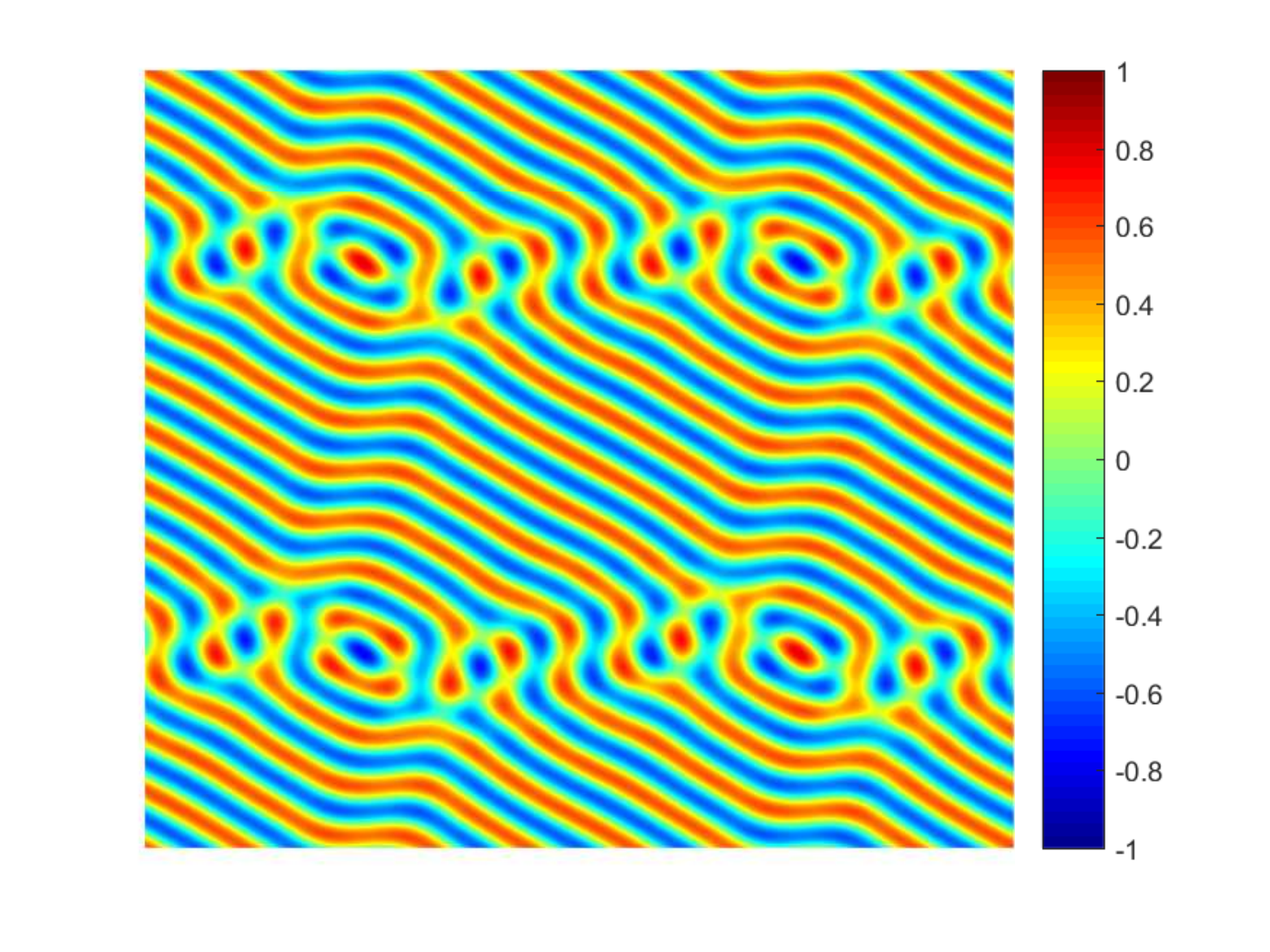}}
\end{minipage}
\caption{Example \ref{Exp6.3.0}.
Snapshots of the numerical solutions at $t = 1000$ derived by {\tt SAV-M(3)} with the de-aliasing. Left: $N=200$; right: $N=400$.}
\label{fig6.8}
\end{figure}

\begin{figure}
\centering
{\includegraphics[width=7.5cm]{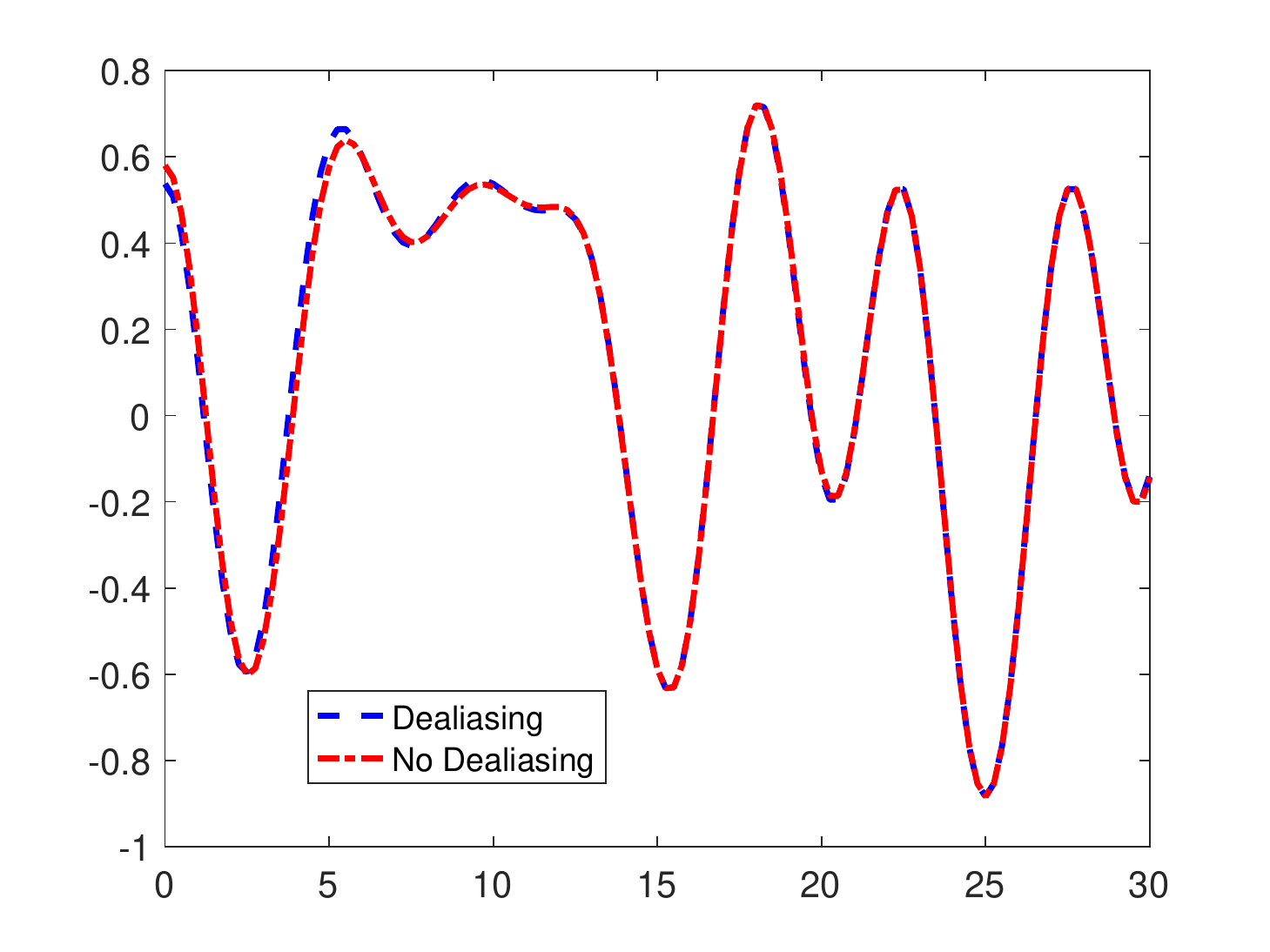}}
{\includegraphics[width=7.5cm]{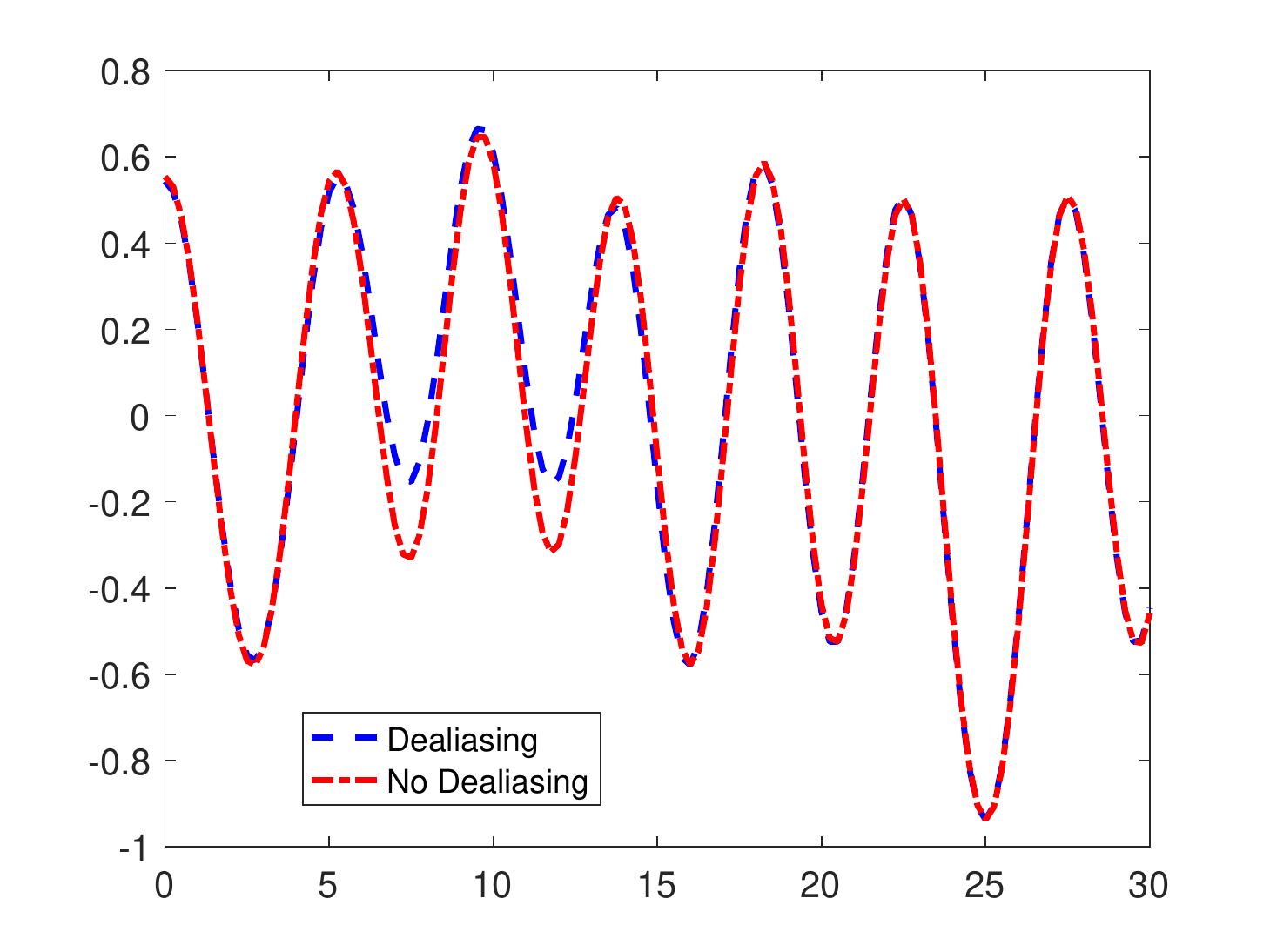}}
\caption{Example \ref{Exp6.3.0}.
Cut lines of the numerical solutions along $y=x$, $x\in[0,30]$, at $t = 545$ (Left) and $t = 670$ (Right).}
\label{fig6.9}
\end{figure}

\begin{example} \label{Exp6.3.1}
It  simulates the polycrystal growth in a supercool liquid
and investigates the modified- and original-energy stabilities of
{\tt SAV-M(1)}$\sim${\tt SAV-M(4)} and {\tt G-SAV-M(1)}$\sim${\tt G-SAV-M(4)}.
For this purpose, the domain $\Omega = (0,400)\times(0,400)$ is partitioned with $N = 400$, the parameter $\epsilon = 0.25$, and the initial value is taken as (see e.g. \cite{LiuZ21})
\begin{align*}
u(x,y,0) =
\begin{cases}
\phi_0  \!+\! B\!\left[\cos\!\left(\frac{\sqrt{6}\vartheta}{6}(y\!-\!x)\right) \cos\!\left(\! \frac{\sqrt{2}\vartheta}{2}(x\!+\!y)\right)\!\! -\! \frac{1}{2} \cos\!\left(\!\frac{\sqrt{6}\vartheta}{3}(y\!-\!x)\right)\! \right]\!,~~(x,y)\!\in\!\Omega_1,  \\[1 \jot]
\phi_0  \!+\! B\!\left[\cos\!\left(\frac{\sqrt{6}\vartheta}{6}(x\!+\!y)\right) \cos\!\left(\! \frac{\sqrt{2}\vartheta}{2}(y\!-\!x)\right)\!\! -\! \frac{1}{2} \cos\!\left(\!\frac{\sqrt{6}\vartheta}{3}(x\!+\!y)\right)\! \right]\!,~~(x,y)\!\in\!\Omega_2,  \\
\phi_0  \!+\! B\!\left[\cos\!\left(\frac{\vartheta}{\sqrt{3}}x\right) \cos\!\left(\! \vartheta y\right)\!\! -\! \frac{1}{2} \cos\!\left(\!\frac{2\vartheta}{\sqrt{3}}x\right)\! \right]\!,~~(x,y)\!\in\!\Omega_3, \\[1 \jot]
\phi_0,~~(x,y)\!\in\!\Omega \backslash (\Omega_1\cup \Omega_2 \cup \Omega_3),
\end{cases}
\end{align*}
where $\phi_0 =0.285$, $ B = 0.446$, $\vartheta = 0.66$, $\Omega_1 = [130,170]\times[130,170]$, $\Omega_2 = [230,270]\times[130,170]$, and $\Omega_3 = [180,220]\times[230,270]$.

Figure \ref{fig5.3.1} presents the discrete total modified-energy curves of {\tt SAV-M(1)}$\sim${\tt SAV-M(4)} and {\tt G-SAV-M(1)}$\sim${\tt G-SAV-M(4)}. They are  monotonically decreasing, and consistent with the theoretical results.
Figure \ref{fig5.3.2} shows the discrete total original-energy curves of {\tt SAV-M(1)}$\sim${\tt SAV-M(4)} and {\tt G-SAV-M(1)}$\sim${\tt G-SAV-M(4)} for \eqref{6.3.1}.
It is shown  that those schemes can preserve the original-energy decay if a suitable time stepsize is chosen.
Figure \ref{fig5.3.3} gives the numerical solution at $t = 2400$ derived by {\tt G-SAV-M(4)} with $\tau = 15$ and $12$. One can find the numerical solution derived by {\tt G-SAV-M(4)} with $\tau = 12$ is similar to that in \cite{LiuZ21,TanZQ22}, but when  $\tau = 15$, an the solution is inaccurate  and the original-energy is not monotonically decreasing.
Remark \ref{remk6.5} will discuss the time stepsize constraints of {\tt SAV-M(1)}$\sim${\tt SAV-M(4)} and {\tt G-SAV-M(1)}$\sim${\tt G-SAV-M(4)}  for the phase field crystal model \eqref{6.3.1}.
\end{example}

\begin{figure}
\centering
{\includegraphics[width=7.5cm]{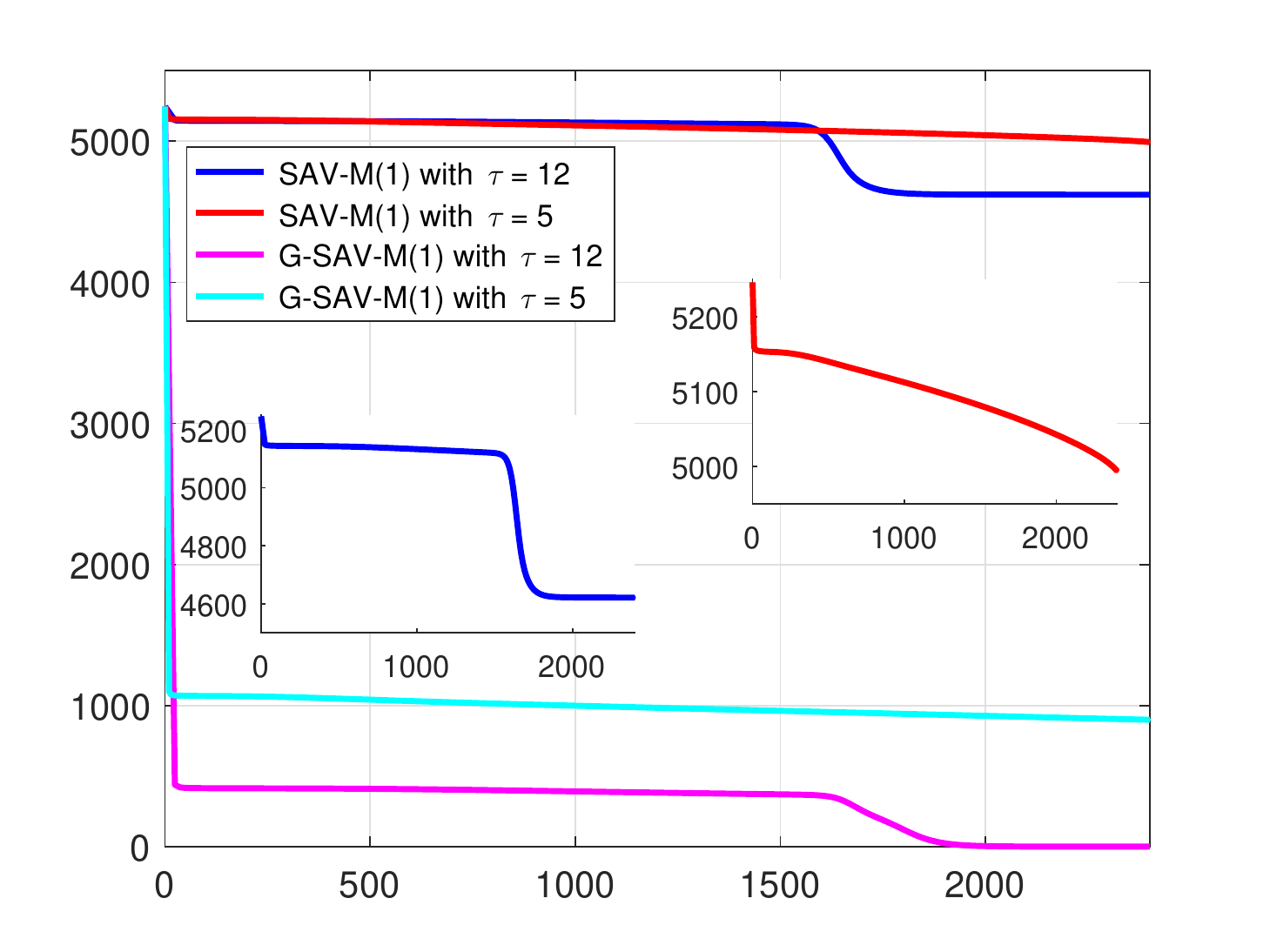} }
{\includegraphics[width=7.5cm]{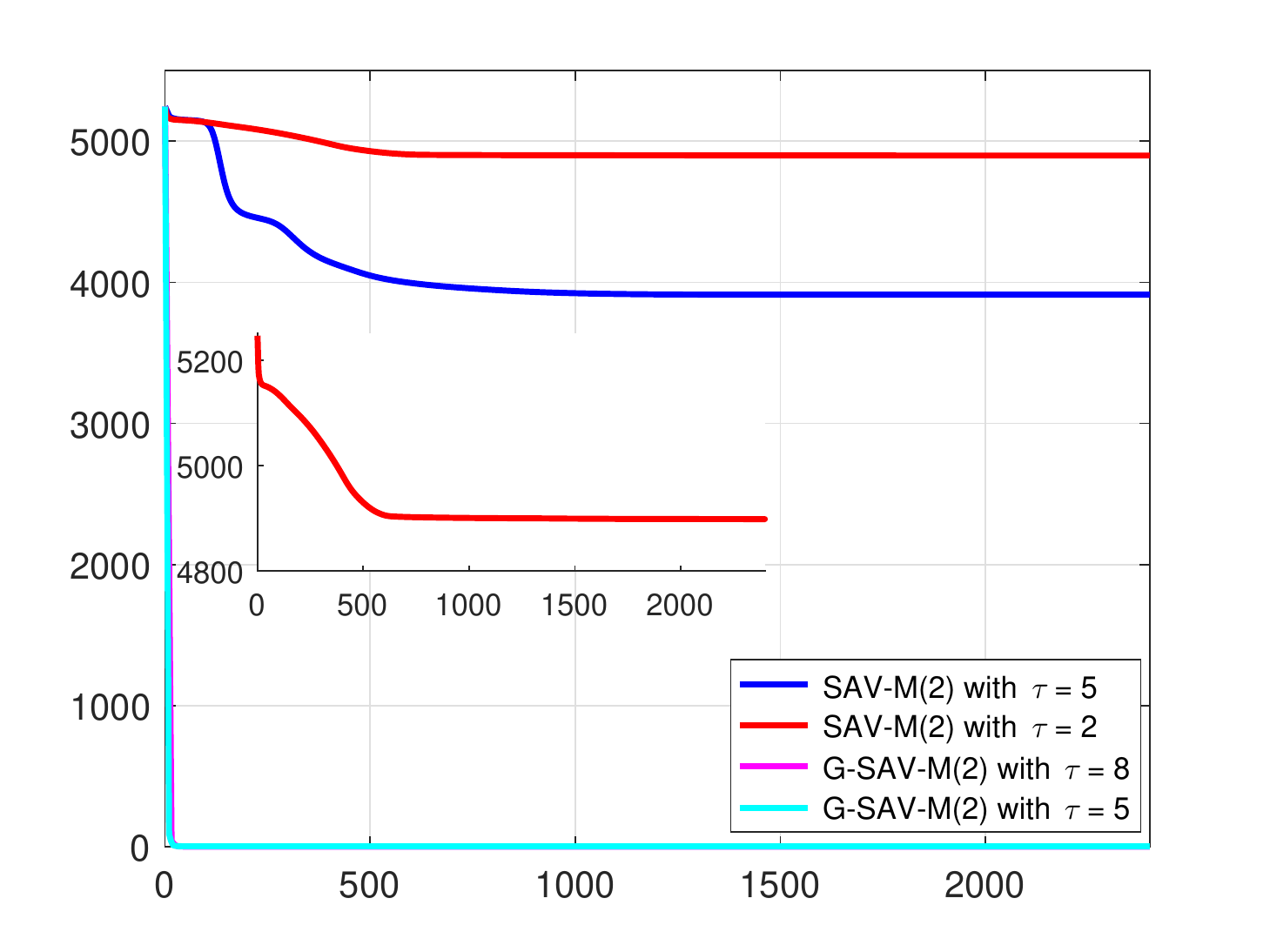} }

{\includegraphics[width=7.5cm]{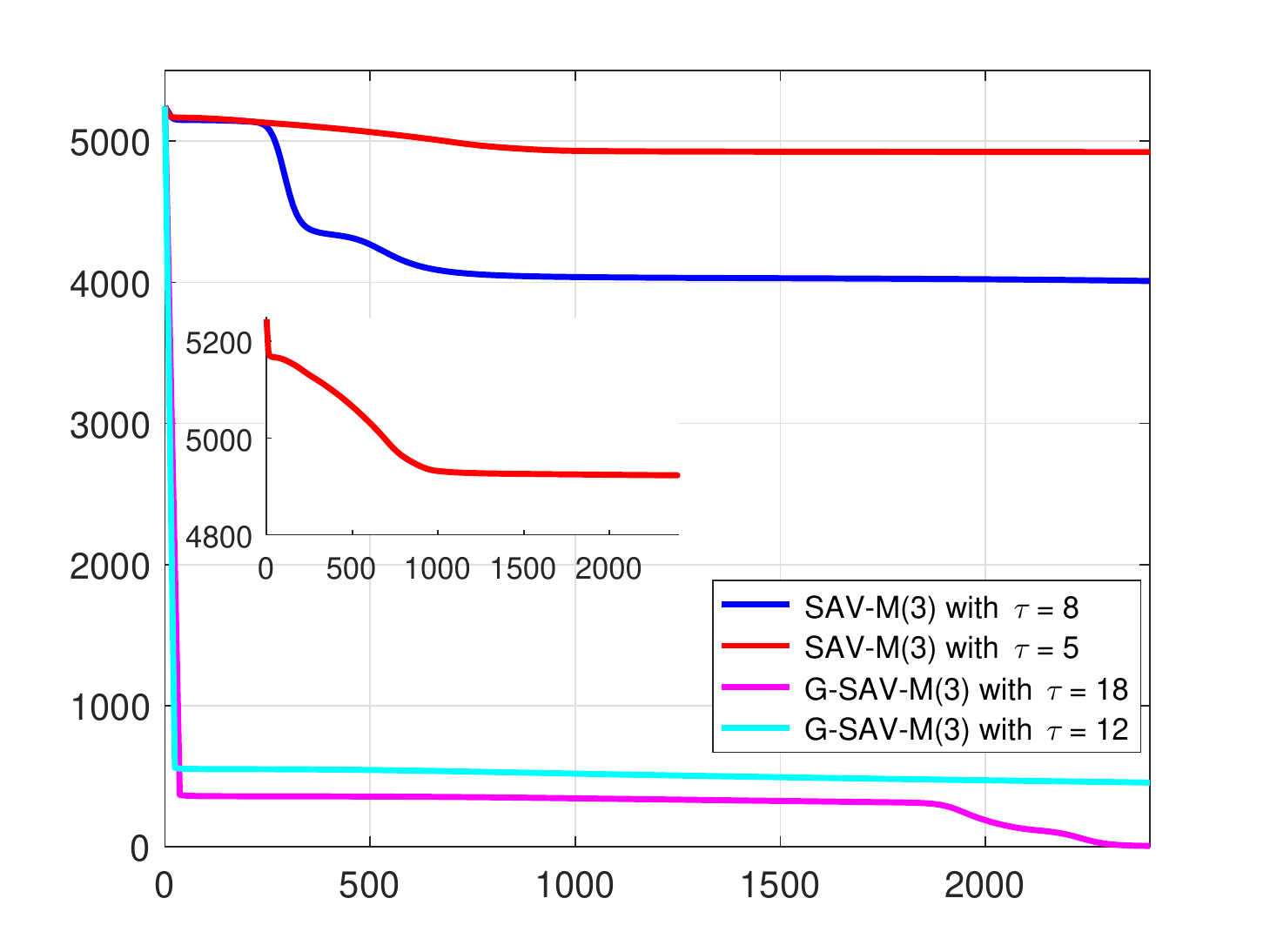} }
{\includegraphics[width=7.5cm]{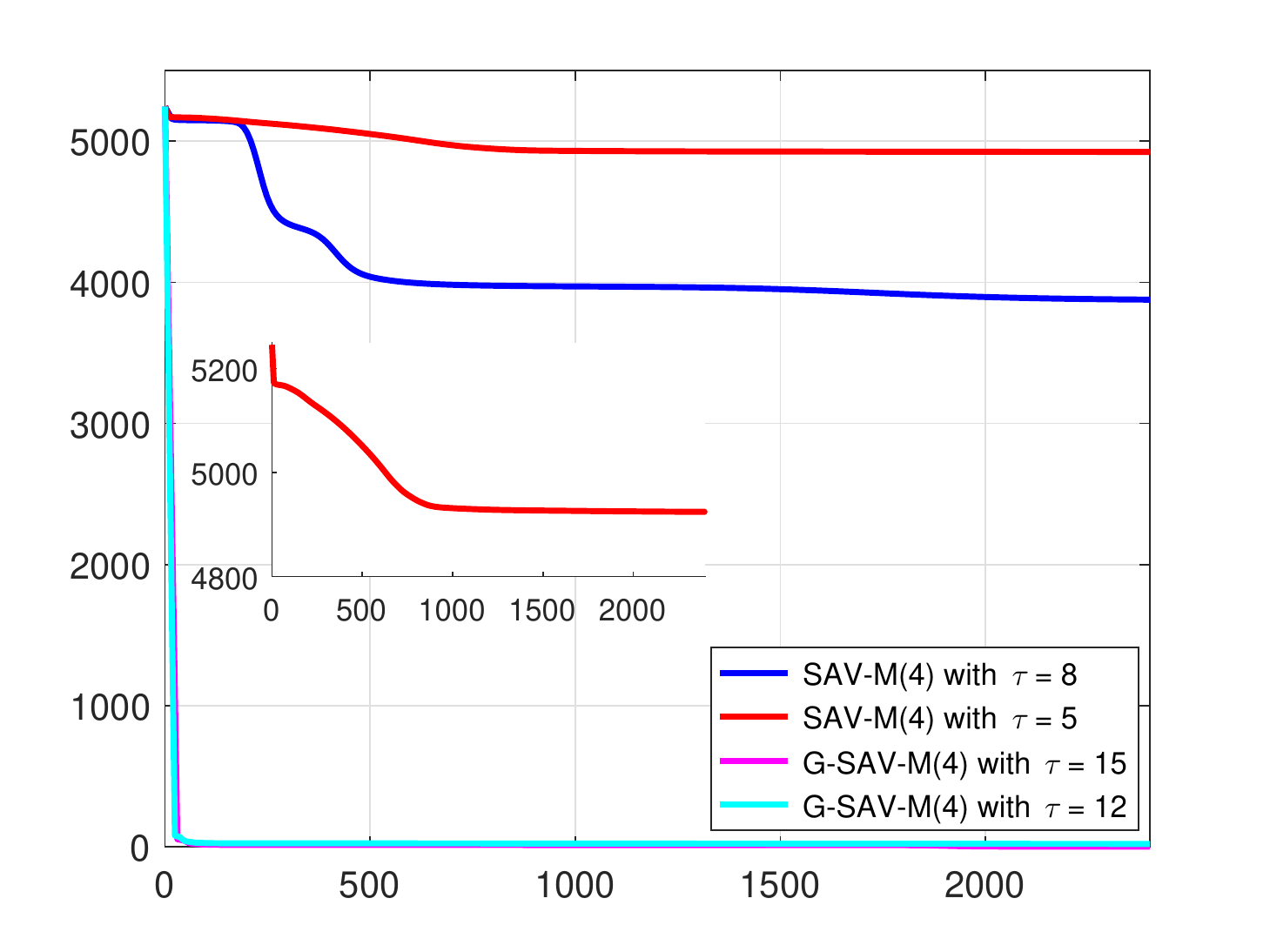} }
\caption{Example \ref{Exp6.3.1}. The discrete total modified-energy curves of {\tt SAV-M(1)}$\sim${\tt SAV-M(4)} and {\tt G-SAV-M(1)}$\sim${\tt G-SAV-M(4)} for the phase field crystal model \eqref{6.3.1}.}
\label{fig5.3.1}
\end{figure}

\begin{figure}
\centering
{\includegraphics[width=7.5cm]{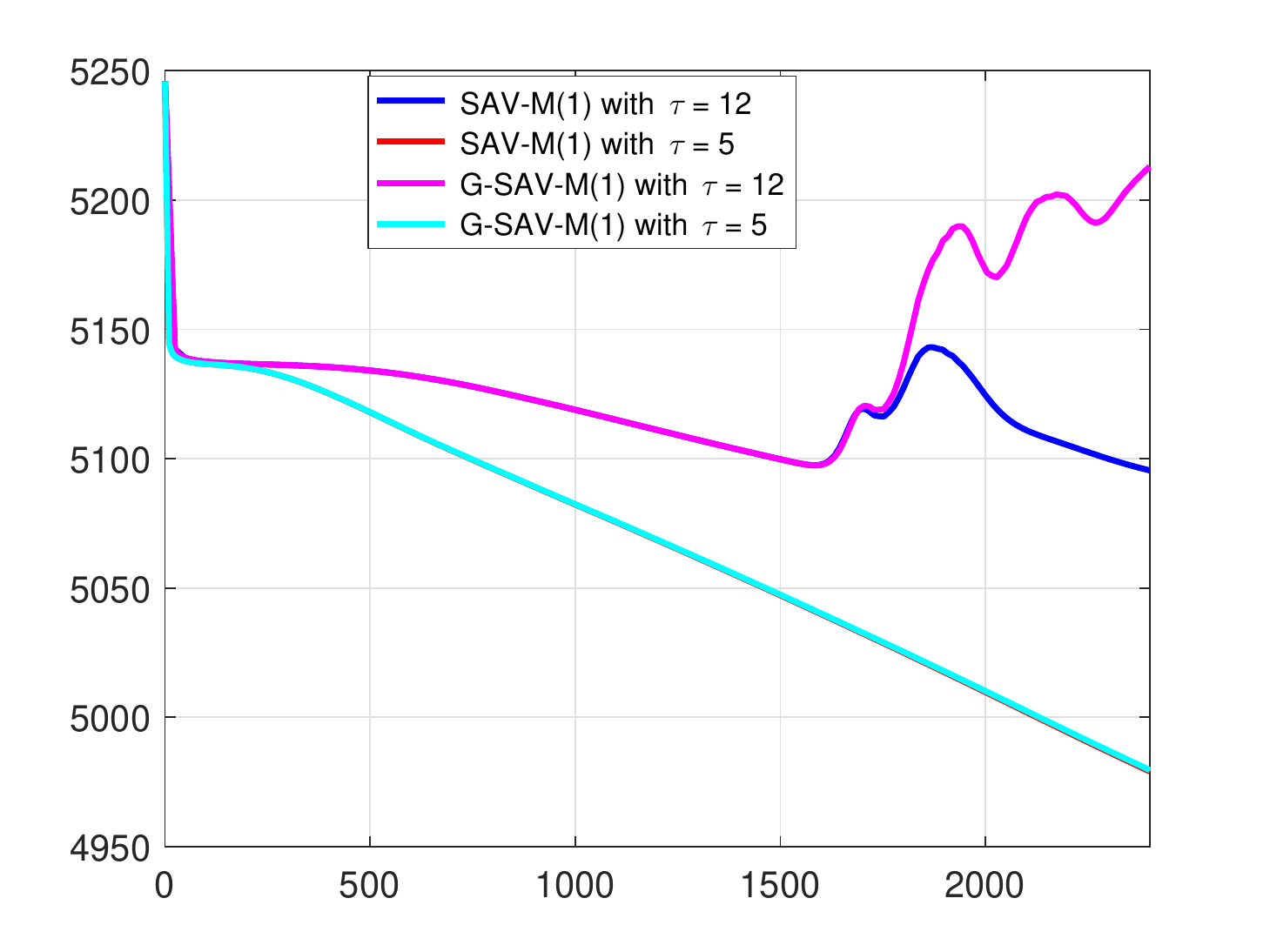} }
{\includegraphics[width=7.5cm]{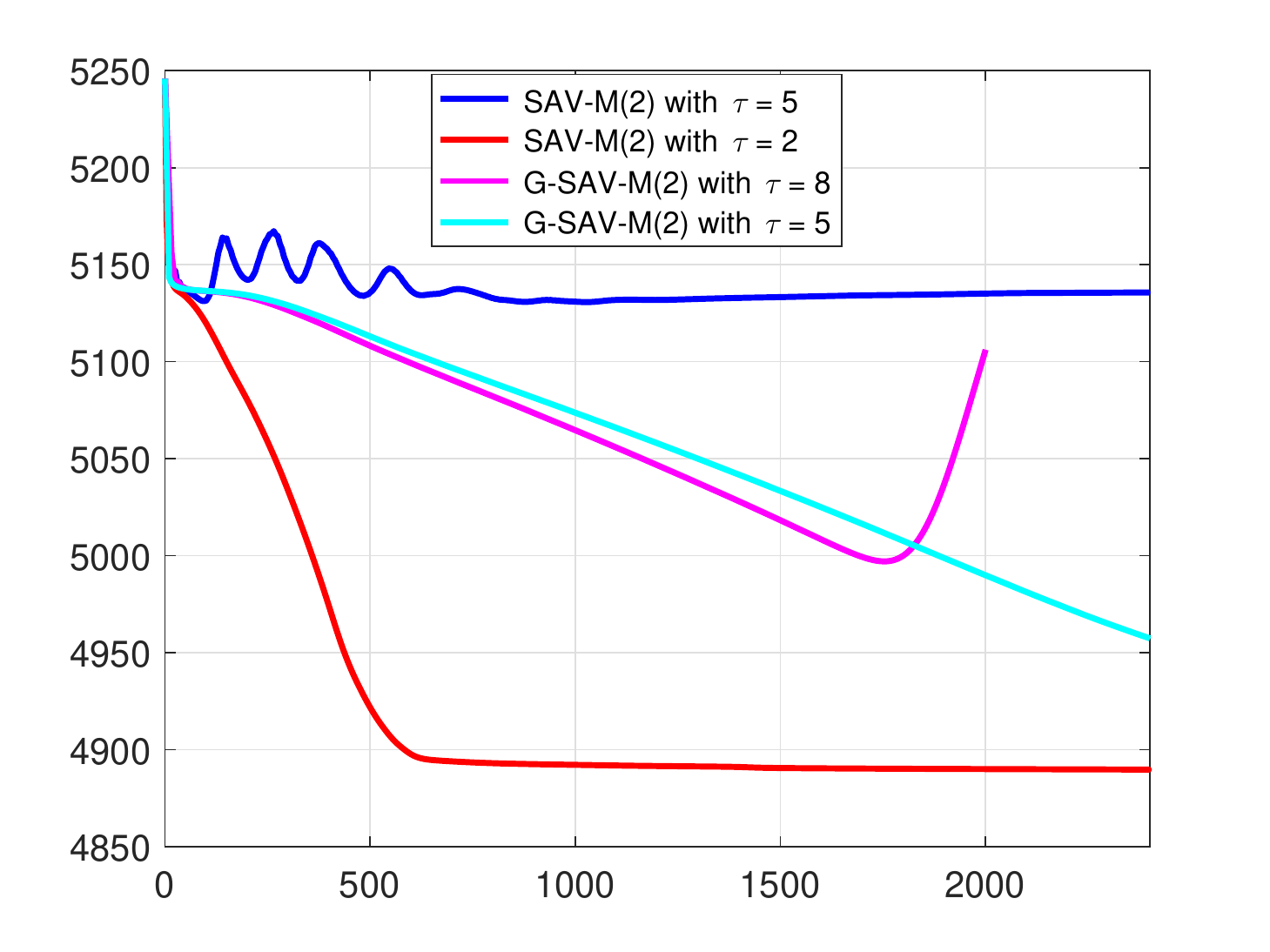} }

{\includegraphics[width=7.5cm]{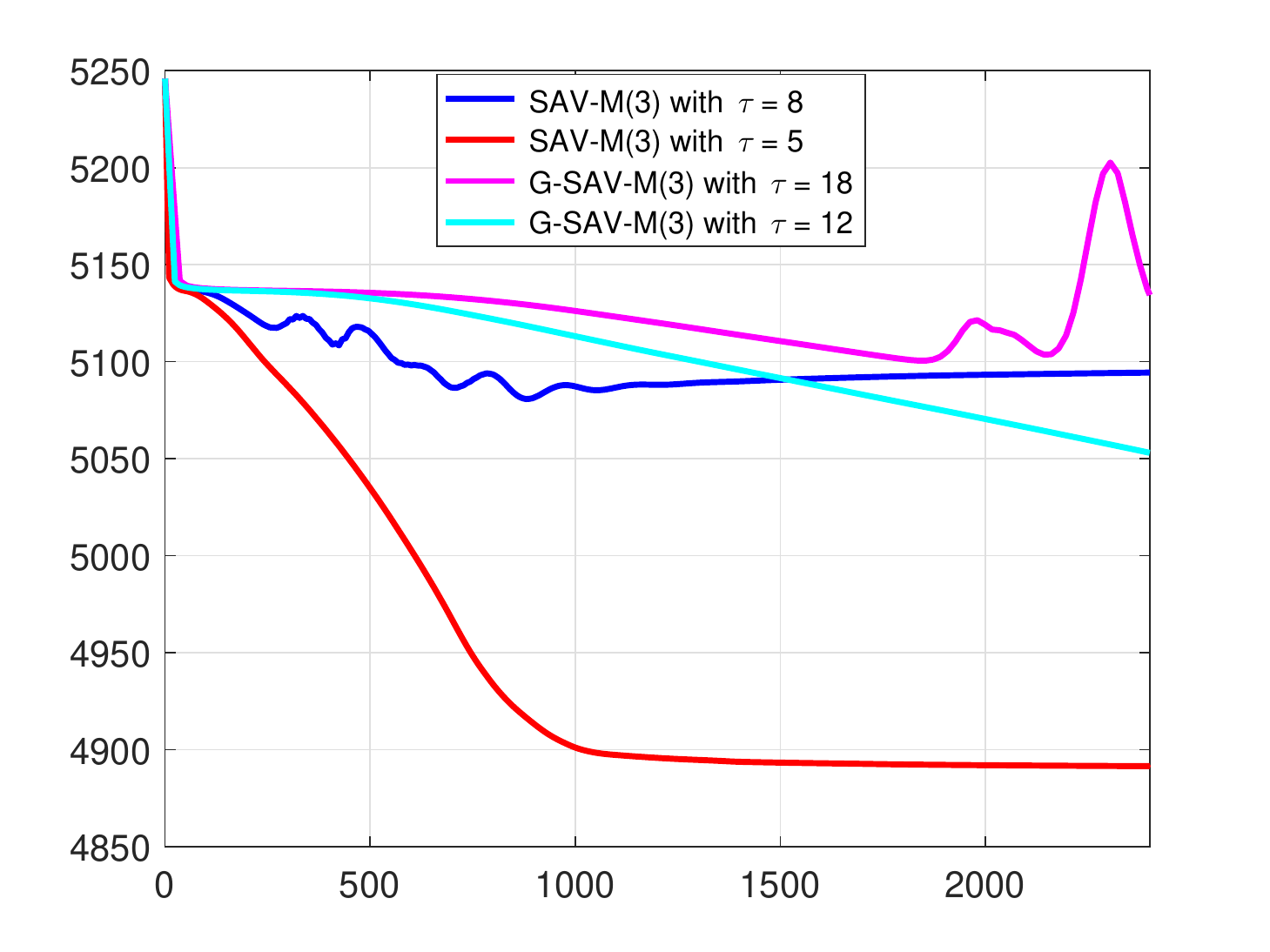} }
{\includegraphics[width=7.5cm]{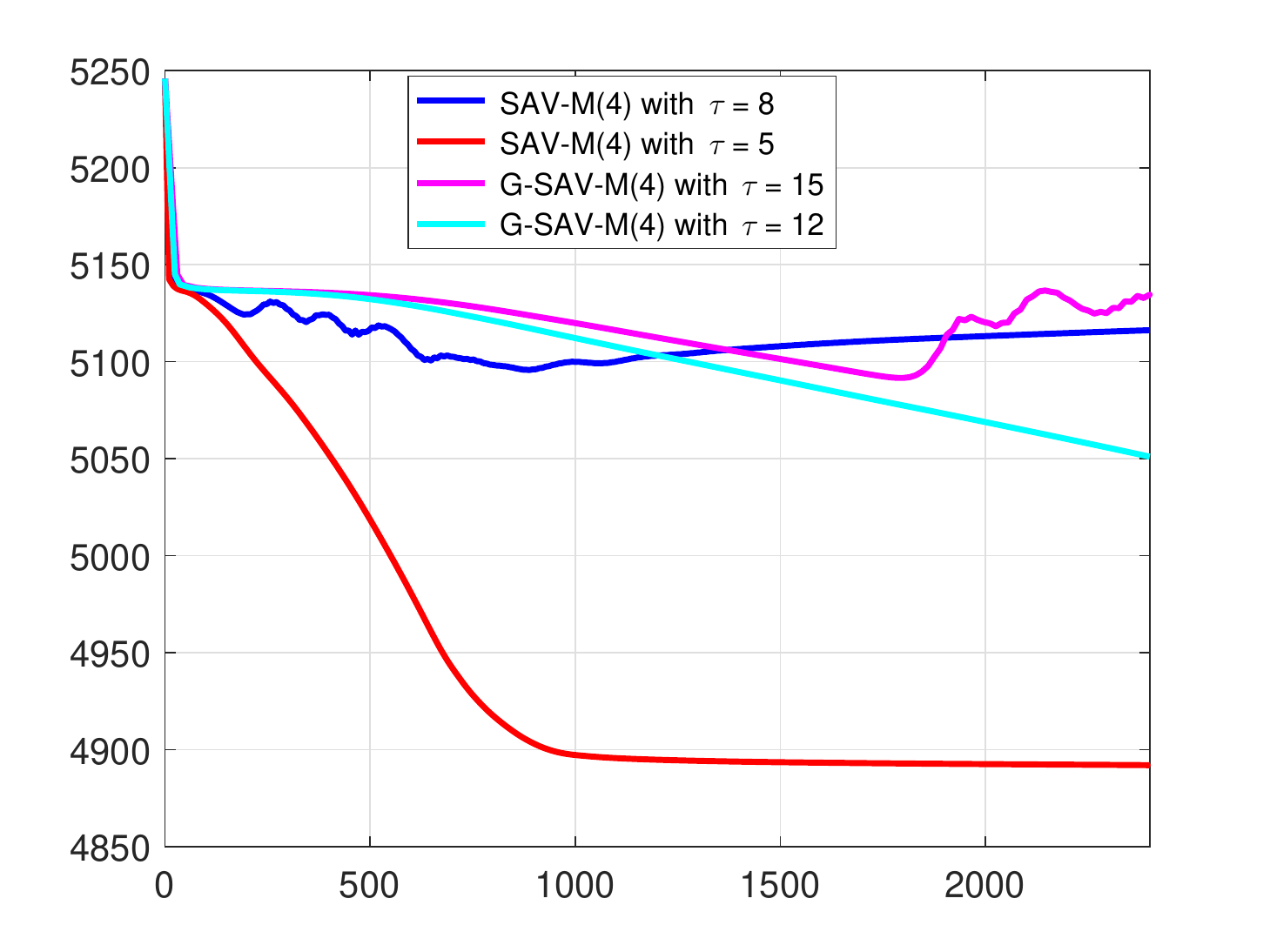} }
\caption{Same as Figure \ref{fig5.3.1}, except for the original-energy.}
\label{fig5.3.2}
\end{figure}

\begin{figure}
\begin{minipage}{0.48\linewidth}
  \centerline{\includegraphics[width=7cm,height=5cm]{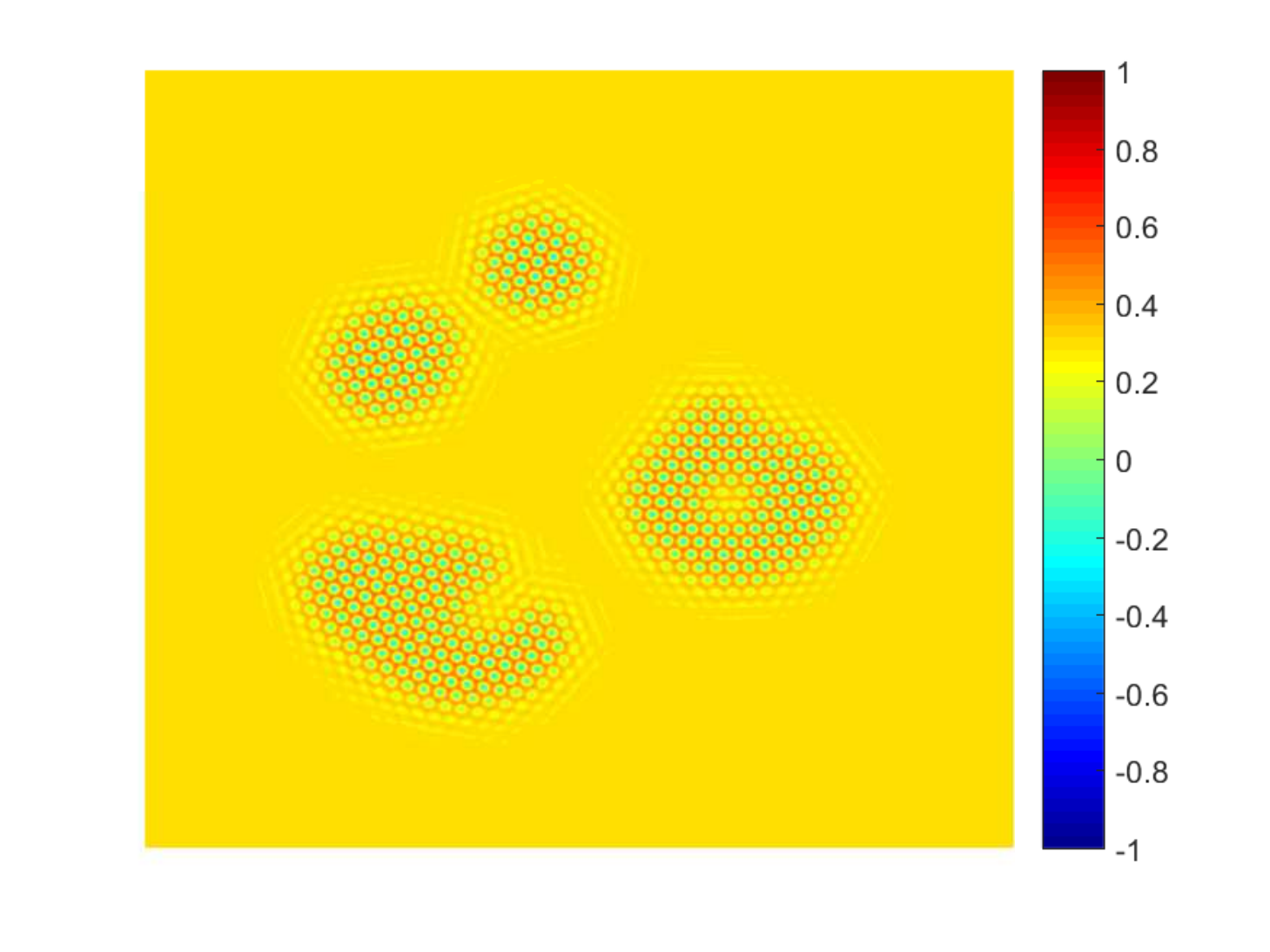}}
\end{minipage}
\hfill
\begin{minipage}{0.48\linewidth}
  \centerline{\includegraphics[width=7cm,height=5cm]{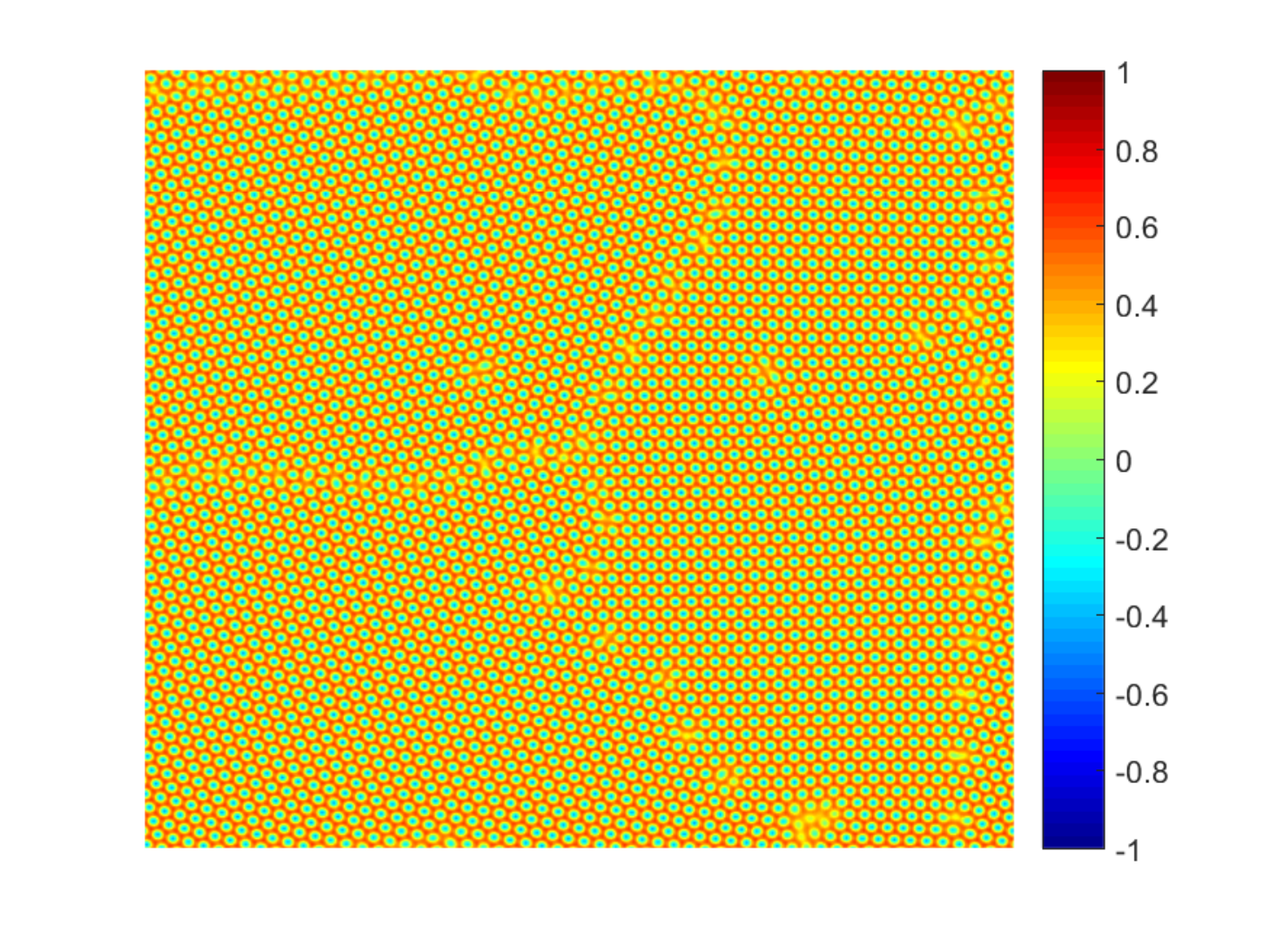}}
\end{minipage}
\caption{Example \ref{Exp6.3.1}. Numerical solutions at $t=2400$  derived by {\tt G-SAV-M(4)} with $\tau = 15$ (Left) and $12$ (Right), respectively.}
\label{fig5.3.3}
\end{figure}

\begin{figure}
\begin{minipage}{0.48\linewidth}
  \centerline{\includegraphics[width=7cm,height=5cm]{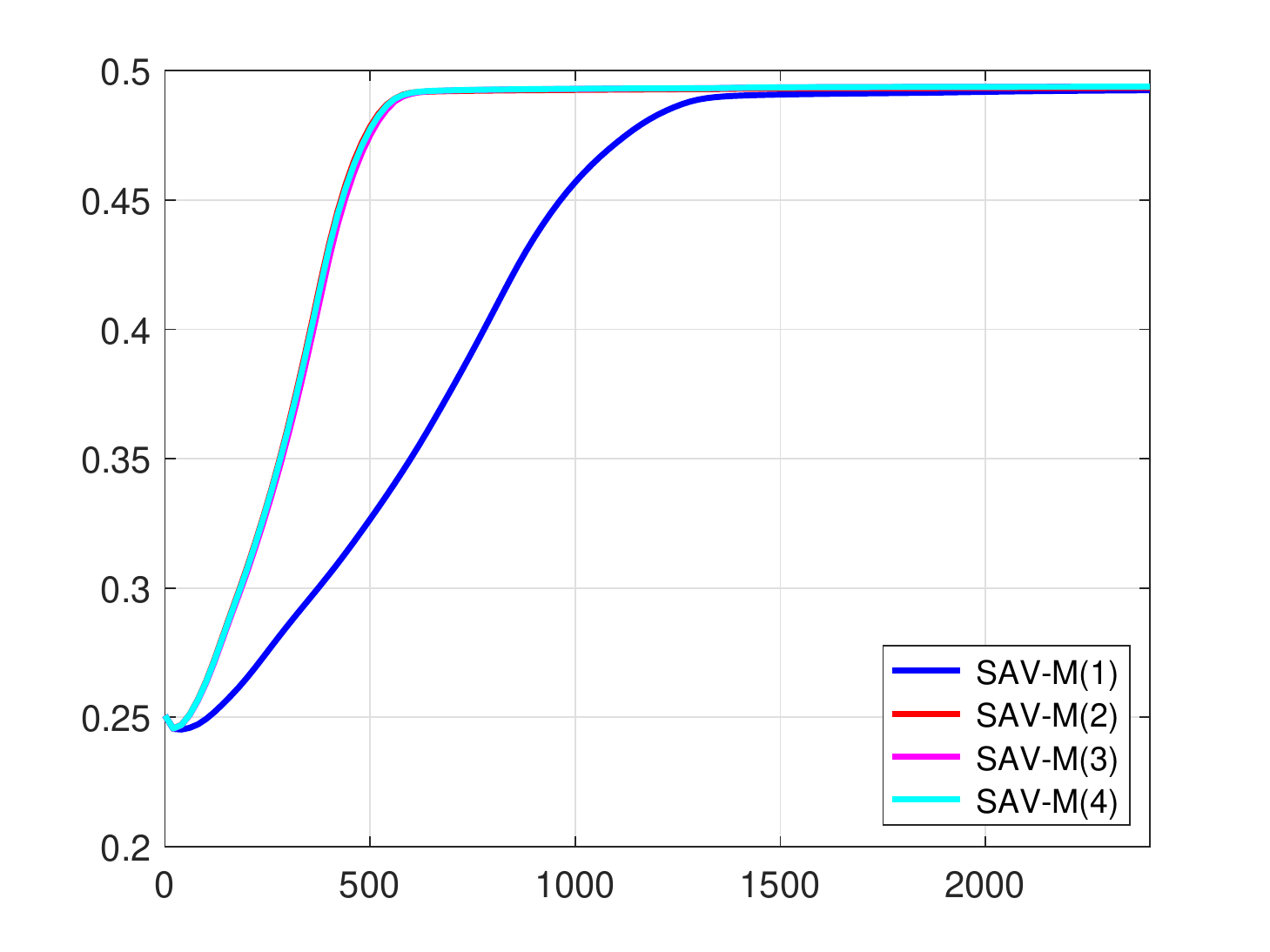}}
\end{minipage}
\hfill
\begin{minipage}{0.48\linewidth}
  \centerline{\includegraphics[width=7cm,height=5cm]{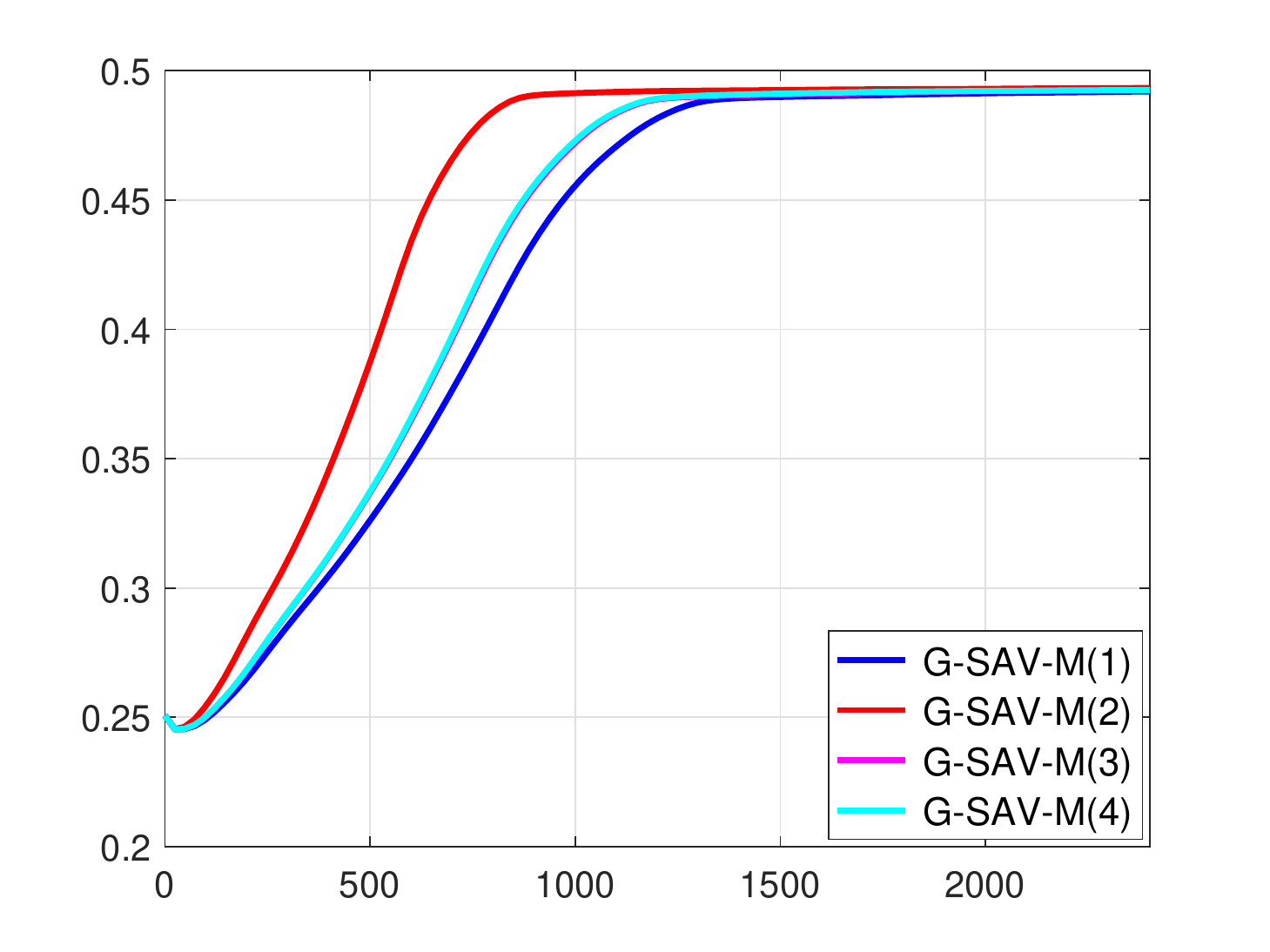}}
\end{minipage}
\caption{Example \ref{Exp6.3.1}.   $\psi^n$ (y-axis) derived by {\tt SAV-M(1)}$\sim${\tt SAV-M(4)} (Left) and {\tt G-SAV-M(1)}$\sim${\tt G-SAV-M(4)} with $\tau = 1$.}
\label{fig5.3.4}
\end{figure}

\begin{remark}   \label{remk6.5}
This remark discusses the time stepsize constraints of {\tt SAV-M(1)}$\sim${\tt SAV-M(4)} and {\tt G-SAV-M(1)}$\sim${\tt G-SAV-M(4)} for the phase field crystal model \eqref{6.3.1}.

Applying the Fourier pseudo-spectral method to \eqref{6.3.1} yields the   ODE system
\begin{align}   \label{6.3.2}
\frac{d \hat{u}_{k,l}}{d t } = - \frac{\pi^3\left(k^2 \!+\! l^2 \right)^3}{200^3}\hat{u}_{k,l} -  \frac{\pi(k^2\!+\!l^2)}{200}\left[ \hat{w}_{k,l} + \frac{3}{8}\hat{u}_{k,l} - \frac{\pi(k^2\!+\!l^2)}{100}\hat{u}_{k,l} \right],
\end{align}
where $\{\hat{w}_{k,l}, (k,l) \in \widehat{\mathbb{S}}_N\}$ are the discrete Fourier coefficients of the cubic term $u^3$ and given by \eqref{6.1.5}.
Similarly, 
\eqref{6.3.2} can be viewed as the test equation \eqref{8.3.1} with
\[ \xi = - \frac{\pi^3\left(k^2 + l^2 \right)^3}{200^3},~~~\zeta = -  \frac{\pi(k^2\!+\!l^2)}{200}\!\left[ \frac{3}{N^2} \!\sum\limits_{(i,j)\in\mathbb{S}_N} \!|u_{i,j}|^2 \!+\! \frac{3}{8} \!-\! \frac{\pi(k^2\!+\!l^2)}{100} \right]\!, ~~~(k,l) \in \widehat{\mathbb{S}}_N. \]
 For Example \ref{Exp6.3.1},
Figure \ref{fig5.3.4} gives the curves of  $\psi^n = \frac{3}{N^2} \sum\limits_{(i,j)\in\mathbb{S}_N} |u_{i,j}|^2$ derived by {\tt SAV-M(1)}$\sim${\tt SAV-M(4)} and {\tt G-SAV-M(1)}$\sim${\tt G-SAV-M(4)} with $\tau = 1$. It is shown that  $\psi^n \le 0.5$ so that $\zeta \ge -  \frac{\pi(k^2+l^2)}{200}\!\left[  \frac{7}{8} \!-\! \frac{\pi(k^2+l^2)}{100} \right]$.
In the following, one may take $\zeta \approx -  \frac{\pi(k^2+l^2)}{200}\!\left[  \frac{7}{8} \!-\! \frac{\pi(k^2+l^2)}{100} \right]$ and use \ref{Appx3} to discuss the time stepsize constraints for {\tt SAV-M(1)}$\sim${\tt SAV-M(4)} and {\tt G-SAV-M(1)}$\sim${\tt G-SAV-M(4)}.
Specifically, when $(\alpha_0,\beta_0,\beta_2) = (0,0,1)$,
 it requires $$\tau < \left\{ \frac{2}{\max(\xi-3\zeta)}: \zeta < \frac{1}{3}\xi\right\},$$
  and a simple calculation shows that when $k=3$ and $l=2$, $\max\left\{\xi-3\zeta\right\} = 0.2788$,  which gives $\tau \lesssim 7.21$;
when $(\alpha_0,\beta_0,\beta_2) = (-1/3,5/12,3/4)$,
 $$ \tau < \left\{ \frac{4}{3\max(\xi-2\zeta)}: \zeta < \frac{1}{2}\xi\right\},$$
  which is combined with $\max\left\{\xi-2\zeta\right\} = 0.2964$ for $k=2$ and $l=3$ to get  $\tau \lesssim 7.32$;
when $(\alpha_0,\beta_0,\beta_2) = (1/3,0,2/3)$,
$$ \tau < \left\{ \frac{4}{\max(\xi-3\zeta)}: \zeta < \frac{1}{3}\xi\right\}, $$
so that $\tau \lesssim 14.42$;
and when $(\alpha_0,\beta_0,\beta_2) = (1/3,-1/6,1/2)$,
$$\tau < -\frac{4}{3\min(\zeta)},$$
which yields  $\tau \lesssim 13.99$ since $\min(\zeta) = 0.1524$ for $k=3$ and $l=2$.
Note that those time stepsize estimates for  {\tt SAV-M(1)}$\sim${\tt SAV-M(4)} and {\tt G-SAV-M(1)}$\sim${\tt G-SAV-M(4)} are sufficient.
Compared to the numerical results shown in Figure \ref{fig5.3.2},
one needs to take slightly smaller time stepsizes to {ensure} the original-energy decay when {\tt SAV-M(1)}$\sim${\tt SAV-M(4)} are applied to \eqref{6.3.1}, but the above time stepsize estimates are almost consistent with those in numerical experiments for {\tt G-SAV-M(1)}$\sim${\tt G-SAV-M(4)}.

\end{remark}

\section{Conclusion}  \label{sec:7}

This paper continued to study linear and unconditionally modified-energy stable numerical schemes (abbreviated as SAV-GL) for the gradient flows.
Those schemes were built on the   SAV  technique and the general linear time discretizations (GLTD) as well as the extrapolation for the nonlinear term,  and two linear systems with the same constant coefficient  were solved at each time step.
Different from \cite{TanZQ22},  the GLTDs with three parameters  discussed here  were not necessarily algebraically stable.
Some algebraic identities were first derived by  using the method of undetermined coefficients
and then   used to establish the  modified-energy  inequalities for the unconditional  modified-energy stability of the semi-discrete-in-time SAV-GL schemes.
It was worth emphasizing that those algebraic identities or energy inequalities are not necessarily unique  for some  choices of  three parameters in the GLTDs.
%
In order to demonstrate numerically the energy stability of our SAV-GL schemes,  the Fourier pseudo-spectral spatial discretization was employed for the gradient flow models with periodic boundary conditions. The effect of the aliasing error and the
de-aliasing by zero-padding provided in \ref{sec:5} on the numerical results were investigated.

Numerical experiments  were conducted on the Allen-Cahn, the Cahn-Hilliard, and the phase field crystal models, and well demonstrated   the unconditional modified-energy stability of {\tt SAV-M(1)}$\sim$ {\tt SAV-M(4)}  in comparison to another SAV-GL schemes (abbreviated as {\tt G-SAV-M(1)}$\sim$ {\tt G-SAV-M(4)}) built on
the generalized SAV
and
the  effectiveness of the de-aliasing by zero-padding.
Numerical results also  showed that a suitable time stepsize were required for the SAV-GL schemes to ensure the original-energy decay.
With the help of discussing the stability regions for the semi-implicit SAV-GL schemes applied to the test equation in \ref{Appx3},
the time stepsizes for {\tt SAV-M(1)}$\sim$ {\tt SAV-M(4)} 
were estimated
for the Allen-Cahn, the Cahn-Hilliard, and the phase field crystal models. Our computations showed that  those time stepsize constraints could ensure  the original-energy decay essentially.

\begin{appendix}

\section{Proof of Lemma \ref{lemma2.1}}  \label{Appx1}

This appendix   proves Lemma \ref{lemma2.1} by
{discussing}
the $A$-stability conditions of the fully implicit time discretizations based on \eqref{3.1.3}-\eqref{3.1.4}.

Applying  \eqref{3.1.3}-\eqref{3.1.4} to the test equation
\begin{align*}
\frac{d u}{dt}=u'(t)  = \xi u(t),~~ t\in(0,T],~~ u(0) = u_0,
\end{align*}
yields
\begin{align}   \label{8.2}
\frac{1}{1\!-\!\alpha_0}u_{n+1} - \frac{1\!+\!\alpha_0}{1\!-\!\alpha_0} u_n + \frac{\alpha_0}{1\!-\!\alpha_0} u_{n-1} = \bar{\xi}\left[ \frac{\beta_2}{1\!-\!\alpha_0} u_{n+1} + \frac{\beta_1}{1\!-\!\alpha_0} u_{n} + \frac{\beta_0}{1\!-\!\alpha_0} u_{n-1} \right],
\end{align}
where $\bar{\xi} = \xi \tau$,  $\xi$ is a complex number,
and  the symbol $u_{n}$ have been temporarily used to replace the previous approximate solution $u^n$
for convenience.
Substituting $u_j = x^j$ into \eqref{8.2} and dividing by $x^{n-1}$ give the characteristic equation
\begin{align}   \label{8.3}
\rho(x) - \bar{\xi} \sigma(x) = 0,
\end{align}
with 
\[ \rho(x) =  \frac{1}{1\!-\!\alpha_0}x^2 - \frac{1\!+\!\alpha_0}{1\!-\!\alpha_0} x + \frac{\alpha_0}{1\!-\!\alpha_0},~~~ \sigma(x) = \frac{\beta_2}{1\!-\!\alpha_0} x^2 + \frac{\beta_1}{1\!-\!\alpha_0} x + \frac{\beta_0}{1\!-\!\alpha_0}. \]
According to {\cite[Def. 1.1]{Hairer}},  the scheme \eqref{8.2} is $A$-stable iff for any $\bar{\xi} \in \mathbb{C}^-$, all solutions of  \eqref{8.3} are smaller or equal to one in modulus, and the multiple solutions are strictly smaller than one.
It is known that  all roots of the polynomial  $c_2 x^2 + c_1 x + c_0$ are smaller or equal to one in modulus iff $|c_0| \le |c_2|$ and $|c_1| \le |c_0+c_2|$, see e.g. \cite{Calvo88}.
Thus,  when
\begin{align}  \label{8.4}
\left| \alpha_0 - \beta_0 \bar{\xi}\right| \le \left| 1 - \beta_2\bar{\xi}\right|,~~~ \left|1+\alpha_0 -\beta_1 \bar{\xi} \right| \le \left|1+\alpha_0- \left(\beta_0\!+\!\beta_2\right) \bar{\xi} \right|,
\end{align}
  for any $ \bar{\xi}\in\mathbb{C}^-$,
all roots of the characteristic polynomial
\[   \mathbb{P}(x):= \rho(x) - \bar{\xi} \sigma(x) =  \left(\frac{1}{1\!-\!\alpha_0} - \frac{\beta_2}{1\!-\!\alpha_0}\bar{\xi} \right)x^2 - \left(\frac{1\!+\!\alpha_0}{1\!-\!\alpha_0} - \frac{\beta_1}{1\!-\!\alpha_0} \bar{\xi}\right)x + \frac{\alpha_0}{1\!-\!\alpha_0}-\frac{\beta_0}{1\!-\!\alpha_0} \bar{\xi}, \]
are smaller or equal to one in modulus so that  the scheme \eqref{8.2} is $A$-stable. Let $\bar{\xi} = a + \imath b$ with $\imath = \sqrt{-1}$, $a\le0$ and $b\in \mathbb{R}$. The first inequality in \eqref{8.4} is equivalent to
\begin{align*}
2a(\alpha_0\beta_0\!-\!\beta_2)\!+\! \left(\beta_2^2\!-\!\beta_0^2\right)\left( a^2 \!+\! b^2\right) \!+\!1\!-\!\alpha_0^2 \ge 0,~~~\forall~a\le 0,~b\in\mathbb{R}.
\end{align*}
A direct check shows that the parameters $\alpha_0, \beta_0$ and $\beta_2$ should satisfy
\begin{align*}
-1\le \alpha_0 < 1,~~~\alpha_0\beta_0 \le \beta_2,~~~|\beta_0| \le |\beta_2|,
\end{align*}
which further gives
\begin{align}  \label{8.5}
-1\le \alpha_0 < 1,~~~\beta_2>0,~~~|\beta_0| \le \beta_2.
\end{align}
On the other hand, the second inequality in \eqref{8.4} is equivalent to
\begin{align*}
\left(2\beta_0\!+\!2\beta_2\!+\!\alpha_0\!-\!1\right)\left(1\!-\!\alpha_0\right)(a^2 \!+ \!b^2) \!+\! 2(1\!+\!\alpha_0)\left(1\!-\!\alpha_0\!-\!2\beta_0\!-\!2\beta_2\right) a \!\ge \! 0,~~~\forall~ a\le 0,~b\in\mathbb{R},
\end{align*}
which yields
\begin{align}  \label{8.6} \left(2\beta_0\!+\!2\beta_2\!+\!\alpha_0\!-\!1\right)\left(1\!-\!\alpha_0\right) \ge 0,~~~ (1\!+\!\alpha_0)\left(1\!-\!\alpha_0\!-\!2\beta_0\!-\!2\beta_2\right) \le 0.
\end{align}
Combining  \eqref{8.5} with \eqref{8.6} yields that
the scheme \eqref{8.2} is $A$-stable
 when the parameters $\alpha_0, \beta_0$ and $\beta_2$ satisfy
\begin{align}  \label{8.10}
-1\le \alpha_0<1,~~~\beta_2>0,~~~|\beta_0| \le \beta_2,~~~1-\alpha_0-2\beta_0-2\beta_2 \le 0.
\end{align}
Some special cases are discussed as follows.

$\bullet$  When $\alpha_0 = \beta_0 = 0$,  \eqref{8.2} reduces to a one-step scheme with parameter $\beta_2$, i.e,
\begin{align}  \label{8.7}
u_{n+1} -  u_n =  \beta_2 \bar{\xi} u_{n+1} + (1\!-\!\beta_2) \bar{\xi}u_{n},
\end{align}
which is second-order accurate only for $\beta_2 = \frac{1}{2}$. A direct check shows that the condition \eqref{8.10} becomes $\beta_2\ge \frac{1}{2}$, under which \eqref{8.7} is $A$-stable.

$\bullet$ When $\beta_2 = \frac{1}{2}(1+\alpha_0)+\beta_0$,
and $\alpha_0$ and $\beta_0$ are not zero simultaneously, \eqref{8.2} reduces to a class of two-step and second-order schemes, i.e.,
\begin{align}   \label{8.8}
\frac{1}{1\!-\!\alpha_0}u_{n+1} \!- \!\frac{1\!+\!\alpha_0}{1\!-\!\alpha_0} u_n \!+ \! \frac{\alpha_0}{1\!-\!\alpha_0} u_{n-1} \!= \!\bar{\xi}\left[ \frac{1\!+\!\alpha_0\!+\!2\beta_0}{2(1\!-\!\alpha_0)} u_{n+1} \!+ \! \frac{1\!-\!3\alpha_0\!-\!4\beta_0}{2(1\!-\!\alpha_0)} u_{n} \!+ \! \frac{\beta_0}{1\!-\!\alpha_0} u_{n-1} \!\right]\!.
\end{align}
It can be  seen that \eqref{8.10} is simplified as $-1\le \alpha_0<1$ and $2\beta_0+\alpha_0\ge 0$, under which  the scheme \eqref{8.8} is $A-$stable.
Moreover, if taking $\alpha_0 = \frac{\lambda-1}{\lambda+1},$
$\beta_0 = \frac{1-\lambda+\delta}{2(1+\lambda)}, \beta_2 = \frac{1+\lambda+\delta}{2(1+\lambda)}$, then \eqref{8.8} is rewritten  into
\begin{align}   \label{8.9}
\frac{1+\lambda}{2}u_{n+1} - \lambda u_n + \frac{\lambda-1}{2} u_{n-1} = \bar{\xi}\left[ \frac{1+\lambda+\delta}{4} u_{n+1} + \frac{1-\delta}{2} u_{n} + \frac{1-\lambda+\delta}{4} u_{n-1} \right],  \end{align}
so that it  is $A-$stable for any $\lambda\ge 0$, $\delta \ge 0$.

$\bullet$ When $\beta_2 \neq \frac{1}{2}(1+\alpha_0)+\beta_0$, and $\alpha_0$ and $\beta_0$ are not zero simultaneously,    \eqref{8.2} is two-step but only first-order accurate. In this case,   the condition \eqref{8.10} can not be simplified.

When the scheme \eqref{8.2} is $A$-stable, it may not be algebraically stable.  For example, the scheme \eqref{8.9} with $\lambda \ge 0$ and $\delta > 0$ is shown to be  algebraically stable with the positive definite matrix
\[ \vec{G} =   \frac{1}{4}\!\left(\! \begin{array}{*{4}{c}}
(1+\lambda)^2+\delta & 1-\delta-\lambda^2 \\ 1-\delta-\lambda^2 & (\lambda-1)^2 + \delta \end{array} \!\right),    \]
see e.g. \cite{Dahlquist75B}, but it is not algebraically stable when $\lambda \ge 0$ and $\delta = 0$.
In fact, if \eqref{8.9} with $\lambda \ge 0$ and $\delta = 0$ is algebraically stable, then  Theorem 3.2 in \cite{Dahlquist75B} shows that corresponding matrix $\bar{\vec{G}}$  should satisfy
\[  (1,1)\bar{\vec{G}} = \frac{1}{2}\left(1\!+\!\lambda, 1\!-\!\lambda\right),~~~(1,0)\bar{\vec{G}} = \frac{1}{4} \left((1\!+\!\lambda)^2, 1\!-\!\lambda^2 \right),  \]
which uniquely gives
\[ \bar{\vec{G}} =   \frac{1}{4}\!\left(\! \begin{array}{*{4}{c}}
(1+\lambda)^2 & 1-\lambda^2 \\ 1-\lambda^2 & (\lambda-1)^2 \end{array} \!\right).  \]
Obviously,   $\bar{\vec{G}}$ is not positive definite so that  \eqref{8.9} with $\lambda \ge 0$ and $\delta = 0$ is not algebraically stable.
Thus, when $\beta_2 = \frac{1}{2}(1+\alpha_0)+\beta_0$, and $\alpha_0$ and $\beta_0$ are not zero simultaneously, \eqref{8.2} is algebraically stable for any $-1\le \alpha_0 < 1$ and $2\beta_0+\alpha_0 >0$, but is not  algebraically stable for $-1\le \alpha_0 < 1$ and $2\beta_0+\alpha_0 =0$.
 \hfill\qed

\section{Proof of Lemma \ref{lem3.1}}  \label{Appx2}
\setcounter{figure}{0}

This appendix proves Lemma \ref{lem3.1}, which   plays an important role for the modified-energy stability of the SAV-GL scheme \eqref{3.1.6}, whose time discretization is not necessarily algebraically stable.

The establishment of the identities  \eqref{3.1.7}-\eqref{3.1.8*}
in Lemma \ref{lem3.1}
 is motivated by  the identities in \cite{ShenJ18b,YangZ19},
and may be completed by using the method of undetermined
coefficients.
%
%
Suppose the parameters $\alpha_0, \beta_0$, and $\beta_2$ in \eqref{3.1.3}-\eqref{3.1.5} satisfy the condition \eqref{8.10}. In order to derive the modified-energy stability of our SAV-GL scheme \eqref{3.1.6}, we expect the following  identity
\begin{align}  \label{8.2.3}
 \big(\frac{1}{1\!-\!\alpha_0}\chi^{n+1} -& \frac{1\!+\!\alpha_0}{1\!-\!\alpha_0} \chi^n + \frac{\alpha_0}{1\!-\!\alpha_0}\chi^{n-1}\big) \big(\frac{\beta_2}{1\!-\!\alpha_0}\chi^{n+1} +\frac{\beta_1}{1\!-\!\alpha_0} \chi^n + \frac{\beta_0}{1\!-\!\alpha_0}\chi^{n-1}\big)    \nonumber \\
= \;& a \left[ \left(\chi^{n+1}\right)^2 -\left(\chi^{n}\right)^2\right]  + b \left[ \left(\chi^{n}\right)^2 -\left(\chi^{n-1}\right)^2\right] + d \left[ \chi^{n+1}\chi^{n} -\chi^{n}\chi^{n-1} \right]  \nonumber \\
& + \left( c_1\chi^{n+1} + c_2 \chi^n + c_3 \chi^{n-1}\right)^2,
\end{align}
where $a,b,d$ and $c_i, i=1,2,3$ are six undetermined real coefficients.  Expanding the term at the left hand side of \eqref{8.2.3}  and then comparing each coefficient with that at the right hand side
yield
\begin{align}  \label{8.2.4}
\begin{aligned}
& a  + c_1^2 = \frac{\beta_2}{(1-\alpha_0)^2},~~~b-a+c_2^2 = -\frac{(1+\alpha_0)\beta_1}{(1-\alpha_0)^2},  \\[1 \jot]
& c_3^2 - b = \frac{\alpha_0\beta_0}{(1-\alpha_0)^2}, ~~~
2c_1c_2 + d = \!\frac{\beta_1}{(1-\alpha_0)^2}- \frac{(1+\alpha_0)\beta_2}{(1-\alpha_0)^2},    \\[1 \jot]
& 2c_2c_3 -d =  \frac{\alpha_0\beta_1}{(1-\alpha_0)^2}- \frac{(1+\alpha_0)\beta_0}{(1-\alpha_0)^2},~~~2c_1c_3 =  \frac{\beta_0}{(1-\alpha_0)^2}+ \frac{\alpha_0\beta_2}{(1-\alpha_0)^2}.
\end{aligned}
\end{align}
Adding all six equations gives $ (c_1+c_2+c_3)^2 = 0$, which implies
\begin{align}   \label{8.2.5}
c_1+c_2+c_3 = 0.
\end{align}
The fourth and fifth equations in \eqref{8.2.4} may gives
\[  2c_2(c_1+c_3) \!=\! \frac{(1+\alpha_0)(\beta_1-\beta_0-\beta_2)}{(1-\alpha_0)^2}, \]
which is combined with \eqref{8.2.5} to give $c_2 \!=\! \pm \frac{\sqrt{2(1+\alpha_0)(\beta_0+\beta_2-\beta_1)}}{2(1-\alpha_0)}$. Note that $(1+\alpha_0)(\beta_0+\beta_2-\beta_1) = (1+\alpha_0)(2\beta_0+2\beta_2+\alpha_0-1) \ge 0$ when the condition \eqref{8.10} holds.  If substituting $c_2$ into \eqref{8.2.5} and combining it with the sixth equation in \eqref{8.2.4},
then it is obvious that $c_1$ and $c_3$ are two solutions of $x^2+c_2 x+\frac{\beta_0+\alpha_0\beta_2}{2(1-\alpha_0)^2} = 0$,
so that
\[  c_1 = -\frac{c_2}{2} + \sqrt{\frac{c_2^2}{4}-\frac{\beta_0\!+\!\alpha_0\beta_2}{2(1\!-\!\alpha_0)^2}} ~~~~
c_3 = -\frac{c_2}{2} - \sqrt{\frac{c_2^2}{4}-\frac{\beta_0\!+\!\alpha_0\beta_2}{2(1\!-\!\alpha_0)^2}},  \]
or
\[ c_1 = -\frac{c_2}{2} - \sqrt{\frac{c_2^2}{4}-\frac{\beta_0\!+\!\alpha_0\beta_2}{2(1\!-\!\alpha_0)^2}},~~~~
c_3 = -\frac{c_2}{2} + \sqrt{\frac{c_2^2}{4}-\frac{\beta_0\!+\!\alpha_0\beta_2}{2(1\!-\!\alpha_0)^2}}.  \]
We expect that the term $\frac{c_2^2}{4}-\frac{\beta_0+\alpha_0\beta_2}{2(1-\alpha)^2}$ is non-negative so that both $c_1$ and $c_3$  are real, and will  discuss
that in three cases below. If $\frac{c_2^2}{4}-\frac{\beta_0+\alpha_0\beta_2}{2(1-\alpha)^2}$ is non-negative, then inserting $c_1$ and $c_3$ into  the first, third and fifth equations in \eqref{8.2.4} yields
\[ a = \frac{2\beta_2\!+\!\beta_0 \!+\!\alpha_0\beta_2}{2(1\!-\!\alpha_0)^2} - \frac{c_2^2}{2}+ c_2 \sqrt{\frac{c_2^2}{4}-\frac{\beta_0\!+\!\alpha_0\beta_2}{2(1\!-\!\alpha_0)^2}}, \]
\[ b = -\frac{\alpha_0\beta_0\!+\!\beta_0 \!+\!\alpha_0\beta_2}{2(1\!-\!\alpha_0)^2} + \frac{c_2^2}{2} + c_2 \sqrt{\frac{c_2^2}{4}-\frac{\beta_0\!+\!\alpha_0\beta_2}{2(1\!-\!\alpha_0)^2}}, \]
\[ d = -\frac{\alpha_0\beta_1\!-\!(1\!+\!\alpha_0)\beta_0}{(1\!-\!\alpha_0)^2} - c_2^2-2c_2\sqrt{\frac{c_2^2}{4} -\frac{\beta_0\!+\!\alpha_0\beta_2}{2(1\!-\!\alpha_0)^2}}, \]
or
\[ a = \frac{2\beta_2\!+\!\beta_0\!+\!\alpha_0\beta_2}{2(1\!-\!\alpha_0)^2} - \frac{c_2^2}{2} - c_2 \sqrt{\frac{c_2^2}{4}-\frac{\beta_0\!+\!\alpha_0\beta_2}{2(1\!-\!\alpha_0)^2}}, \]
\[ b = -\frac{\alpha_0\beta_0\!+\!\beta_0 \!+\!\alpha_0\beta_2}{2(1\!-\!\alpha_0)^2} + \frac{c_2^2}{2} - c_2 \sqrt{\frac{c_2^2}{4}-\frac{\beta_0\!+\!\alpha_0\beta_2}{2(1\!-\!\alpha_0)^2}}, \]
\[ d = -\frac{\alpha_0\beta_1\!-\!(1\!+\!\alpha_0)\beta_0}{(1\!-\!\alpha_0)^2} - c_2^2+2c_2\sqrt{\frac{c_2^2}{4} -\frac{\beta_0\!+\!\alpha_0\beta_2}{2(1\!-\!\alpha_0)^2}}. \]
Those undetermined coefficients  can give the final identity \eqref{8.2.3}, which may be not unique.

Let us discuss when $\frac{c_2^2}{4}-\frac{\beta_0+\alpha_0\beta_2}{2(1-\alpha)^2}$ is non-negative. 

\noindent
$\bullet$ When $\alpha_0 = \beta_0 = 0$,  the condition \eqref{8.10} reduces to $\beta_2 \ge \frac{1}{2}$ so that   $\frac{c_2^2}{4}-\frac{\beta_0+\alpha_0\beta_2}{2(1-\alpha)^2} = \beta_2-\frac{1}{2}$ is non-negative and six undetermined coefficients reduce to
\[ c_2 = \pm\sqrt{\beta_2-\frac{1}{2}},~~ c_1 = 0,~~c_3 = \mp\sqrt{\beta_2-\frac{1}{2}},~~ a = \beta_2,~~ b = \beta_2-\frac{1}{2},~~d = 1 - 2\beta_2,   \]
or
\[ c_2 = \pm\sqrt{\beta_2-\frac{1}{2}},~~ c_1 = \mp\sqrt{\beta_2-\frac{1}{2}},~~c_3 = 0,~~ a = \frac{1}{2},~~ b = 0,~~d = 0.   \]
Therefore, when $\alpha_0 = \beta_0 = 0$ and $\beta_2\ge \frac{1}{2}$, the identity \eqref{8.2.3} becomes
\begin{align}   \label{8.2.6}
& \left(\chi^{n+1} - \chi^n \right) \left(\beta_2\chi^{n+1} + (1\!-\!\beta_2) \chi^n \right) =  \beta_2 \left[ \left(\chi^{n+1}\right)^2 -\left(\chi^{n}\right)^2\right] + \left(\beta_2-\frac{1}{2} \right) \left[ \left(\chi^{n}\right)^2 -\left(\chi^{n-1}\right)^2\right] \nonumber \\[1 \jot]
&\hspace{2cm} + \left( 1 - 2\beta_2 \right) \left[ \chi^{n+1}\chi^{n} -\chi^{n}\chi^{n-1} \right] + \left( \beta_2\!-\!\frac{1}{2}\right)\left( \chi^{n} - \chi^{n-1} \right)^2,
\end{align}
or
\begin{align}   \label{8.2.7}
\left(\chi^{n+1} - \chi^n \right) \left(\beta_2\chi^{n+1} + (1\!-\!\beta_2) \chi^n \right) =  \frac{1}{2} \left[ \left(\chi^{n+1}\right)^2 -\left(\chi^{n}\right)^2\right] + \left( \beta_2\!-\!\frac{1}{2}\right)\left( \chi^{n+1} - \chi^n \right)^2.
\end{align}
Both of them are equivalent to each other, and can be used to study the modified-energy stability of the SAV-GL scheme \eqref{3.1.6} with different energy inequalities by ignoring the last positive terms in \eqref{8.2.6} and \eqref{8.2.7}. 

\noindent
$\bullet$ When $\beta_2 = \frac{1+\alpha_0}{2}+\beta_0$, $\alpha_0$ and $\beta_0$ are not zero simultaneously, it can be checked  $ \frac{c_2^2}{4}-\frac{\beta_0+\alpha_0\beta_2}{2(1-\alpha)^2} = 0$ and the condition \eqref{8.10} reduces to $-1\le\alpha<1$, $2\beta_0+\alpha_0\ge 0$ so that six undetermined coefficients reduce to
\[ c_2  \!=\! \pm \frac{\sqrt{(1+\alpha_0)(2\beta_0+\alpha_0)}}{1-\alpha_0},~~ c_1 = c_3 = -\frac{c_2}{2},~~a = \frac{2+\alpha_0-\alpha_0^2 +  2\beta_0(1-\alpha_0) }{4(1-\alpha_0)^2},   \]
\[ b = \frac{\alpha_0+\alpha_0^2 +  2\beta_0(1-\alpha_0) }{4(1-\alpha_0)^2}, ~~ d = \frac{(\alpha_0-1)(2\beta_0+\alpha_0-1)-(\alpha_0+1)}{2(1-\alpha_0)^2}, \]
which  uniquely determine the identity
\begin{align}  \label{8.2.8}
& \left(\frac{1}{1\!-\!\alpha_0}\chi^{n+1} - \frac{1\!+\!\alpha_0}{1\!-\!\alpha_0} \chi^n + \frac{\alpha_0}{1\!-\!\alpha_0}\chi^{n-1}\right) \left(\frac{\beta_2}{1\!-\!\alpha_0}\chi^{n+1} +\frac{\beta_1}{1\!-\!\alpha_0} \chi^n + \frac{\beta_0}{1\!-\!\alpha_0}\chi^{n-1}\right)    \nonumber \\
= \;& \frac{2\!+\!\alpha_0\!-\!\alpha_0^2 \!+  \!2\beta_0(1\!-\!\alpha_0) }{4(1\!-\!\alpha_0)^2} \!\left[ \left(\chi^{n+1}\right)^2 \!-\!\left(\chi^{n}\right)^2\right] \!\! +\! \frac{\alpha_0\!+\!\alpha_0^2 \!+  \!2\beta_0(1\!-\!\alpha_0) }{4(1\!-\!\alpha_0)^2} \left[ \left(\chi^{n}\right)^2 \!-\!\left(\chi^{n-1}\right)^2\right] \nonumber \\
& +\! \frac{(\alpha_0\!-\!1)(2\beta_0\!+\!\alpha_0\!-\!1) \!-\!(\alpha_0\!+\!1)} {2(1\!-\!\alpha_0)^2} \!\left[ \chi^{n+1}\chi^{n} \!-\!\chi^{n}\chi^{n-1} \right]\!\!   \nonumber \\
& + \! \frac{(1\!+\!\alpha_0)(2\beta_0\!+\!\alpha_0)}{4(1\!-\!\alpha_0)^2} \!\left( \chi^{n+1} \!-\! 2 \chi^n \!+ \!\chi^{n-1}\right)^2.
\end{align}

\noindent
$\bullet$ When $\beta_2 \neq \frac{1+\alpha_0}{2}+\beta_0$ and $\alpha_0$ and $\beta_0$ are not zero simultaneously,
the condition \eqref{8.10} can not guarantee $\frac{c_2^2}{4}-\frac{\beta_0+\alpha_0\beta_2}{2(1-\alpha_0)^2}$ to be non-negative.
For this reason, 
we add a parameter constraint
\begin{align}   \label{8.2.9}
\frac{c_2^2}{4}-\frac{\beta_0\!+\!\alpha_0\beta_2}{2(1\!-\!\alpha_0)^2} = \frac{(1\!+\!\alpha_0)(2\beta_0\!+\!2\beta_2\!+\!\alpha_0\!-\!1) \!- \!4\beta_0 \!- \! 4\alpha_0\beta_2}{8(1\!-\!\alpha_0)^2} \ge 0,
\end{align}
which implies $\beta_2 \ge \frac{1+\alpha_0}{2} + \beta_0$.
Figure \ref{figC.0} (a) shows the region of the parameters $\alpha_0, \beta_0$ and $\beta_2$ satisfying  \eqref{8.10} and \eqref{8.2.9}.
Specifically, when $\beta_0 = 0$,  the conditions \eqref{8.10}, \eqref{8.2.9} and $\beta_2 \neq \frac{1+\alpha_0}{2}+\beta_0$ reduce to
\[\beta_2>0,~~~\beta_2\ge \frac{1-\alpha_0}{2},~~~\beta_2 > \frac{1+\alpha_0}{2},  \]
and the region of $\alpha_0$ and $\beta_2$ satisfying the above inequalities is shown in Figure \ref{figC.0} (b).
As an example, one chooses $\beta_0 = 0, \alpha_0 = \frac{1}{2}$ and $\beta_2 =1$, which locates in the region of green color. In that case, the values of six undetermined coefficients  are
\[ c_1 = \sqrt{2},~~c_2 = -\frac{3\sqrt{2}}{2},~~c_3 = \frac{\sqrt{2}}{2},~~a = 2,~~~b = \frac{1}{2},~~~d = -2,  \]
or
\[ c_1 = \frac{\sqrt{2}}{2},~~c_2 = -\frac{3\sqrt{2}}{2},~~c_3 = \sqrt{2},~~a = \frac{7}{2},~~~b = 2,~~~d = -5, \]
which can determine the following identities
\begin{align*}
& \left(2\chi^{n+1} - 3 \chi^n + \chi^{n-1}\right) \left(2\chi^{n+1} - \chi^n \right) = 2 \left[ \left(\chi^{n+1}\right)^2 -\left(\chi^{n}\right)^2\right]  \\
+ \;&   \frac{1}{2} \left[ \left(\chi^{n}\right)^2 -\left(\chi^{n-1}\right)^2\right] -2 \left[ \chi^{n+1}\chi^{n} -\chi^{n}\chi^{n-1} \right]
+ \left( \sqrt{2}\chi^{n+1} -\frac{3\sqrt{2}}{2}  \chi^n + \frac{\sqrt{2}}{2} \chi^{n-1}\right)^2,
\end{align*}
and
\begin{align*}
& \left(2\chi^{n+1} - 3 \chi^n + \chi^{n-1}\right) \left(2\chi^{n+1} -  \chi^n \right) = \frac{7}{2} \left[ \left(\chi^{n+1}\right)^2 -\left(\chi^{n}\right)^2\right]    \\
+\;&  2 \left[ \left(\chi^{n}\right)^2 -\left(\chi^{n-1}\right)^2\right] -5 \left[ \chi^{n+1}\chi^{n} -\chi^{n}\chi^{n-1} \right]
+ \left( \frac{\sqrt{2}}{2}\chi^{n+1} -\frac{3\sqrt{2}}{2} \chi^n + \sqrt{2} \chi^{n-1}\right)^2.
\end{align*}
However, when taking $\beta_0 = 0$,  $\alpha_0 = \frac{1}{2}$ and $\beta_2 = \frac{2}{3}$,  the condition \eqref{8.10} holds but \eqref{8.2.9} does not hold, so that one can not obtain six undetermined real coefficients in \eqref{8.2.3}.
In summary, when $\beta_2 \neq \frac{1+\alpha_0}{2}+\beta_0$ and $\alpha_0$ and $\beta_0$ are not zero simultaneously, under the conditions \eqref{8.10} and \eqref{8.2.9},
the identity \eqref{8.2.8}  can be derived
as follows
\begin{align}   \label{8.2.10}
& \left(\frac{1}{1\!-\!\alpha_0}\chi^{n+1} - \frac{1\!+\!\alpha_0}{1\!-\!\alpha_0} \chi^n + \frac{\alpha_0}{1\!-\!\alpha_0}\chi^{n-1}\right) \left(\frac{\beta_2}{1\!-\!\alpha_0}\chi^{n+1} +\frac{\beta_1}{1\!-\!\alpha_0} \chi^n + \frac{\beta_0}{1\!-\!\alpha_0}\chi^{n-1}\right)    \nonumber \\
= \;& \!\!\left[\frac{1\!-\!\alpha_0^2 \!+  \!2\beta_2\!-\!2\alpha_0\beta_0}{4(1\!-\!\alpha_0)^2} \!+\! c\tilde{c}\right]\!\left[ \left(\chi^{n+1}\right)^2 \!-\!\left(\chi^{n}\right)^2\right] \!\! +\!\! \left[\frac{2\beta_2\!+\!\alpha_0^2\!-\!1}{4(1\!-\!\alpha_0)^2} \!+ \!c\tilde{c}\right] \!\left[ \left(\chi^{n}\right)^2 \!-\!\left(\chi^{n-1}\right)^2\right] \nonumber \\
& \!\!+\! \!\left[\frac{1}{2}\!+\!\frac{\alpha_0\beta_0\!-\!\beta_2}{(1\!-\!\alpha_0)^2} \!-\!2c \tilde{c}\!\right] \!\left[ \chi^{n+1}\chi^{n} \!-\!\chi^{n}\chi^{n-1} \right]\!\!   + \!\!\left[\!\left(\!c\!-\!\frac{\tilde{c}}{2}\right) \!\chi^{n+1} \!+\!\tilde{c} \chi^n \!-\! \left(\!c\!+\!\frac{\tilde{c}}{2}\right) \!\chi^{n-1}\right]^2\!,
\end{align}
or
\begin{align}   \label{8.2.11}
& \left(\frac{1}{1\!-\!\alpha_0}\chi^{n+1} - \frac{1\!+\!\alpha_0}{1\!-\!\alpha_0} \chi^n + \frac{\alpha_0}{1\!-\!\alpha_0}\chi^{n-1}\right) \left(\frac{\beta_2}{1\!-\!\alpha_0}\chi^{n+1} +\frac{\beta_1}{1\!-\!\alpha_0} \chi^n + \frac{\beta_0}{1\!-\!\alpha_0}\chi^{n-1}\right)    \nonumber \\
= \;& \!\!\left[\frac{1\!-\!\alpha_0^2 \!+  \!2\beta_2\!-\!2\alpha_0\beta_0}{4(1\!-\!\alpha_0)^2} \!-\! c\tilde{c}\right]\!\left[ \left(\chi^{n+1}\right)^2 \!-\!\left(\chi^{n}\right)^2\right] \!\! +\!\! \left[\frac{2\beta_2\!+\!\alpha_0^2\!-\!1}{4(1\!-\!\alpha_0)^2} \!- \!c\tilde{c}\right] \!\left[ \left(\chi^{n}\right)^2 \!-\!\left(\chi^{n-1}\right)^2\right] \nonumber \\
& +\! \!\left[\frac{1}{2}\!+\!\frac{\alpha_0\beta_0\!-\!\beta_2}{(1\!-\!\alpha_0)^2} \!+\!2c\tilde{c}\!\right] \!\left[ \chi^{n+1}\chi^{n} \!-\!\chi^{n}\chi^{n-1} \right]\!\!   + \!\!\left[\!-\!\!\left(\!c\!+\!\frac{\tilde{c}}{2}\!\right) \!\chi^{n+1} \!+\!\tilde{c} \chi^n \!+\!\! \left(\!c\!-\!\frac{\tilde{c}}{2}\!\right) \!\chi^{n-1}\!\right]^2\!,
\end{align}
where
\[ c = \sqrt{\frac{2\beta_2\!-\!2\beta_0\!-\!\alpha_0\!-\!1}{8(1-\alpha_0)}},~~~\tilde{c} = - \frac{\sqrt{2(1\!+\!\alpha_0)(2\beta_0\!+\!2\beta_2\!+\!\alpha_0\!-\!1)}} {2(1\!-\!\alpha_0)}.    \]
 \hfill\qed

\begin{figure}
\begin{minipage}{0.48\linewidth}
  \centerline{\includegraphics[width=7cm,height=5cm]{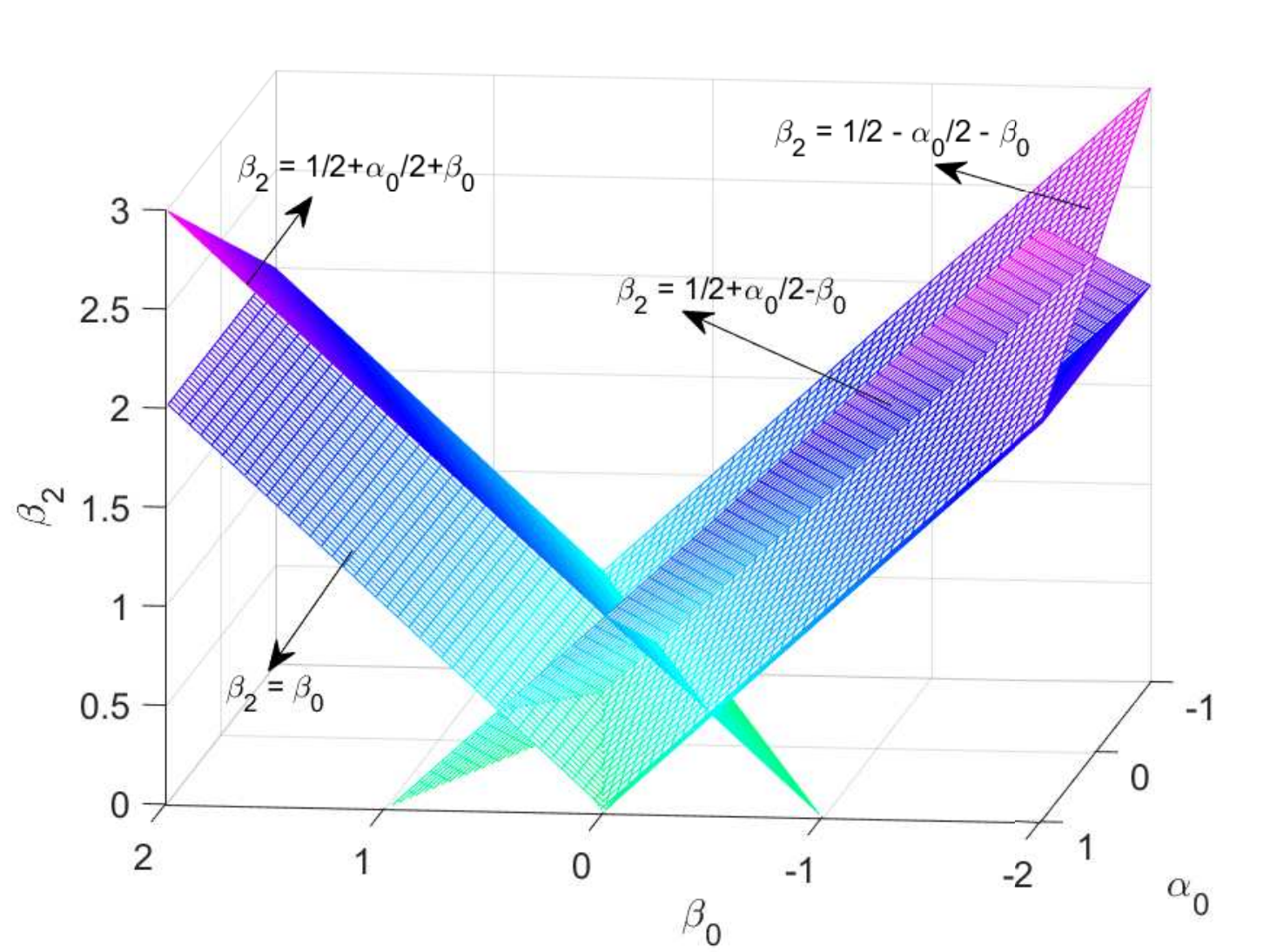}}
  \centerline{\scriptsize (a) Conditions  \eqref{8.10} and \eqref{8.2.9}.}
\end{minipage}
\hfill
\begin{minipage}{0.48\linewidth}
  \centerline{\includegraphics[width=7cm,height=5cm]{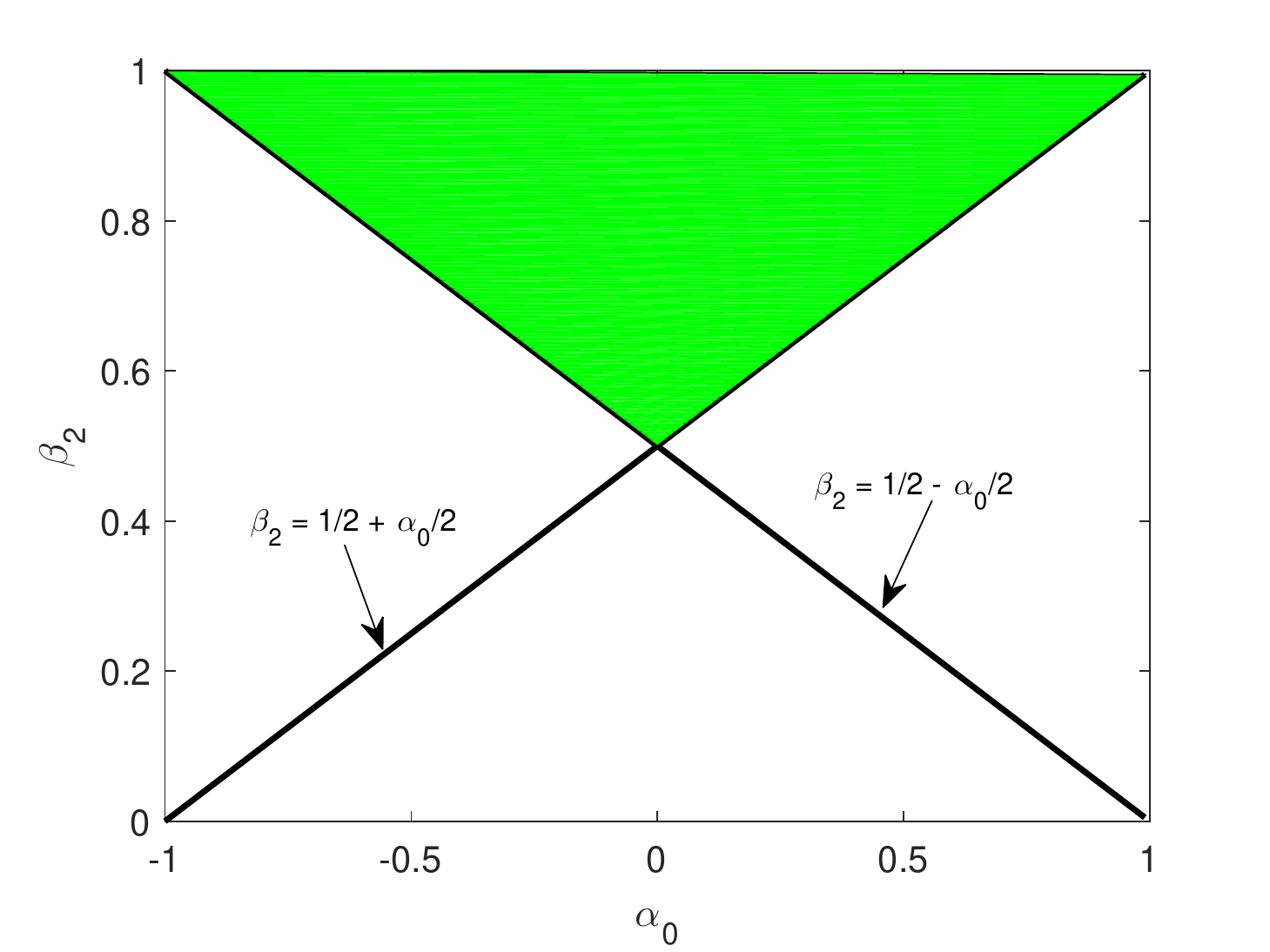}}
  \centerline{\scriptsize (b) Conditions  \eqref{8.10}, \eqref{8.2.9},   and $\beta_0 = 0$.}
\end{minipage}
\caption{The region of the parameters $\alpha_0, \beta_0$, and $\beta_2$. }
\label{figC.0}
\end{figure}

\section{De-aliasing in FFT by zero-padding}  \label{sec:5}

This appendix introduces the de-aliasing by zero-padding
for the cubic term
when  the Fourier pseudo-spectral method is used for the spatial discretization of the semi-discrete-in-time SAV-GL scheme \eqref{3.1.6}
in  our numerical experiments on the Allen-Cahn, the Cahn-Hilliard and the phase field crystal models in Section \ref{sec:6}.


Let $N$ be an even integer and $\widehat{\vec{u}} = (\widehat{u}_{k,l})_{N\times N}$ be the discrete Fourier coefficients  of $\vec{u}=(u_{i,j})_{N\times N}$, and  define $\vec{w} = (w_{i,j})_{N\times N} $ with $w_{i,j} = u_{i,j}^3$.  Suppose $\widehat{\vec{w}} = (\widehat{w}_{k,l})_{N\times N}$ is the discrete Fourier coefficients  of $\vec{w}$, then a simple calculation shows that
\begin{align}  \label{5.1}
& \widehat{w}_{k,l} = \frac{1}{N^4}\sum_{(m,n),(p,q)\in\widehat{\mathbb{S}}_N}  \widehat{u}_{m,n}\widehat{u}_{p,q} \widehat{u}_{k-m-p,l-n-q}  \nonumber  \\
=\; & \frac{1}{N^4}\underset{(k-m-p,l-n-q)\in\widehat{\mathbb{S}}_N} {\sum_{(m,n),(p,q)\in\widehat{\mathbb{S}}_N}}  \widehat{u}_{m,n} \widehat{u}_{p,q} \widehat{u}_{k-m-p,l-n-q} + \frac{1}{N^4}\underset{(k-m-p,l-n-q)\notin\widehat{\mathbb{S}}_N} {\sum_{(m,n),(p,q)\in\widehat{\mathbb{S}}_N}}  \widehat{u}_{m,n} \widehat{u}_{p,q} \widehat{u}_{k-m-p,l-n-q}.
\end{align}
{The second summation} on the right hand side of \eqref{5.1} is called {\em the aliasing error}, and it can be observed that the modes with wave number $k-m-p > \frac{N}{2}$ or $l-n-q > \frac{N}{2}$ are aliased to those with $k-m-p-N$ or $l-n-q-N$, while the modes with wave number
$k-m-p < -\frac{N}{2}+1$ or $l-n-q < -\frac{N}{2}+1$ are aliased to those with $k-m-p+N$ or $l-n-q+N$.

The importance of eliminating the aliasing errors, called {\em de-aliasing}, has been studied by Orszag \cite{Orszag71}.
Here, we consider the zero-padding, see e.g. \cite[\S3.4.2]{CHQZ1988},
whose main idea is to use the discrete inverse Fourier transform for
$\breve{\vec{u}} = (\breve{u}_{k,l})_{K\times K}$ instead of $\widehat{\vec{u}} = (\widehat{u}_{m,n})_{N\times N}$, where $K>N$ is an undetermined number,
and $\breve{\vec{u}}$ is defined 
by zero padding  as follows
\begin{align*}
\breve{u}_{k,l} =
\begin{cases}
\widehat{u}_{k,l}, ~~~ (k,l) \in \widehat{\mathbb{S}}_N, \\
0,~~~~~~ \mbox{otherwise}.
\end{cases}
\end{align*}
If letting $\widetilde{\vec{u}} = (\widetilde{u}_{i,j})_{K\times K}$ be   the inverse Fourier transform of $\breve{\vec{u}}$,  defining $\widetilde{\vec{w}} = (\widetilde{w}_{i,j})_{K\times K} $ with $\widetilde{w}_{i,j} = \widetilde{u}_{i,j}^3$,
and computing the discrete Fourier coefficients of $\widetilde{\vec{w}}$  by
\begin{align}  \label{5.2}
\breve{w}_{k,l} := \frac{1}{K^4}\!\!\underset{(k-m-p,l-n-q)\in\widehat{\mathbb{S}}_K} {\sum_{(m,n),(p,q)\in\widehat{\mathbb{S}}_K}} \!\!\! \breve{u}_{m,n} \breve{u}_{p,q} \breve{u}_{k-m-p,l-n-q} + \frac{1}{K^4}\!\!\underset{(k-m-p,l-n-q)\notin\widehat{\mathbb{S}}_K} {\sum_{(m,n),(p,q)\in\widehat{\mathbb{S}}_K}}  \!\!\!\breve{u}_{m,n} \breve{u}_{p,q} \breve{u}_{k-m-p,l-n-q},
\end{align}
then one can choose the smallest $K>N$ such that the second summation on the right-hand side of \eqref{5.2} vanishes for $(k,l) \in \widehat{\mathbb{S}}_N$, and then the de-aliased discrete Fourier coefficients of $\vec{w}=(u_{i,j}^3)$ are derived by
\[  \widehat{w}_{k,l}^{\text{\tiny De}} = \left(\frac{K}{N}\right)^4 \breve{w}_{k,l}, ~~~ (k,l)\in \widehat{\mathbb{S}}_N. \]
It can be observed that  the de-aliased coefficients $\widehat{w}_{k,l}^{\text{\tiny De}},~(k,l)\in \widehat{\mathbb{S}}_N$ is equivalent to the first summation on the right hand side of \eqref{5.1}.

The remaining issue is how to determine $K$. In order to make the second summation on the right-hand side of \eqref{5.2} to be zero, one needs $\breve{u}_{m,n} \breve{u}_{p,q} \breve{u}_{k-m-p,l-n-q} = 0$ for any $(m,n),(p,q)\in \widehat{\mathbb{S}}_{K}$ and $(k-m-p,l-n-q)\notin\widehat{\mathbb{S}}_K$.
Let $\widehat{\mathbb{S}}_{KN}=\{(k,l) \in\mathbb{Z}^2 | (k,l)\in \widehat{\mathbb{S}}_{K} ~\mbox{and}~(k,l)\notin\widehat{\mathbb{S}}_{N}\}$.
It is obvious that $\breve{u}_{m,n} \breve{u}_{p,q} \breve{u}_{k-m-p,l-n-q} = 0$ for $(m,n)\in \widehat{\mathbb{S}}_{KN}$ or $(p,q)\in \widehat{\mathbb{S}}_{KN}$. Hence, one only needs to consider the indexes $(m,n)\in \widehat{\mathbb{S}}_{N}$ and $(p,q)\in \widehat{\mathbb{S}}_{N}$.
In that case, the modes $\breve{u}_{m,n} $ and $ \breve{u}_{p,q}$ usually are not zero so that it requires $\breve{u}_{k-m-p,l-n-q} = 0$ for $(k-m-p,l-n-q)\notin\widehat{\mathbb{S}}_K$.
Consequently, when the wave number $k-m-p>\frac{K}{2}$ or $l-n-q>\frac{K}{2}$, one needs $k-m-p-K<-\frac{N}{2} + 1$ and $l-n-q-K<-\frac{N}{2} + 1$, since the modes with $k-m-p>\frac{K}{2}$ or $l-n-q>\frac{K}{2}$ are aliased to those with $k-m-p-K$ or $l-n-q-K$.
The largest possible value of $k-m-p$ and $l-n-q$ is $\frac{3}{2}N-2$, and thus the inequality $\frac{3}{2}N-2-K<-\frac{N}{2} +1$ gives $K>2N -3$.
In a similar way, when the wave number $k-m-p< -\frac{K}{2}+1$ or $l-n-q< -\frac{K}{2}+1$, it requires $k-m-p+K>\frac{N}{2}$ and $l-n-q+K>\frac{N}{2}$ such that the modes with those wave numbers are zero. Since the smallest possible value of $k-m-p$ and $l-n-q$ is $-\frac{3}{2}N+1$,  one can deduce $K>2N-1$.
In summary, one can take $K = 2N$ in actual applications, and
the de-aliased discrete Fourier coefficients of ${\vec{w}=(u_{i,j}^3)}$ with  zero padding are computed as follows:
\begin{description}
\item[(1)] For given 2D vector $\vec{u}=(u_{i,j})_{N\times N}$, compute the discrete Fourier coefficients $\widehat{\vec{u}}$ by the FFT;
\item[(2)] Extend $\widehat{\vec{u}}$ to $\breve{\vec{u}}$ by zero padding with $K=2N$,
 perform the inverse Fourier transform  of $\breve{\vec{u}}$ to derive $\widetilde{\vec{u}}$ by the inverse FFT, and then compute $\widetilde{\vec{w}}=(\widetilde{u}_{i,j}^3)_{K\times K}$;
\item[(3)] Compute the discrete Fourier coefficients $\breve{\vec{w}}$ of $\widetilde{\vec{w}}$ by the FFT,  then multiply a scaling factor $\left(\frac{K}{N}\right)^4$ and drop the extra wave numbers to  obtain $\widehat{\vec{w}}^{\text{\tiny De}}$, the de-aliased discrete Fourier coefficients of $\vec{w}=(u_{i,j}^3)$.
\end{description}

Several numerical examples in Section \ref{sec:6} will be given to demonstrate the effectiveness  of the above de-aliasing procedure.
Moreover, such  de-aliasing by zero-padding  can be easily extended to a general polynomial nonlinear term $u^p$, $p\ge 3$, by setting $K = \frac{p+1}{2}N$ and the scaling factor in step (3) as $\left(\frac{K}{N}\right)^{2(p-1)}$, where  $2$ in the exponent is the spatial dimension.

\section{Estimating the time stepsize
for the SAV-GL scheme}
\label{Appx3}\setcounter{figure}{0}

This appendix estimates the time stepsize
of the SAV-GL scheme \eqref{3.1.6} with the Fourier pseudo-spectral spatial discretization
with the help of the following test equation
\begin{align}   \label{8.3.1}
u'(t)  = \xi u(t)  + \zeta u(t),
\end{align}
where $\xi<0$, $|\zeta|\le |\xi|$.
Applying \eqref{3.1.3}-\eqref{3.1.5} to \eqref{8.3.1} yields the semi-implicit scheme
\begin{align}   \label{8.3.2}
& \frac{1}{1\!-\!\alpha_0} u_{n+1} - \frac{1\!+\!\alpha_0}{1\!-\!\alpha_0} u_n + \frac{\alpha_0}{1\!-\!\alpha_0} u_{n-1} = \bar{\xi}\left[ \frac{\beta_2}{1\!-\!\alpha_0} u_{n+1} +\frac{\beta_1}{1\!-\!\alpha_0} u_{n} + \frac{\beta_0}{1\!-\!\alpha_0} u_{n-1} \right] \nonumber \\
& \hspace{2cm}+ \bar{\zeta} \left[ \frac{1\!-\!\alpha_0\!+\!\beta_2\!-\!\beta_0}{1\!-\!\alpha_0} u_n - \frac{\beta_2\!-\!\beta_0}{1\!-\!\alpha_0} u_{n-1}\right],
\end{align}
where $\bar{\xi} = \xi \tau$, $\bar{\zeta} = \zeta \tau$, the parameters $\alpha_0, \beta_0$ and $\beta_2$ are assumed to satisfy \eqref{8.10}.
It is known that  \eqref{8.3.2} is stable iff all roots of the characteristic polynomial defined by
\begin{align*}
\mathbb{Q}(x) \!=\! \!\left[\frac{1}{1\!-\!\alpha_0} \!- \! \frac{\beta_2}{1\!-\!\alpha_0}\bar{\xi} \right]\!x^2 \!-\!  \!\left[\frac{1\!+\!\alpha_0}{1\!-\!\alpha_0} \!+\! \frac{\beta_1}{1\!-\!\alpha_0} \bar{\xi} \!+\!  \!\frac{1\!-\!\alpha_0\!+\!\beta_2\!-\!\beta_0}{1\!-\!\alpha_0}\! \bar{\zeta}\right]\!x \!+  \!\frac{\alpha_0}{1\!-\!\alpha_0} \!- \! \frac{\beta_0}{1\!-\!\alpha_0} \bar{\xi} \!+\! \frac{\beta_2\!-\!\beta_0}{1\!-\!\alpha_0} \bar{\zeta}.
\end{align*}
are smaller or equal to one in modulus.
In order to make sure the roots of $\mathbb{Q}(x)$ are smaller or equal to one in modulus, one requires
\begin{align*}  
\begin{aligned}
& \left| \alpha_0 - \beta_0 \bar{\xi} + \left(\beta_2-\beta_0\right) \bar{\zeta} \right| \le \left| 1 -  \beta_2\bar{\xi} \right|, \\
& \left|1+\alpha_0 + \beta_1 \bar{\xi} +  \left(1-\alpha_0+\beta_2-\beta_0\right) \bar{\zeta} \right| \le  \left| 1+\alpha_0 - \left(\beta_0+\beta_2\right) \bar{\xi} +  \left(\beta_2-\beta_0\right) \bar{\zeta} \right|,
\end{aligned}
\end{align*}
which is equivalent to
\begin{align}   \label{8.3.4}
\begin{aligned}
& 1\!+\!\alpha_0\!-\!\left( \beta_0\!+\!\beta_2\right)\bar{\xi} \!+ \!\left( \beta_2\!-\!\beta_0\right)\bar{\zeta} \ge 0,~~1\!-\!\alpha_0\!+\!\left( \beta_0\!-\!\beta_2\right)\bar{\xi} \!-\! \left( \beta_2\!-\!\beta_0\right)\bar{\zeta} \ge 0, \\[1 \jot]
& 2(1\!+\!\alpha_0)\!+\!\left( \beta_1\!-\!\beta_0\!-\!\beta_2\right)\bar{\xi} \!+ \!\left( 1\!-\!\alpha_0\!+\!2\beta_2\!-\!2\beta_0\right)\bar{\zeta} \!\ge \! 0,~\!-\!\left( \beta_1\!+\!\beta_0\!+\!\beta_2\right)\bar{\xi} \!- \!\left( 1\!-\!\alpha_0\right)\bar{\zeta} \!\ge\! 0.
\end{aligned}
\end{align}
Thus, 
 the boundary of the  stability region  of    \eqref{8.3.2}  can be represented by the curves $1\!+\!\alpha_0\!-\!\left( \beta_0\!+\!\beta_2\right)\bar{\xi} \!+ \!\left( \beta_2\!-\!\beta_0\right)\bar{\zeta} = 0$, $1\!-\!\alpha_0\!+\!\left( \beta_0\!-\!\beta_2\right)\bar{\xi} \!-\! \left( \beta_2\!-\!\beta_0\right)\bar{\zeta} = 0$, $2(1\!+\!\alpha_0)\!+\!\left( \beta_1\!-\!\beta_0\!-\!\beta_2\right)\bar{\xi} \!+ \!\left( 1\!-\!\alpha_0\!+\!2\beta_2\!-\!2\beta_0\right)\bar{\zeta} = 0$ and $-\!\left( \beta_1\!+\!\beta_0\!+\!\beta_2\right)\bar{\xi} \!- \!\left( 1\!-\!\alpha_0\right)\bar{\zeta} = 0$.

Next,  we discuss two special cases.

\noindent
$\bullet$ When $\alpha_0 = \beta_0 = 0$ and $\beta_2\ge \frac{1}{2}$, the condition \eqref{8.3.4} reduces to
\begin{align}   \label{8.3.5}
1\!-\!\beta_2\bar{\xi} \!+ \! \beta_2\bar{\zeta} \ge 0,~~~1\!-\!\beta_2\bar{\xi} \!-\!\beta_2 \bar{\zeta} \ge 0,~~~2\!+\!\left( 1\!-\!2\beta_2\right)\bar{\xi} \!+ \!\left( 1\!+\!2\beta_2\right)\bar{\zeta} \!\ge \! 0,~~~\bar{\xi} +\bar{\zeta} \!\le\! 0.
\end{align}
Since $\xi<0$ and $|\zeta|\le |\xi|$,  the latter two inequalities
imply the first two inequalities in \eqref{8.3.5}, so that the boundary of the  stability regions of  \eqref{8.3.2}  is determined by the curves $2\!+\!\left( 1\!-\!2\beta_2\right)\bar{\xi} \!+ \!\left( 1\!+\!2\beta_2\right)\bar{\zeta} = 0$ and $\bar{\xi} + \bar{\zeta} = 0$. Figure \ref{figC.1}  gives the stability
regions   of  \eqref{8.3.2} with $(\alpha_0,\beta_0,\beta_2) = (0,0,1)$ and $(0,0,2)$.  One can deduce  that
 the scheme \eqref{8.3.2} is unconditionally stable when $\xi< 0$ and $\frac{2\beta_2-1}{2\beta_2+1}\xi \le \zeta < |\xi|$,
and  is stable under the time stepsize condition $\tau < \frac{2}{(2\beta_2-1) \xi - (2\beta_2+1)\zeta}$ when $\xi< 0$ and $\zeta < \frac{2\beta_2-1}{2\beta_2+1}\xi$.

\begin{figure}
\begin{minipage}{0.48\linewidth}
  \centerline{\includegraphics[width=7cm,height=5cm]{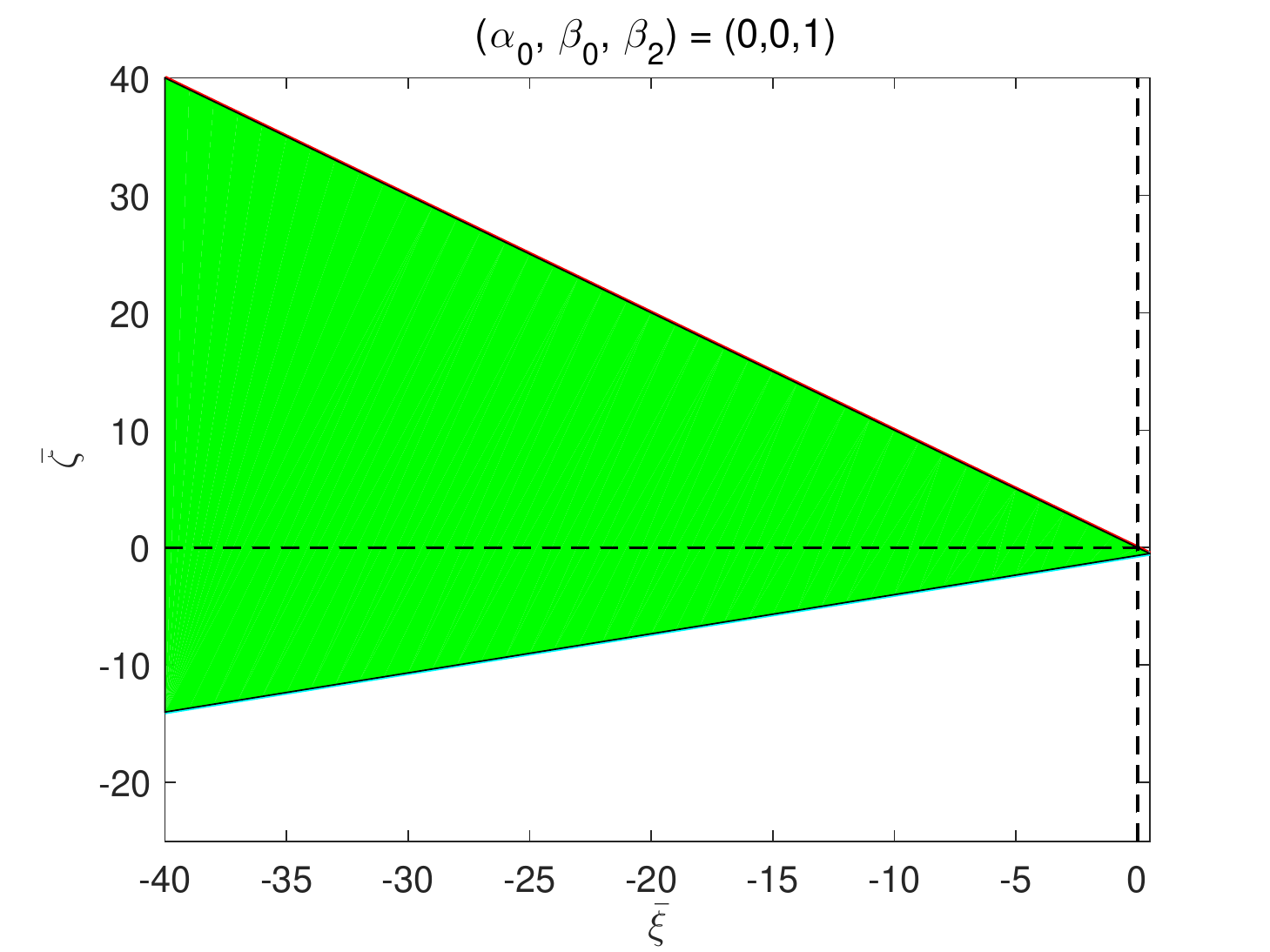}}
\end{minipage}
\hfill
\begin{minipage}{0.48\linewidth}
  \centerline{\includegraphics[width=7cm,height=5cm]{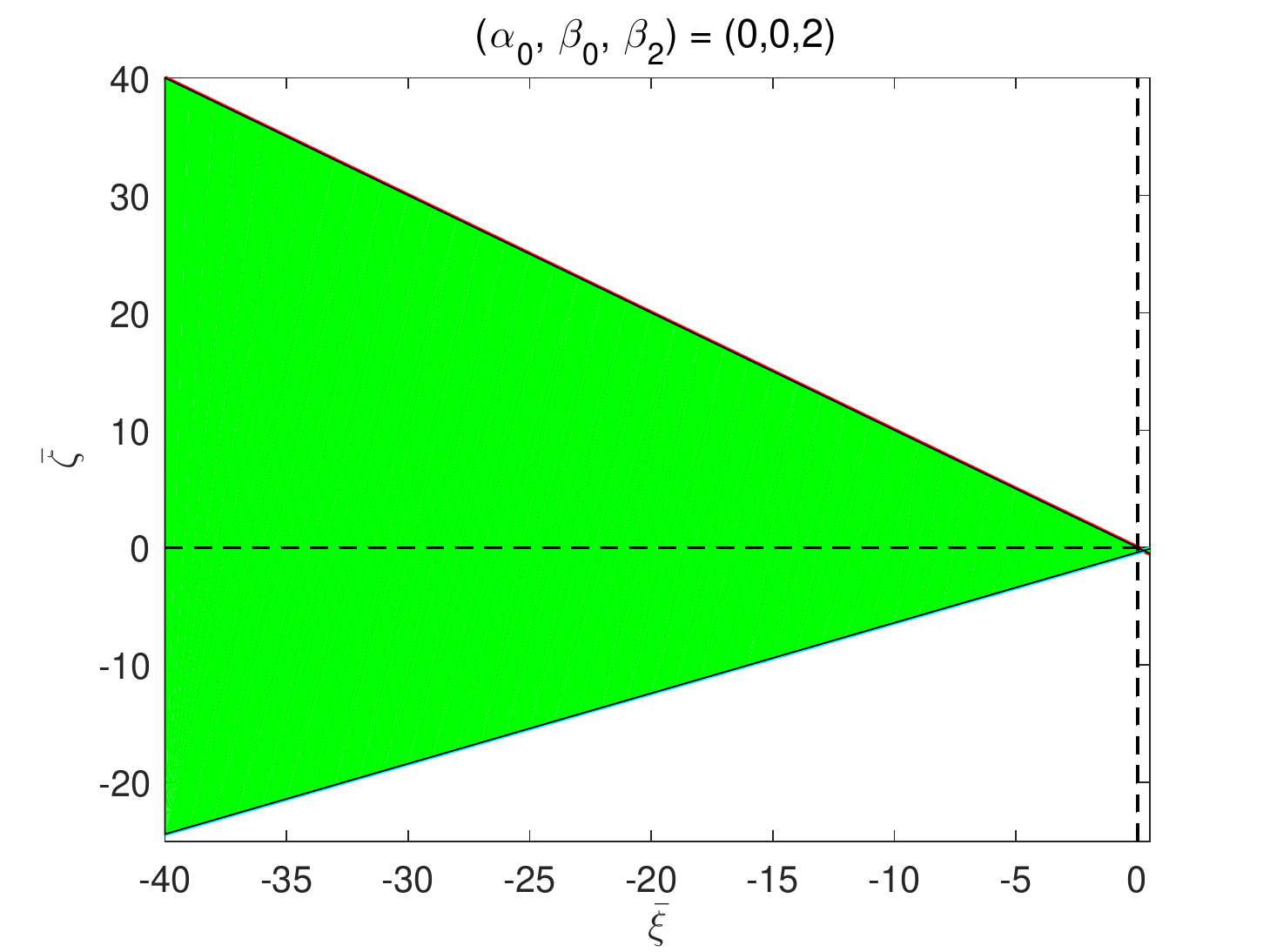}}
\end{minipage}
\caption{Stability regions (in green) of \eqref{8.3.2} for $(\alpha_0,\beta_0,\beta_2) = (0,0,1)$ and $(0,0,2)$. }
\label{figC.1}
\end{figure}

\noindent
$\bullet$ When $\beta_2 = \frac{1+\alpha_0}{2}+\beta_0$,  $\alpha_0$ and $\beta_0$ are not zero  simultaneously, and  $-1\le \alpha_0<1$ and $2\beta_2+\alpha_0\ge 0$, the condition \eqref{8.3.4} reduces to
\begin{align}  \label{8.3.6}
\begin{aligned}
& 2(1\!+\!\alpha_0) - (4\beta_0\!+\!\alpha_0\!+\!1) \bar{\xi} + (1\!+\!\alpha_0) \bar{\zeta} \ge 0,~~~ 2(1\!-\!\alpha_0)-(1\!+\!\alpha_0) \bar{\xi} - (1\!+\!\alpha_0)\bar{\zeta} \ge 0, \\[1 \jot]
& 1+\alpha_0 - (2\beta_0\!+\!\alpha_0) \bar{\xi} + \bar{\zeta}  \ge 0,~~~ \bar{\xi} + \bar{\zeta} \le  0.
\end{aligned}
\end{align}
Since $\xi<0$ and $|\zeta|\le |\xi|$, a direct check shows that the boundary of the stability region  of \eqref{8.3.2}  can be represented only by the curves $1+\alpha_0 - (2\beta_0\!+\!\alpha_0) \bar{\xi} + \bar{\zeta} = 0$ and $\bar{\xi} + \bar{\zeta} = 0$.
Figure \ref{figC.2} gives the stability
regions   of \eqref{8.3.2} with $(\alpha_0,\beta_0,\beta_2) = (-1/3,5/12,3/4), (1/3,0,1), (-1/3,1/6,1/2)$ and $(1/3,-1/6,1/2)$, from which one can see that the  stability region  of \eqref{8.3.2} with $2\beta_0+\alpha_0\neq 0$ is much larger than that with $2\beta_0+\alpha_0 = 0$, so that the scheme \eqref{8.3.2} with $2\beta_0+\alpha_0 \neq 0$ possesses better stability properties.
More specifically, for \eqref{8.3.2} with $2\beta_0+\alpha_0 = 0$, the upper and lower boundaries of the stability region are determined by the curves $\bar{\zeta} = - \bar{\xi}$ and $\bar{\zeta} = -(1+\alpha_0)$, respectively. Therefore,  \eqref{8.3.2} with $2\beta_0+\alpha_0  = 0$ is unconditionally stable when  $\xi< 0$ and $0 \le \zeta < |\xi|$, and  is stable under the time stepsize condition $\tau < -\frac{1+\alpha_0}{\zeta}$ when  $\xi< 0$ and $ \zeta < 0$. 
For \eqref{8.3.2} with $2\beta_0+\alpha_0 \neq 0$, the upper and lower boundaries of the stability region are the curves $\bar{\zeta} = -\bar{\xi}$ and $\bar{\zeta} = (2\beta_0\!+\!\alpha_0) \bar{\xi} - (1\!+\!\alpha_0)$, respectively. Therefore,  \eqref{8.3.2} with $2\beta_0+\alpha_0 \neq 0$ is unconditionally stable when $\xi< 0$ and $(2\beta_0\!+\!\alpha_0)\xi \le \zeta < |\xi|$, and
is stable under the condition $\tau < \frac{1+\alpha_0}{(2\beta_0+\alpha_0) \xi - \zeta}$ when $\xi< 0$ and $\zeta < (2\beta_0\!+\!\alpha_0)\xi$.

\begin{figure}
\begin{minipage}{0.48\linewidth}
  \centerline{\includegraphics[width=7cm,height=5cm]{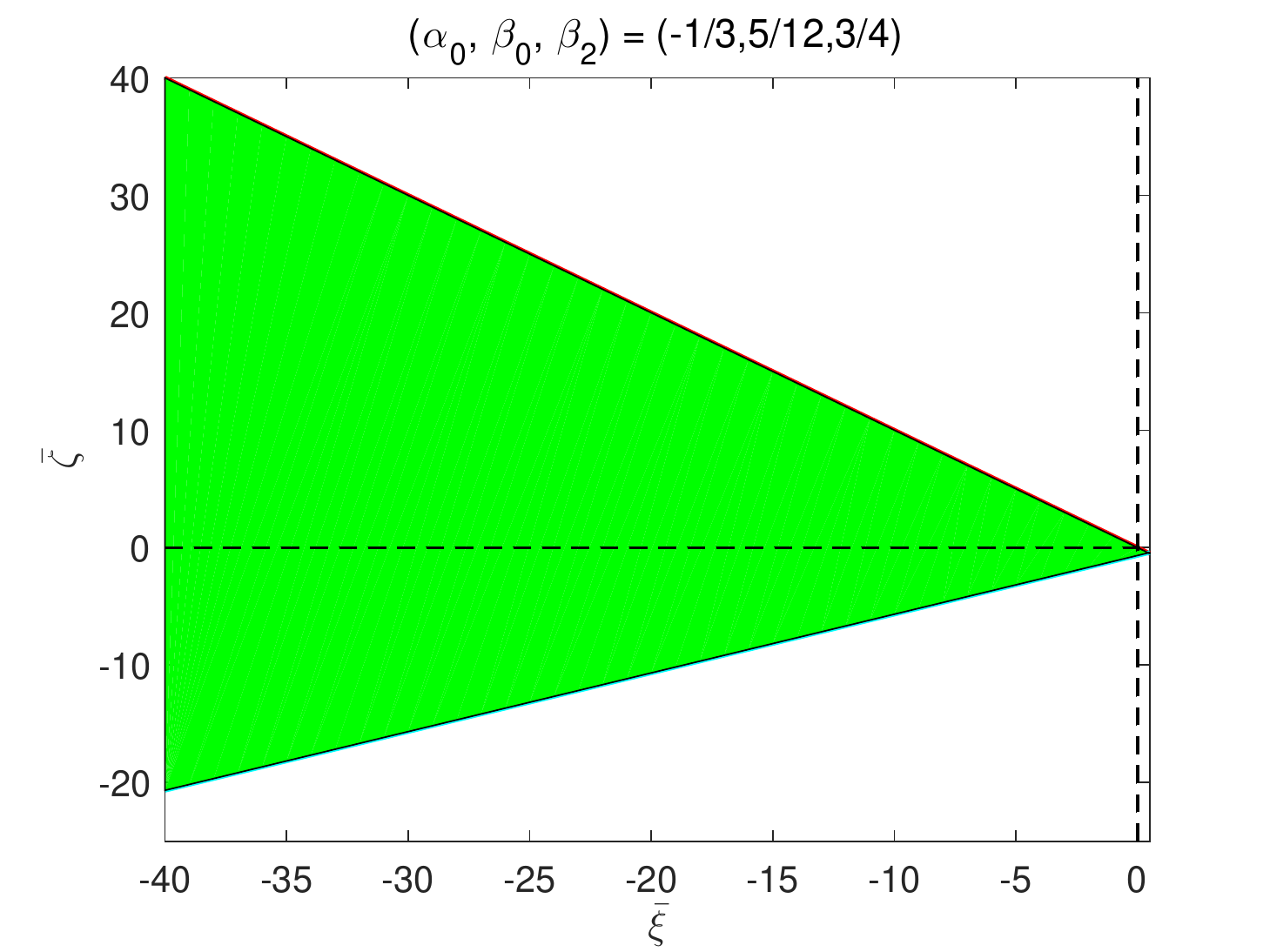}}
\end{minipage}
\hfill
\begin{minipage}{0.48\linewidth}
  \centerline{\includegraphics[width=7cm,height=5cm]{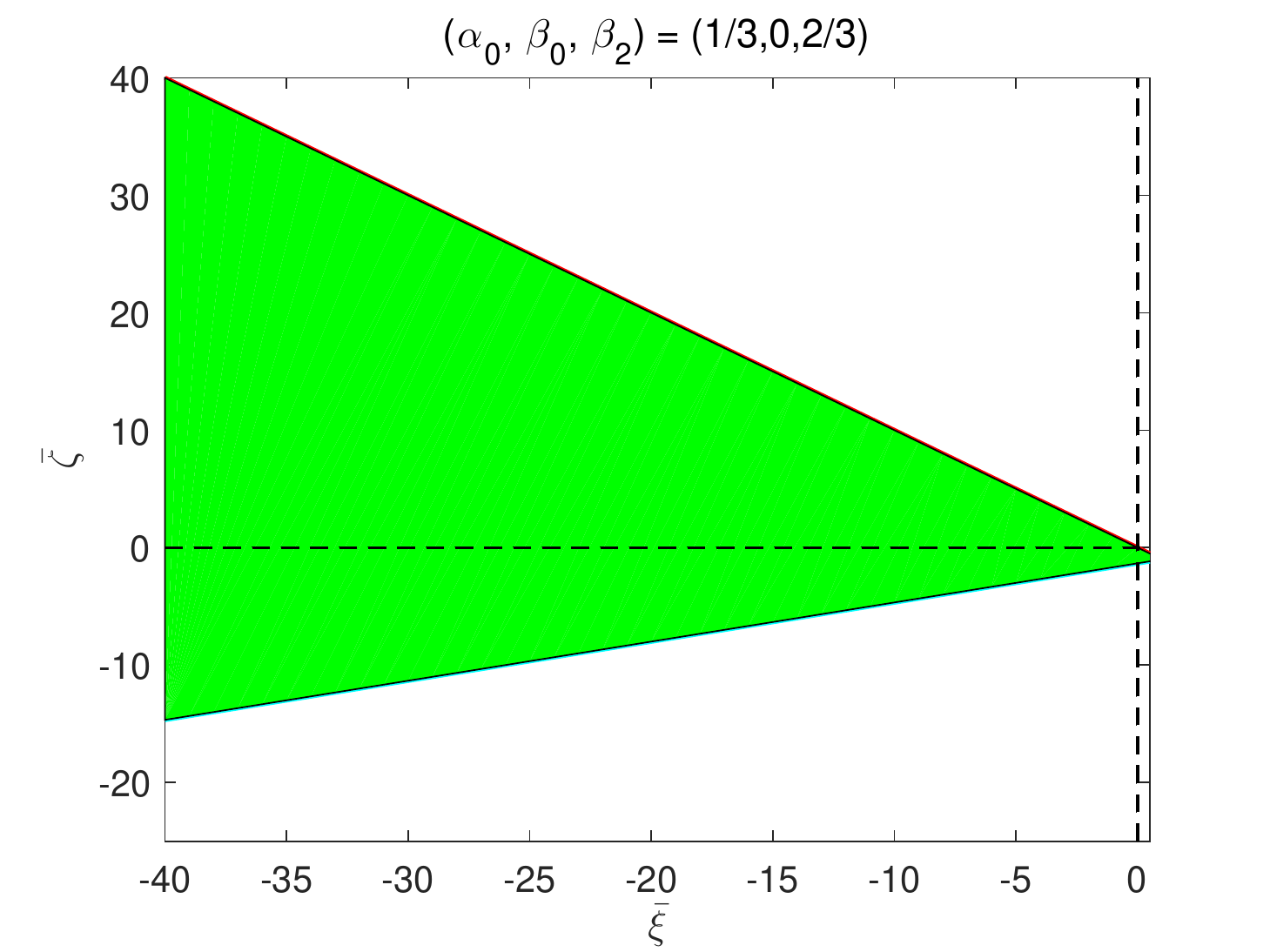}}
\end{minipage}
\\
\begin{minipage}{0.48\linewidth}
  \centerline{\includegraphics[width=7cm,height=5cm]{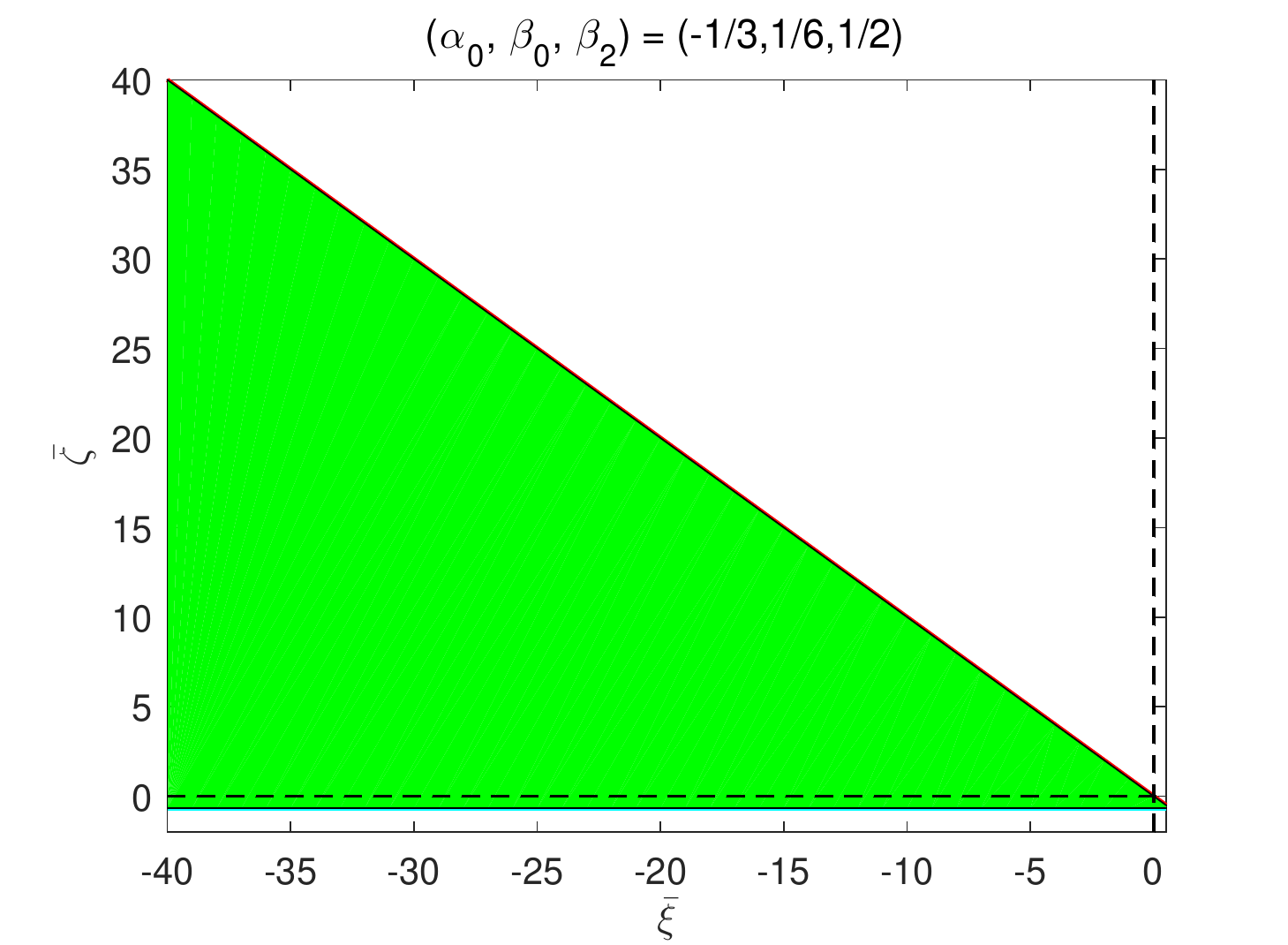}}
\end{minipage}
\hfill
\begin{minipage}{0.48\linewidth}
  \centerline{\includegraphics[width=7cm,height=5cm]{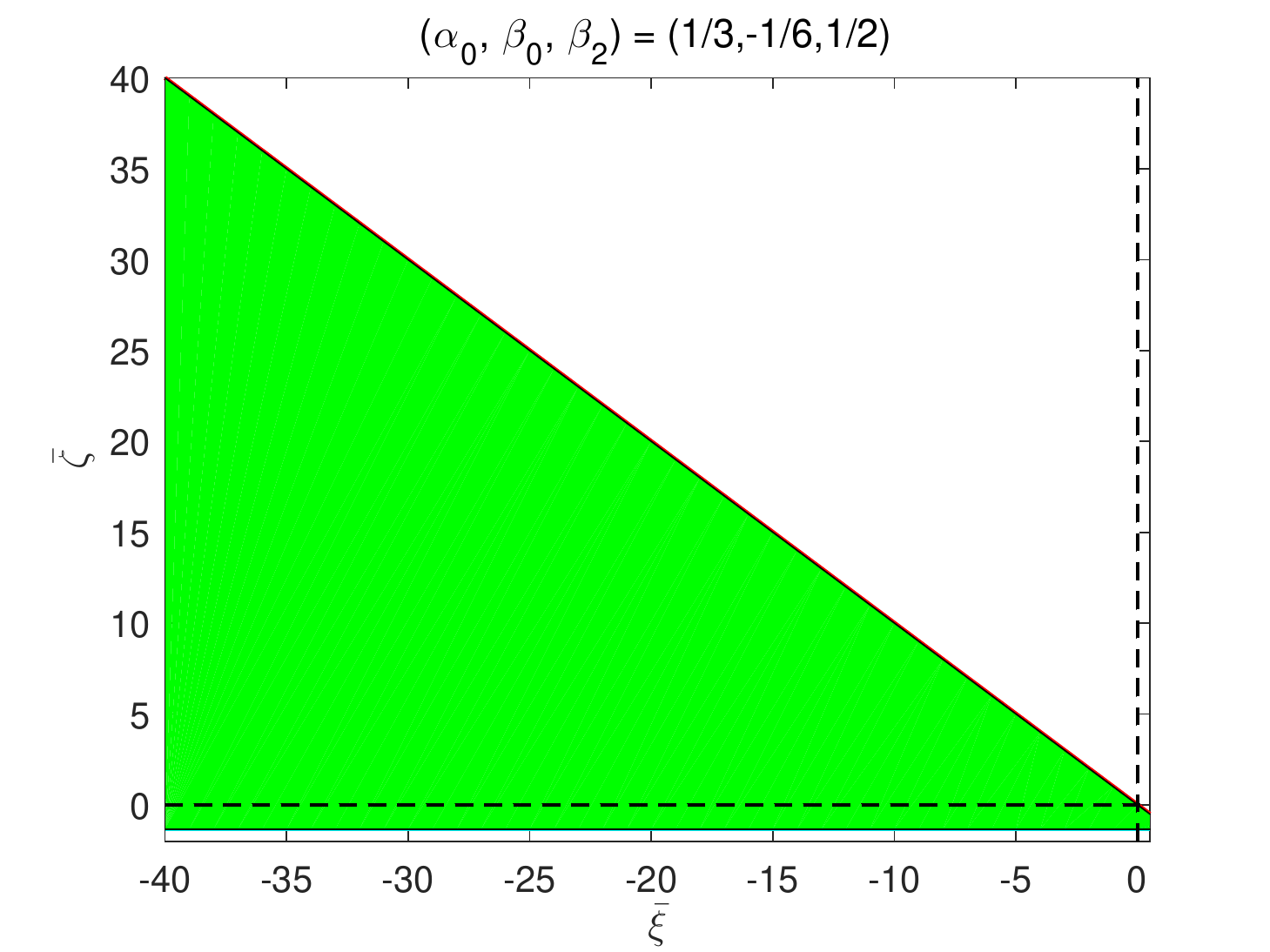}}
\end{minipage}
\caption{Stability regions of the scheme \eqref{8.3.2} with $(\alpha_0,\beta_0,\beta_2) = (-1/3,5/12,3/4)$, $(1/3,0,1)$, $(-1/3,1/6,1/2)$ and $(1/3,-1/6,1/2)$. }
\label{figC.2}
\end{figure}

\end{appendix}

\end{document}